\documentclass{article}
\usepackage{amssymb,amsthm,amsmath,mathrsfs, mathtools}
\usepackage[dvipsnames]{xcolor}
\usepackage{bbm} 
\usepackage{geometry}     
\usepackage{hyperref} 
\hypersetup{colorlinks=true,allcolors=blue}
\usepackage{hypcap}
\usepackage{comment}

\usepackage{imakeidx}
\makeindex[title=Index of terminology, columns=2] 
\makeindex[name=symbols, title=Index of symbols, columns=3] 

\usepackage[utf8]{inputenc}
\usepackage{indentfirst}

\newtheorem{theorem}{Theorem}[section]
\newtheorem{lem}[theorem]{Lemma}
\newtheorem{prop}[theorem]{Proposition}

\newtheorem{defin}[theorem]{Definition}
\newtheorem{fait}[theorem]{Fact}

\theoremstyle{remark}
\newtheorem{rem}[theorem]{Remark}

\renewcommand{\rm}[1]{\mathrm{#1}}
\renewcommand{\cal}[1]{\mathcal{#1}}
\newcommand{\bb}[1]{\mathbb{#1}}

\renewcommand{\frak}[1]{\mathfrak{#1}}
\newcommand{\dd}{\mathrm{d}}

\newcommand{\R}{\bb R}
\newcommand{\Q}{\bb Q}
\newcommand{\Z}{\bb Z}
\newcommand{\C}{\bb C}

\newcommand{\N}{\bb N}

\renewcommand{\P}{\bb P}

\newcommand{\calM}{\cal M}
\newcommand{\calF}{\cal F}
\newcommand{\calG}{\cal G}
\newcommand{\calO}{\cal O}
\newcommand{\calW}{\cal W}

\newcommand{\calL}{\cal L}

\newcommand{\sld}{\rm{SL}_d(\Z)}
\newcommand{\slr}{\rm{SL}_d(\R)}
\newcommand{\pslr}{\rm{PSL}_d(\R)}

\newcommand{\inj}{{inj} }
\newcommand{\supp}{\rm{supp}}
\DeclareMathOperator{\vol}{vol}
\DeclareMathOperator{\leb}{Leb_{\mathfrak{a}}}

\renewcommand{\c}{c}

\newcommand{\hypG}{Let $G$ be a connected, real linear, semisimple Lie group of non-compact type. }

\title{Equidistribution and counting of periodic flat tori }
\author{Nguyen-Thi Dang and Jialun Li}
\date{}

\begin{document}

\maketitle
\begin{abstract}
Let $G$ be a semisimple Lie group without compact factor and $\Gamma < G$ a torsion-free, cocompact, irreducible lattice.
According to Selberg, periodic orbits of regular Weyl chamber flows live on maximal flat periodic tori of the space of Weyl chambers. 
We prove that these flat periodic tori equidistribute exponentially fast towards the quotient of the Haar measure. From the equidistribution formula, we deduce a higher rank prime geodesic theorem.
These counting and equidistribution results also hold in the non cocompact, finite covolume case for $G=\mathrm{SL}(d,\mathbb{R})$ and $\Gamma<\mathrm{SL}(d,\mathbb{Z})$ a finite index subgroup.
\end{abstract}

\section{Introduction}

Let $G$ be a semisimple, connected, real linear Lie group without compact factor. 
Let $K$ be a maximal compact subgroup, $A$ be a maximal $\mathbb{R}$-split torus, $A^{+}\subset A$ a closed positive chamber such that the Cartan decomposition $G=KA^+K$ holds.
Denote by $M:=Z_K(A)$ the centralizer of $A$ in $K$, by $\frak a:= \mathrm{Lie} \; A$ the Cartan subspace, by $ \frak{a}^+$ the closed positive chamber in the Lie algebra and by $\frak{a}^{++}$ its interior.  

Let $\Gamma < G$ be a torsion-free, cocompact lattice. 
The double coset space $\Gamma \backslash G/M$ is called the \emph{space of Weyl chambers} of the symmetric space $\Gamma \backslash G/K$.
We count and study the equidistribution of the compact right $A$-orbits in the space of Weyl chambers.

\subsection{Pioneering works on hyperbolic surfaces}
In this case, $G= \mathrm{PSL}(2,\mathbb{R})$ is the isometry group of the Poincaré half-plane $\mathbb{H}^2$, the space of Weyl chamber is the unit tangent bundle of the hyperbolic surface $\Gamma \backslash \mathbb{H}^2$ and the right action of $A$ on $\Gamma \backslash G/M$ corresponds to the geodesic flow.
Periodic orbits of the geodesic flow projects in the surface to primitive closed geodesics.

\paragraph{Prime geodesic theorems} 
In 1959, Huber \cite{huberZurAnalytischenTheorie1959} proved a prime geodesic theorem for compact hyperbolic surfaces.
He obtained an estimate of the number of primitive closed geodesics as their length grows to infinity. 
More precisely, let $N(T)$ be the number of primitive closed geodesics of length less than $T$ on a hyperbolic surface. 
He proved that as $T$ tends to infinity, 
$$N(T) \sim e^T/T.$$
This term is similar to the asymptotic $x/\log x$ given by the prime number 
theorem\footnote{See Pollicott's research statement §1.2 \cite{pollicott}}
for the number of primes less than $x$. 
In 1969, using dynamical methods, Margulis \cite{margulisCertainApplicationsErgodic1969} extended the prime geodesic theorem to negatively curved compact manifolds. 
He proved that the exponential growth rate of $N(T)$ is equal to the topological entropy of the geodesic flow. 
Later on, relying on Selberg's Trace formula, Hejhal 
\cite{hejhal_selberg_1976} and Randol \cite{randol_asymptotic_1977} obtained a complete asymptotic development of the counting function in terms of the spectrum of the Laplace-Beltrami operator. 
In 1980, Sarnak \cite{sarnak1980} extended their complete asymptotic development to finite area surfaces.

Let us state one of the various equivalent formulations of the prime geodesic theorem.
For a closed geodesic $\c$ on $\Gamma\backslash \mathbb{H}^2$, denote by $\ell(\c)$ the length of this geodesic. Let $c_0$ be the primitive closed geodesic underlying $\c$.
Then as $T\rightarrow +\infty$
\begin{equation}\label{equ-prime}
\sum_{\c_0}\left\lfloor \frac{T}{\ell(\c_0)} \right\rfloor\ell(c_0)=\sum_{\c,\ell(c)\leq T}\ell(c_0)\sim C_G \vol(D_t),
\end{equation}
where the first sum is over all primitive closed geodesics, the second sum is over all closed geodesics, $C_G$ is some positive constant, $\vol$ is the Haar measure on $G$ induced by the Riemannian metric on $\mathbb{H}^2$ and $D_t$ is the preimage in $G$ of a ball in $\mathbb{H}^2$ of radius $T$ centered in $i$. 
This sum is similar to the second Chebyshev function: the weighted sum of the logarithms of primes less than a given number, where the weight is the highest power of the prime that does not exceed the given number. 
The second Chebyshev function is essentially equivalent to the prime counting function and their asymptotic behavior is the same. 

\paragraph{Equidistribution of closed geodesics}
Margulis in his 1970 thesis\footnote{See Parry's review \cite{parry}} and Bowen \cite{bowenPeriodicOrbitsHyperbolic1972}, \cite{bowen_equidistribution_1972} independently studied the spatial distribution of the closed orbits of the geodesic flow. They proved that they uniformly equidistribute towards a measure of maximal entropy as their period tends to infinity. 
In the second 1972 paper, Bowen proved the uniqueness of the measure of maximal entropy for the geodesic flow.
As a consequence, the measure of maximal entropy of the geodesic flow is equal to the quotient of the Haar measure.
Later, Zelditch \cite{zelditchSelbergTraceFormulae1992} generalized Bowen's equidistribution theorem to finite area hyperbolic surfaces. 

Let us recall Bowen and Margulis' result for a compact hyperbolic surface. 
For every (primitive) periodic orbit $\c \subset  \Gamma \backslash \mathrm{PSL}(2,\mathbb{R})$, denote by $\cal{P}_{\c}$ the unique probability measure invariant under the geodesic flow supported on $c$.
For every $T >0$, we denote by $\calG_p(T)$ the set of (primitive) periodic orbits of minimal period less that $T$.
Bowen and Margulis proved that for every bounded smooth function $f$,
$$ \frac{T}{e^{ T}} \sum_{\c \in \calG_p(T) } \int f \; \dd\cal{P}_{\c} \xrightarrow[T \rightarrow \infty]{} \int f \; \dd m_{\Gamma},$$
where $m_{\Gamma}$ is the measure of maximal entropy, which also corresponds in our case to the quotient measure of the Haar measure on $ \Gamma \backslash \mathrm{PSL}(2,\R)$.

The following non exhaustive list \cite{degeorge_length_1977}, \cite{gangolli_zeta_1980}, \cite{parryAnaloguePrimeNumber1983}, 
\cite{roblin}, \cite{naud2005expanding}, \cite{margulis_closed_2014} provides some of the many subsequent works tackling the counting and equidistribution problem in several different rank one generalisations.

\subsection{Main results}

In this article, we focus on the higher rank case\footnote{more precisely, we do not have restrictions on the rank of $G$} for $G$,
meaning that $\dim_{\mathbb{R}} A \geq 2$.  

\begin{defin}[Maximal flat periodic tori]\label{defin-flat-per-tori}
For any right $A$-orbit $F$ in $\Gamma\backslash G/M$, we define the set of \emph{periods} of $F$ as
$$ \Lambda (F):= \lbrace Y \in \frak{a} \; \vert \; ze^Y=z ,\;   \forall z \in F \rbrace. $$
A period $Y$ in $\Lambda(F)$ is called \emph{regular} if $Y\in \frak a^{++}$.
When $\Lambda(F)$ is a lattice of $\frak{a}$, we say $F$ a \emph{maximal flat periodic torus} or a \emph{compact periodic $A$-orbit}. 
\end{defin}
Denote by $C(A)$ the set of maximal flat periodic tori in $\Gamma\backslash G/M$. 
For every $F \in C(A)$, we denote by $L_F$ the quotient measure on $F$ of $\leb$, the Lebesgue measure on $\frak{a}$. 
Note that $L_F$ is not a probability measure. Its total mass, denoted by $\vol_{\frak{a}} (F)$, is the Lebesgue measure of any fundamental domain in $\mathfrak{a}$ of the lattice $\Lambda (F)$.


\paragraph{Main counting result}

Adopting the same notation as above, $\vol$ denotes the Haar measure on $G$ whose quotient on the symmetric space $X:=G/K$ equals the measure induced by the Riemannian metric. 
Denote by $\Vert \; \Vert$ the Euclidean norm on $\frak{a}$ coming from the Killing form on $\mathfrak{g}$ and by $B_\frak{a}$ the balls for this norm.
For every $T>0$, set $B_{\frak a}^{++}(0,T):=B_{\frak a}(0,T)\cap\frak a^{++}$ and
$D_T:= K \exp\big( B_{\frak{a}}(0,T) \big) K$, which is the preimage by the quotient map $G\rightarrow X$ of the ball of radius $T$ centered at $eK$ in the symmetric space $X$.

\begin{theorem}\label{corol-counting}
Let $G$ be a semisimple, connected, real linear Lie group without compact factor and $\Gamma < G$ be a torsion-free, cocompact irreducible lattice or $G=\mathrm{SL}(d,\mathbb{R})$ with $d \geq 2$ and $\Gamma < \mathrm{SL}(d,\mathbb{Z})$ a finite index subgroup. 
Then there exist constants $C_G>0$ and $u>0$ such that for $T>0$
\begin{equation}
	\sum_{F\in C(A)}|\Lambda(F)\cap B_{\frak a}^{++}(0,T)| \;
	\vol_{\frak{a}}(F)
	=\vol(D_T)(C_G+O(e^{-u T})).
\end{equation}
\end{theorem}
We deduce this counting result from the subsequent equidistribution statement.
\paragraph{Main equidistribution result}
Denote by
$\pi : G \rightarrow \Gamma \backslash G/M$ the projection and by $\widetilde m_{\Gamma}$ the quotient measure of the Haar measure $vol$.
We normalise $\widetilde m_{\Gamma}$ to obtain a probability measure that we denote by $m_{\Gamma}$. 
We obtain a higher rank version of the Bowen-Margulis equidistribution formula with an exponential rate of convergence. 

\begin{theorem}\label{thm-introequid}
Under the same hypothesis and for the same constants $C_G>0$ and $u>0$ as in the previous Theorem \ref{corol-counting}, for all $T>0$ and every Lipschitz function $f$ on $\Gamma\backslash G/M$ we have
\begin{equation}
	\frac{1}{\vol(D_T)}\sum_{F\in C(A)}|\Lambda(F)\cap B_{\frak a}^{++}(0,T)|\int_F f \; \dd L_F= C_G \int_{\Gamma\backslash G/M} f \; \dd m_{\Gamma}+O(e^{-u T}|f|_{Lip}).
\end{equation}
\end{theorem}
The asymptotic behavior of the main term as $T$ tends to infinity is
$\vol(D_T)\sim C_0 T^{\frac{ \dim A -1}{2} } e^{\delta_0 T} $, where $\delta_0>0$ is determined by the root system of $\mathfrak{g}$, the Lie algebra of $G$ and $C_0>0$ is given by the Borel--Harish-Chandra formula.
\begin{figure}[h!]
    \centering
    \includegraphics[width=10cm]{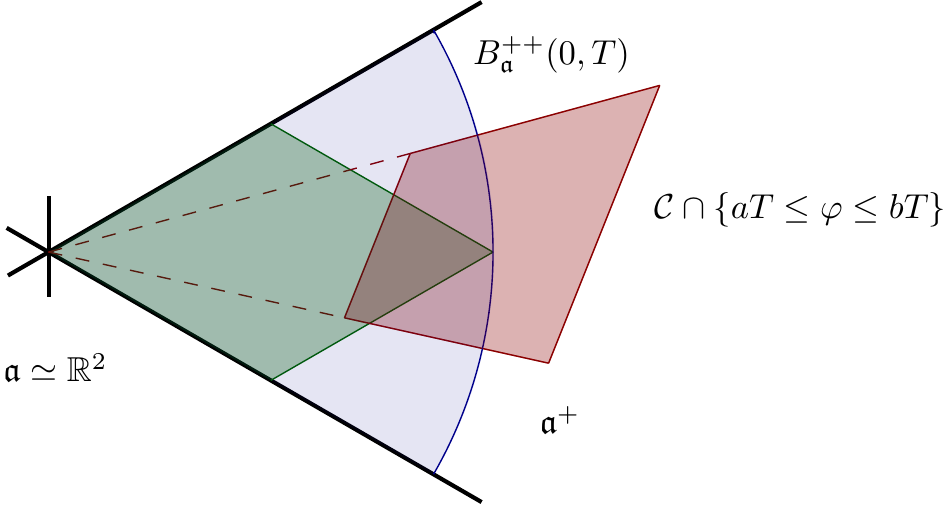}
    \caption{This is a positive Weyl chamber for $\mathrm{SL}(3,\mathbb{R})$ and $T>0$ is large. 
    In blue, our counting region $B_{\frak a}^{++}(0,T)$.
    In green, Deitmar's \cite{deitmar} counting region.
    In red, Guedes Bonthonneau--Guillarmou--Weich's \cite{bonthonneau_srb_2021} counting region where $\mathcal{C}$ is a convex cone strictly inside $\mathfrak{a}^{++}$ delimitted by the red dashed lines, $0<a<b$ are real numbers and $\varphi$ is a linear form strictly positive in $\mathfrak{a}^+$.
    }
    \label{fig:domains}
\end{figure}

\begin{rem}
Note that in the rank one case, any flat periodic torus $F$ corresponds to a primitive closed geodesic.
Furthermore, both $\vol_{\frak a}(F)$ and its smallest regular period correspond to the length of the geodesic. 
Therefore Theorem \ref{corol-counting} is a higher rank version of the prime geodesic theorem \eqref{equ-prime}. 

1. In the compact case, Spatzier in his thesis \cite{spatzier83} computed using the root spaces of the Lie algebra of $G$, the topological entropy of every regular Weyl chamber flows: right action of $\exp(\R Y)$ on $\Gamma \backslash G/M$, where $Y\in \mathfrak{a}^{++}$ is non zero. 
Furthermore, $\delta_0$, the exponential growth rate of $\vol(D_T)$, is a sharp upper bound of the topological entropy of regular Weyl chamber flow.
He also proved that $\delta_0$ is equal to the exponential growth rate of 
the sum over maximal flat periodic tori of smallest regular period less than $t$ of $\vol_{\frak a}(F)$, as $t$ goes to infinity. 
Knieper \cite{knieper2005uniqueness} studied the equidistribution of periodic orbits of regular Weyl chamber flows in the same setting. 
He obtained an equidistribution formula towards the measure of maximal entropy of the most chaotic regular Weyl chamber flow, the one whose topological entropy is $\delta_0$. 

1.1 In the finite volume case, Oh \cite{oh} proved that the number of maximal flat periodic tori of bounded volume is always finite. 

2. In the compact case, Deitmar \cite{deitmar} used a Selberg trace formula and methods from analytical number theory to prove a similar version of Theorem \ref{corol-counting} with a different summation region in the Weyl chamber, the one in green in Figure \ref{fig:domains}. 
He later on generalized this counting result to the non compact finite volume case $\mathrm{SL}(3,\Z) \backslash \mathrm{SL}(3,\R)$, in a joint work with Gon and Spilioti in \cite{deitmar_prime_2019}. 

3. Recently and for the compact case, Guillarmou--Guedes Bonthonneau--Weich \cite[Theorem 2, equation (0.3)]{bonthonneau_srb_2021} obtained an equidistribution formula. 
The region where they count the multiplicity of periodic tori is defined using any convex non degenerate closed cone $\mathcal{C}$ stricly inside $\mathfrak{a}^{++}$, any choice of positive numbers $0<a<b$ and any linear form $\varphi$ that takes positive values in $\mathfrak{a}^+$ as shown in red in Figure \ref{fig:domains}.
They take a different approach, relying on the spectral properties of the $A$-action via their previous study of Ruelle-Taylor resonances with Hilgert \cite{bonthonneau_taylor_2020}.

4. Because our counting region is different (shown in blue in Figure \ref{fig:domains}), our first asymptotic term is new in higher rank. 
It would be interesting to see whether we could generalize our methods to the regions in red and green in order to recover the results of \cite{deitmar}, \cite{bonthonneau_srb_2021}.
None of the above works provides estimates on the speed of convergence.

5. For the non-compact, finite volume case $\mathrm{SL}(3,\mathbb{Z}) \backslash \mathrm{SL}(3,\R)$, Einsiedler--Lindenstrauss--Michel--Venkatesh in \cite{einsiedler_distribution_2011} use the classification of diagonal invariant measures and subconvexity estimates to deduce an equidistribution result for the following collection of tori. 
They take sets of flat periodic tori of the same discriminant and prove that the sum of Lebesgue measures on those tori, normalized by the total mass, equidistributes towards the quotient measure of the Haar measure as the discriminant goes to infinity.

6. 
Using a dictionary between closed geodesic orbits of hyperbolic surfaces and objects coming from number theory, \cite{sarnak} could deduce counting results on class numbers of totally real orders. Later, \cite{deitmar_prime_2019} did the same for $\mathrm{SL}(3,\mathbb{Z})$.
It would be interesting to use a dictionary between compact periodic $A$-orbits and number theory to deduce a number theoretic version of Theorem \ref{corol-counting}.

7. Since the multiplicity term $ \vert \Lambda(F) \cap B_{\frak{a}}^{++}(0,T) \vert$ is at most polynomial and the asymptotic behavior of $\vol(D_T)$ is dominated by the exponentiel term, we expect that the equidistribution formula holds if we replace this multiplicity term by $1$ and sum over the flat periodic tori whose smallest regular period is of length less than $T$. 
\end{rem}

\subsection{On the proof of the main theorem}

The proof of the equidistribution result in the cocompact case follows Roblin's proof \cite{roblin} closely, where he proved counting and equidistribution results for some infinite covolume hyperbolic manifolds. We replace all the ingredients from hyperbolic geometry with their higher rank counterparts, such as Hopf coordinates, Patterson-Sullivan measures, the angular distribution of lattice points. One significant difference in higher rank cases is that we need to carefully treat the boundary of the Weyl chamber, while in the hyperbolic case, it is just a point.

For the non-cocompact case of finite index subgroups $\Gamma <\mathrm{SL}(d,\Z)$, we first prove the equidistribution (Theorem \ref{thm-introequid}) on compact sets of $\Gamma \backslash G/M$.
Then we prove the non-escape of mass for compact periodic $A$-orbits. 
The critical observation is that there exist two large compact sets $\Omega_T,\Omega_T'\subset \Gamma\backslash G/M$ depending on the parameter $T$ such that for any compact periodic $A$-orbit $F$ with a regular period of length less than $T$, the measure of $F$ outside the compact set, $F\cap \Omega_T^c$, is bounded by its measure inside the compact set, $F\cap (\Omega_T \setminus \Omega_T')$. 
Equidistribution is known for functions supported on $\Omega_T$. 
Therefore we bound the mass outside the compact set $\Omega_T$ of the measure in Theorem \ref{thm-introequid} (an average of measures on compact periodic $A$-orbits with a regular period of length less than $T$) by the Haar measure of $\Omega_T \setminus \Omega_T'\subset (\Omega_T')^c$, which decays exponentially fast as $T$ goes to infinity due to the choice of $\Omega_T'$.

\subsection*{Organization of the paper}

In Section 2, we gather the basic facts and preliminaries about semisimple real Lie groups, the Furstenberg boundary, Hopf coordinates, higher rank Patterson-Sullivan measure, volume estimates and the angular distribution of lattice points.

In Section 3, we prove a lemma about comparing the angular part of an element in $G$ with its contracting and repelling fixed points in the Furstenberg boundary. In Section 4, we relate loxodromic elements and periodic tori.

In Section 5 and 6, we prove Theorem \ref{thm-introequid} for cocompact lattices and for $\Gamma<\sld$ acting on $G/M$ freely, respectively.

In Appendix A, we introduce the language of orbifolds to treat the general case of finite index subgroups of $\sld$.

In Appendix B, we follow the works of Gorodnik-Nevo \cite{gorodnikCountingLatticePoints2012} \cite{gorodnikLiftingRestrictingSifting2012} and explain why their results work in our setting.

\paragraph{Notation.} In the paper, given two real functions $f$ and $g$, we write $f \ll g$ or $f=O(g)$ if there exists a
constant $C > 0$ only depending on $G, \Gamma$ such that $f \leq Cg$. We write $f \asymp g$ if $f \ll g$ and $g\ll f$ .

\section*{Acknowledgement}
We would like to thank Alex Gorodnik for explaining his results with Amos Nevo, Mark Pollicott for telling the first author to look at the counting problem in Ralf Spatzier's thesis and Viet Dang because his inspirational habilitation defense sparked the collaboration.
Part of this work was done while two authors were at the \textit{Hyperbolic dynamical systems and resonances} conference in Porquerolles and the \textit{Anosov}$^3$ mini-workshop in Oberwolfach; we would like to thank the organizers Colin Guillarmou, Benjamin Delarue (formerly K\"{u}ster), Maria Beatrice Pozzetti, Tobias Weich, and the hospitality of the centres. We would also like to thank the hospitality of Institut für Mathematik
Universität Zürich and Fakultät für Mathematik und Informatik Universität Heidelberg for each time the authors visit each other.
The first author acknowledges funding by the Deutsche Forschungsgemeinschaft (DFG, German Research Foundation) – 281869850 (RTG 2229). The second author acknowledges the funding by Alex Gorodnik's SNF grant 200021--182089.

\section{Background}

\begin{center}
\fbox{
\begin{minipage}{.7 \textwidth}
In the whole article, $G$ is a semisimple, connected, real linear Lie group, without compact factor.
\end{minipage}
}
\end{center}

Classical references for this section are \cite{thirion}, \cite{guivarc2012compactifications}, \cite{helgason1978differential}. One also may refer to the exposition in \cite{dg20}.

Let $K$ be a maximal compact subgroup of $G$. Then $X= G/K$ is a globally symmetric space of non-compact type and $G= \mathrm{Isom}_0(X)$. 
We fix a base point $o \in X$ such that $K= \mathrm{Stab}_G(o)$. 
For every $x\in X$, we denote by $K_x:= \mathrm{Stab}_G(x)$. Note that for any $h_x \in G$ such that $h_x o=x$, then $K_x=h_xK h_x^{-1}$, independently of the choice of $h_x$.

\paragraph{Geometric Weyl chambers}
Denote by $\frak{g}$ (resp. $\frak{k}$) the Lie algebra of $G$ (resp. $K$) and consider the Cartan decomposition in the Lie algebra $\frak{g}= \frak{k} \oplus \frak{p}$.
Let $\frak{a} \subset \frak{p}$ be a \emph{Cartan subspace} of $\frak{g}$.
Then $A:= \exp(\frak{a})$ is a maximal $\mathbb{R}$-split torus of $G$.
Denote by $M:= Z_K(A)$ the centralizer of $A$ in $K$.
The \emph{real rank} of $G$, denoted by $r_G$, is equal to $\dim_{\mathbb{R}} \frak{a}$.
We say that $G$ is \emph{higher rank} when $r_G \geq 2$.

For any linear form $\alpha$ on $\frak{a}$, set $\frak{g}_\alpha := \lbrace v \in \frak{g} \; \vert \; \forall u \in \frak{a}, \; [u,v]=\alpha(u) v \rbrace.$
The \emph{set of restricted roots} is denoted by $\Sigma:= \lbrace \alpha \in \frak{a}^*\setminus \lbrace 0\rbrace \; \vert \; \frak{g}_{\alpha} \neq \lbrace 0\rbrace \rbrace.$
The kernel of each restricted root is a hyperplane of $\frak{a}$. The \emph{Weyl chambers} of $\frak{a}$ are the connected components of $\frak{a} \setminus \cup_{\alpha\in \Sigma} \ker(\alpha)$. 
We choose a \emph{positive Weyl chamber} by fixing such a connected component and denote it (resp. its closure) by $\frak{a}^{++}$ (resp. $\frak{a}^+$).
In the Lie group, we denote by $A^{++}:=\exp(\frak{a}^{++})$ (resp. $A^+:= \exp(\frak{a}^+)$). Denote by $N_K(A)$ the normalizer of $A$ in $K$. The group $N_K(A)/M$ is the \emph{Weyl group}, denoted by $\calW$. The Weyl group also acts on the Lie algebra $\frak a$ by the adjoint action, which acts transitively on the set of connected components of $\frak{a} \setminus \cup_{\alpha\in \Sigma} \ker(\alpha)$.

A \emph{geometric Weyl chamber} is a subset of $X$ of the form $g. (A^+ o)$, where $g \in G$.
The \emph{base point} of the geometric Weyl chamber $g A^+o$ is the point $go \in X$.
In \cite[§2]{dg20}, we obtained the following identifications between the space of Weyl chambers and the set of geometric Weyl chambers of $X$,
\begin{equation}\label{eq-geomWC}
 G/M \simeq G. (A^+ o).
\end{equation}

\paragraph{Cartan projection}
\begin{defin}\label{defin-cartan-proj}
For any $g \in G$, we define, by Cartan decomposition, a unique element $\underline{a}(g) \in \frak{a}^+$ such that $g \in K \exp(\underline{a}(g)) K$.
The map $\underline{a}:G \rightarrow \frak{a}^+$ is called the \emph{Cartan projection}.
\end{defin}
The Cartan projection allows to define an $\frak{a}^+$-valued function on $X \times X$, denoted by $d_{\underline{a}}$. 
For every $x,y\in X$, any choice $h_x,h_y \in G$ such that $h_x o =x$ and $h_yo=y$, we set
$$d_{\underline{a}} (x,y) := \underline{a}(h_x^{-1} h_y).$$
This function does not depend on the choice of $h_x$ and $h_y$ up to right multiplication by $K$.
By \cite[Chapter V, Lemma 5.4]{helgason1978differential}, we endow $\frak{a}$ with a scalar product coming from the Killing form on $\frak{g}$.
We denote by $\Vert . \Vert$ the associated norm on $\frak{a}$ and define the $G$-invariant riemannian distance on $X$
$$ \dd_X(x,y):= \Vert d_{\underline{a}}(x,y) \Vert .$$
The following fact is standard for symmetric spaces of non-compact type.
\begin{fait}\label{fait-a-dist}
For every $x,y\in X$, there is a geometric Weyl chamber based on $x$ containing $y$.
If furthermore, $d_{\underline{a}}(x,y)\in \frak{a}^{++}$, such a geometric Weyl chamber is defined by a unique element $h_{xy}M \in G/M$ such that $h_{xy}o=x$ and $h_{xy}  e^{ d_{\underline{a}}(x,y) } o = y$.
\end{fait}
\begin{proof}
Fix $x,y \in X$ and choose $h_x,h_y \in G$ such that $h_x o =x$ and $h_y o =y$.
By Cartan decomposition of $h_x^{-1}h_y$ and Definition \ref{defin-cartan-proj}, there exists $k,l \in K$ such that 
$$ h_x^{-1}h_y = k e^{d_{\underline{a}}(x,y) }l^{-1} .$$
Set $h_{xy}:= h_x k$. Since $K$ fixes $o$, we deduce that $h_{xy}o=x$ and $h_{xy}  e^{ d_{\underline{a}}(x,y) } o = y$.

For all $k',l' \in K$, the elements $h_x k'$ and $h_y l'$ also respectively send $o$ to $x$ and $y$.
Note that we have another Cartan decomposition of $(h_x k')^{-1}h_y l'$ given by $(k')^{-1}  k \;  e^{d_{\underline{a}}(x,y) } \; l^{-1}l' $. 
Applying the same construction, we still recognize that $h_x k' (k')^{-1} k=h_{xy}$.
Hence $h_{xy}$ does not depend on the choice of representatives $h_x$ and $h_y$, and it depends on the choice of $k,l \in K$ in the Cartan decomposition. 

It remains to show that $h_{xy}$ is unique up to right multiplication by elements of $M$ when $d_{\underline{a}}(x,y)\in \frak{a}^{++}$.
In this case, the elements $k,l \in K$ given by Cartan decomposition are defined up to right multiplication by elements in $M$. 
Hence the fact.
\end{proof}

\paragraph{Jordan projection}
Denote by $\Sigma^+$ the subset of roots which take positive values in the positive Weyl chamber. 
It allows to define the following nilpotent subalgebras $\frak{n}:= \oplus_{\alpha \in \Sigma^+} \frak{g}_{\alpha}$ and $\frak{n}^-=\oplus_{\alpha \in \Sigma^+} \frak{g}_{-\alpha}$.
Denote by $N:= \exp (\frak{n})$ and $N^-:=\exp(\frak{n}^-)$ two maximal unipotent subgroups of $G$.

By Jordan decomposition, every element $g \in G$ admits a unique decomposition $g=g_e g_h g_u$ where $g_e, g_h$ and $g_u$ commute and such that $g_e$ (resp. $g_h, g_u$) is conjugated to an element in $K$ (resp. $A^+$, $N$). 
The element $g_e$ (resp. $g_h$, $g_u$) is called the \emph{elliptic part} (resp. \emph{hyperbolic part}, \emph{unipotent part}) of $g$.

\begin{defin}\label{def-Jordan-proj}
For any element $g \in G$, there is a unique element $\lambda(g) \in \frak{a}^+$ such that the hyperbolic part $g_h$ is conjugated to $\exp (\lambda(g)) \in A^+$.
The map $\lambda : G \rightarrow \frak{a}^+$ is called the \emph{Jordan projection}.

\noindent Any element $g \in G$ such that $\lambda(g) \in \frak{a}^{++}$ is called \emph{loxodromic.}

\noindent Denote by $G^{lox}$ the set of loxodromic elements of $G$ and for any subset $S\subset G$, denote by $S^{lox}:= S \cap G^{lox}$.
\end{defin}
Equivalently (Cf. §4 \cite{dang20}), loxodromic elements are conjugated in $G$ to elements in $A^{++}M$.

\paragraph{Asymptotic Weyl chambers}
Denote by $P:=MAN$ and by $\mathcal{F}:=G/P$ the \emph{Furstenberg boundary}.
We recall the interpretation of $\mathcal{F}$ in terms of asymptotic Weyl chambers.

Following the exposition in \cite{dg20}, we introduce the following equivalence relation between geometric Weyl chambers:
$$ g_1 A^+ o \sim g_2 A^+ o \Longleftrightarrow \sup_{a \in A^{++}} \dd_X( g_1 a o, g_2 a o ) < + \infty.$$
Equivalence classes for this relation are called \emph{asymptotic Weyl chambers}.
Denote by $\eta_0$ (resp. $\zeta_0$) the asymptotic Weyl chamber of $A^+o$ (resp. $(A^+)^{-1} o$).
The set of asymptotic Weyl chambers identifies with the Furstenberg boundary (see for instance \cite[Fact 2.5]{dg20} for a proof),  
\begin{equation} \label{eq-asymptotic-furstenberg}
 \mathcal{F} \simeq \big( G. (A^+o)/\sim  \big) \simeq K/M \simeq K. \eta_0 .
\end{equation}
Remark that $\zeta_0= k_\iota \eta_0$ where $k_\iota \in N_K(A)$ satisfies $k_\iota A^+ k_\iota^{-1} = (A^+)^{-1}$.
Furthermore, $\mathrm{Stab}_G(\eta_0)=P$ and $\mathrm{Stab}_G(\zeta_0) = MAN^-$.

In the remainder of the article, we identify $G. (A^+ o ) / \sim$ with $\mathcal{F}$ and $G . (A^+ o)$ with $G/M$.
We prove that a geometric Weyl chamber is uniquely determined by its base point in $X$ and the asymptotic Weyl chamber it represents.
\begin{fait}\label{fait-asymptotic-point}
The following $G$-equivariant map is a diffeomorphism:
\begin{align*}
    G/M & \overset{\sim}{\longrightarrow} X \times \mathcal{F} \\
    gM & \longmapsto (go, g \eta_0).
\end{align*}
\end{fait}
\noindent For every $(x,\xi) \in X \times \mathcal{F}$, we denote by $g_{x,\xi} M \in G/M$ the geometric Weyl chamber of base point $x$ asymptotic to $\xi$.
\begin{proof}
Note that $M \in G/M$ corresponds to the geometric Weyl chamber $A^+o$, of base point $o$ and asymptotic Weyl chamber $\eta_0$. 
Since $\mathrm{Stab}_G(o)=K$ and $\mathrm{Stab}_K(\eta_0)=M$, we deduce that the map is injective.

Let us prove that the map is surjective.
Fix $x \in X$ and $\xi \in \mathcal{F}$. Choose a representative $h_x \in G$ such that $h_x o=x$. 
By the identification $\mathcal{F} \simeq K/M \simeq K. \eta_0 $ in \eqref{eq-asymptotic-furstenberg}
there exists $k_{x,\xi}M \in K/M$, such that
$ k_{x,\xi} \eta_0= h_x^{-1} \xi.$
Hence, by $G$-equivariance and since $k_{x,\xi} o=o$ we deduce that $h_x k_{x,\xi}M \in G/M$ maps to $(x,\xi)$.

By $G$-equivariance, we only need to prove that the map is a diffeomorphism at $M$.
By Bruhat decomposition (Cf. \cite[Chapter IX, Cor. 1.8, Cor. 1.9 ]{helgason1978differential}), $ g = \frak{n} \oplus \frak{m} \oplus \frak{a} \oplus \frak{n}^-$ and $T_{\eta_0} \mathcal{F} = \frak{n}^-$.
By Isawasa decomposition (Cf. \cite[Chapter IX, Thm 1.3]{helgason1978differential}) $G= NAK$, we deduce that $T_o X \simeq \frak{n} \oplus \frak{a}$.
Hence $T_M G/M = \frak{n} \oplus \frak{a} \oplus \frak{n}^- \simeq T_o X \times T_{\eta_0} \mathcal{F}$. 
\end{proof}

\paragraph{Busemann and Iwasawa cocycle}
For every $\xi \in \mathcal{F}$ and $g \in G$, consider, by Iwasawa decomposition $KAN$, the unique element $\sigma(g,\xi) \in \frak{a}$, called the \emph{Iwasawa cocycle}, such that if $k_\xi \in K$ satisfies $k_\xi \eta_0 = \xi$, then
\begin{equation}\label{eq-iwasawa-cocycle}
    g k_\xi \in K \exp( \sigma (g, \xi) ) N .
\end{equation}
The \emph{cocycle relation} holds (Cf. \cite[Lemma 5.29]{benoist-quint}) i.e. for all $g_1, g_2 \in G$ and $\xi \in \mathcal{F}$, then
\begin{equation}\label{eq-relation-Icocycle}
\sigma(g_1 g_2,\xi) = \sigma(g_1, g_2 \xi) + \sigma(g_2 , \xi).    
\end{equation}
Note that restricted to $K \times \mathcal{F}$, the Iwasawa cocycle is the zero function, i.e. for every $k \in K$ and $\xi \in \mathcal{F}$, then $\sigma(k,\xi)=0$.
This motivates the following Definition of the Busemann cocycle for two points of $X$ and an asymptotic Weyl chamber. 
\begin{defin}\label{defin-buseman}
For every $x,y \in X$ and $\xi \in \mathcal{F}$, we define the \emph{Busemann cocycle} by 
$$ \beta_{\xi} (x,y):= \sigma (h_x^{-1} h_y, h_y^{-1} \xi ) $$
independently of the choice of $h_x, h_y \in G$ such that $h_x o = x$ and $h_y o = y$.
\end{defin}
Remark that for every $x,y \in X$ and $\xi \in \mathcal{F}$, for all $g \in G$ and all $z \in X$,
\begin{align}\label{eq-relationBcocycle}
    \beta_{ \xi}(x, y )  &= \beta_{g\xi}(gx,gy) \\
    \beta_\xi (x,y) &= \beta_\xi(x,z) + \beta_\xi(z,y).
\end{align}
The first equation is the $G$-invariance of the formula, whereas the second is due to the cocycle relation of the Iwasawa cocycle.

\paragraph{Transverse points in $\cal{F}$}
The subset of \emph{ordered transverse pairs} of $\cal{F} \times \cal{F}$ is defined by the $G$-orbit
\begin{equation} \label{eq-defin-transverse-pts}
    \cal{F}^{(2)}:= \lbrace (g \eta_0,g\zeta_0) \; \vert \; g \in G \rbrace. 
\end{equation}
We say that $\xi, \eta \in \cal{F}$ are \emph{opposite} or \emph{transverse} if $(\xi,\eta) \in \cal{F}^{(2)}$.

In terms of asymptotic Weyl chambers, $\xi, \eta \in \cal{F}$ are opposite when there exists a geometric Weyl chamber $g. (A^+ o)$ asymptotic to $\xi$ such that $g. ( (A^+)^{-1}o )$ is asymptotic to $\eta$.
Note that (Cf. §3.2 \cite{thirion}) we have the following identifications
$$ \cal{F}^{(2)} \simeq G/AM .$$
\begin{defin}\label{defin-max-flat}
For every $(\xi,\eta)\in \cal{F}^{(2)}$, for any choice $g_{\xi,\eta} \in G$ such that $g_{\xi,\eta} (\eta_0,\zeta_0)= (\xi,\eta)$, we denote by 
$$(\xi \eta)_X:= g_{\xi,\eta} . (Ao)$$
the associated maximal flat in the symmetric space $X$.

For every $(x,\xi) \in X \times \cal{F}$,  we denote by 
$\xi_x^\perp \in \cal{F}$ the unique opposite point to $\xi$ such that $x \in (\xi \xi_x^\perp )_X$.
Equivalently, $\xi_x^\perp:= g_{x,\xi}\zeta_0$, where $g_{x,\xi}M\in G/M$ corresponds (Cf. Fact \ref{fait-asymptotic-point}) to the geometric Weyl chamber of base point $x$ asymptotic to $\xi$.
\end{defin}
\begin{rem}\label{rem-zeta-eta}
Note that $(\zeta_0)_o^\perp = \eta_0$ and vice-versa.
\end{rem}

\paragraph{Hopf coordinates}
Let $\cal{H}$ be the Hopf coordinate map of $G/M$ (Cf. \cite[Chapter 8, §8.G.2]{thirion} or \cite{dg20})
\begin{align*}
\cal{H}: G/M &\rightarrow \calF^{(2)}\times \frak a\\
gM &\mapsto (g\eta_0,g\zeta_0,\sigma(g,\eta_0)).
\end{align*}
Hopf coordinates are left $G$-equivariant and right $A$-equivariant in the following sense: 
\begin{itemize}
    \item[(i)] the left action of $G$ on $G/M$ reads in those coordinates equivariantly on $\cal{F}^{(2)}$ and using the Iwasawa cocycle as follows. For all $h\in G$ and $(\xi,\eta,Y)\in\cal{F}^{(2)}\times \frak a$,
\begin{equation}\label{equ-hxi}
h(\xi,\eta,Y)=(h\xi,h\eta, Y+\sigma(h,\xi)).
\end{equation}
    \item[(ii)] the right action of $A$ on $G/M$ reads for all $(\xi,\eta,Y) \in \cal{F}^{(2)} \times \frak{a}$ and $a \in A$ by keeping the first two coordinates constant and translating the last one by $\log(a)$ 
    $$\cal{H}( \cal{H}^{-1}(\xi,\eta,Y) a  ) = (\xi,\eta, Y + \log(a)) .$$
\end{itemize}

Using the geometric Weyl chamber interpretation and the Busemann cocycle notations,
the Hopf map reads the same as in Roblin's work \cite{roblin}:
\begin{equation}\label{equ_hopfmap}
\begin{split}
    X \times \calF & \longrightarrow \calF^{(2)} \times \frak{a} \\
    (x,\xi) & \longmapsto (\xi, \xi_x^\perp , \beta_{\xi}(o,x) ).
\end{split}
\end{equation}
This translated map is also left $G$-equivariant in the sense that for every $g\in G$ and every $(x,\xi) \in X \times \calF$, using the cocycle relation \eqref{eq-relationBcocycle}, the element $(gx,g\xi)$ has Hopf coordinates
$$(g \xi, g \xi_x^\perp , \beta_{g \xi}(o,go) + \beta_{\xi}(o,x) ).$$
Note that $\beta_{g \xi} (o,go) = \sigma(g,\xi)$, therefore the notations are consistent.

\subsection{The Furstenberg boundary}
\paragraph{Representations of a semisimple Lie group}
Let us first recall a few facts about representations of a semisimple Lie group.
Let $(V,\rho)$ be a representation of $G$ into a real vector space of finite dimension.
For every real character $\chi : \frak a \rightarrow \mathbb{R}$, we denote by
$$ V_\chi := \lbrace v \in V \; \vert \; \rho(u)v=\chi(u)v, \; \forall u \in \frak a \rbrace $$
the associated vector space. 
The set of \emph{restricted weights} is the subset 
$$ \Sigma(\rho):= \lbrace \chi \; \vert \; V_\chi \neq \{ 0\} \rbrace. $$
They are partially ordered using the positive Weyl chamber as follows.
$$ (\chi_1 \leq \chi_2) \Leftrightarrow ( \chi_1 (u) \leq \chi_2(u), \; \forall u \in \frak a^+  ). $$
When the representation $\rho$ is irreducible, the set of restricted weights admits a maximum, called the \emph{maximal restricted weight}.
The irreducible representation $\rho$ is \emph{proximal} when the subspace of the maximal restricted weight is a line.

\paragraph{Restricted weights of the fundamental representations}
For the adjoint representation, the set of restricted weights coincides with the set of restricted roots $\Sigma$. Denote by $\Sigma^+$ the set of positive restricted roots and by $\Pi \subset \Sigma^+$ the set of simple roots. 
Tits (\cite[Lemma 6.32]{benoist-quint}) proved that for every $\alpha\in \Pi$, there exists an irreducible and proximal representation $(\rho_\alpha,V^\alpha)$ of $G$ such that the restricted weights are in 
\begin{equation}\label{eq-restricted-weights}
\bigg\lbrace \chi^\alpha,\; \chi^\alpha - \alpha ,\; \chi^\alpha -\alpha - \sum_{\beta \in \Pi } n_\beta \beta \; \bigg\vert \; n_\beta \in \mathbb{Z}_+ \bigg\rbrace.
\end{equation}
Furthermore, the maximal weights $(\chi^\alpha)_{\alpha\in \Pi}$ of these representations form a basis of $\frak{a}^*$.

\paragraph{Distances in the projective space}
For every $\alpha\in \Pi$, we choose a Euclidean norm $\Vert . \Vert$ on $V^\alpha$ such that the elements in $\rho_\alpha(A)$ (resp. $\rho_\alpha(K)$) are symmetric (resp. unitary).
Note that $\Vert \rho_\alpha(a) \Vert=\exp( \chi^\alpha(\log a))$  for all $a\in A^+$.
Abusing notation, we denote by $\Vert . \Vert$ the induced  Euclidean norm on $V^\alpha \wedge V^\alpha$.
Remark that for all $a\in A^+$,
\begin{equation}\label{norm-wedge}
     \Vert \wedge_2 \rho_\alpha(a) \Vert =\exp( (2\chi^\alpha- \alpha)\log a) .
\end{equation}
We define the distance in the projective space for all $x,y\in\mathbb{P}(V^\alpha)$ as follows,
\begin{equation}\label{defin-dist-proj}
\dd (x,y):= \frac{\Vert v_x \wedge v_y \Vert}{\Vert v_x \Vert. \Vert v_y\Vert}
\end{equation}
independently of the choice of $v_x, v_y \in V$ such that $x=\mathbb{R}v_x$ and $y=\mathbb{R}v_y$.
For all $x \in \mathbb{P}(V^\alpha)$ and $\varepsilon \in (0,1]$, denote by $B(x,\varepsilon)$ the ball centered at $x$ of radius $\varepsilon$ for this distance.

Denote by $x_+^\alpha \in \mathbb{P}(V^\alpha)$ the projective point corresponding to the eigenspace for the maximal restricted weight $\chi^\alpha$. 
Since $\rho_\alpha(A)$ are symmetric endomorphisms for the Euclidean norm on $V^\alpha$, the orthogonal hyperplane to $x_+^\alpha$ is $\rho_\alpha(A)$-invariant and abusing notations we write 
$$(x_+^\alpha)^\perp = \oplus_{\chi \in \Sigma(\rho_\alpha) \setminus \lbrace \chi^\alpha \rbrace } V^\alpha_{\chi}.$$
For all projective point $y\in \mathbb{P}(V^\alpha)$, we denote by $y^\perp \subset V^\alpha$ the orthogonal hyperplane and by $\varphi_y \in (V^\alpha)^*$ a linear form such that $\ker \varphi_y =y^\perp$. 
For all $x,y \in \mathbb{P}(V^\alpha)$, we define (independently of the choice of non-zero $v_x \in x$)
\begin{equation}\label{defin-delta-proj}
\delta(y,x):= \frac{ \vert \varphi_y(v_x) \vert }{ \Vert \varphi_y \Vert. \Vert v_x \Vert}.
\end{equation}
By properties of the norms and distances on the projective space, the previous function is symmetric and for all $x,y\in \mathbb{P}(V^\alpha)$,
\begin{equation}\label{defin-delta-dist}
\delta(y,x)= \delta(x,y)= \dd (y^\perp, x) = \dd (y, x^\perp).
\end{equation}
Hence $\dd (x_+^\alpha, (x_+^\alpha)^\perp)=1.$
For all $\varepsilon>0$, denote by 
$\cal{V}_\varepsilon ( (x_+^\alpha)^\perp  )^\complement := \lbrace y^\alpha \in \mathbb{P}(V^\alpha) \;\vert \; \delta(y^\alpha, x_+^\alpha) \geq \varepsilon \rbrace.$ 
We prove the following dynamical lemma.
\begin{lem}\label{lem_rhoa-action}
Let $\varepsilon >0$ and $a \in A^{+}$. Assume there exists $\alpha\in \Pi$ such that $\alpha (\log a) \geq -2\log (\varepsilon)$.
Then $\rho_\alpha (a) \cal{V}_\varepsilon ((x_+^\alpha)^\perp )^\complement  \subset B(x_+^\alpha,\varepsilon)$.  
\end{lem}
\begin{proof}
We use the notations in §14.1 \cite{benoist-quint}. 
Let $\alpha \in \Pi$ such that $\alpha(\log a) \geq -2\log(\varepsilon)$.
Recall \eqref{norm-wedge} that $\Vert \wedge_2 \rho_\alpha(a) \Vert =\exp( (2\chi^\alpha- \alpha)\log a)$ and $\Vert \rho_\alpha(a) \Vert=\exp(\chi^\alpha(\log a))$.
We compute the gap between the first and second eigenvalues of $\rho_\alpha(a)$,
\[\gamma_{1,2}(\rho_\alpha(a)):=\frac{\|\wedge_2 \rho_\alpha(a)\|}{\|\rho_\alpha(a) \|^2}=e^{-\alpha(\log a)}. \]
By assumption, $e^{-\alpha(\log a)} < \varepsilon^2$, hence $\gamma_{1,2}(\rho_\alpha(a)) < \varepsilon^2$.
Then we apply Lemma 14.2 (iii) in \cite{benoist-quint}, for every $y \in \cal{V}_\varepsilon ( (x_+^\alpha)^\perp )^\complement$,
\[ \dd( \rho_\alpha(a) y ,x_+^\alpha)\delta(x_+^\alpha, y)<\gamma_{1,2}(\rho_\alpha(a)). \]
By definition $\delta (y,x_+^\alpha) \geq \varepsilon$, hence $\dd(\rho_\alpha(a)y,x_+^\alpha) <\varepsilon$ and we deduce that $\rho_\alpha(a) \cal{V}_\varepsilon ( (x_+^\alpha)^\perp)^\complement \subset B(x_+^\alpha,\varepsilon)$.
\end{proof}

\paragraph{Distances and balls in $\cal{F}$}
Using the fundamental representations $(\rho_\alpha)_{\alpha\in \Pi}$, Tits (Cf. \cite[Lemma 6.32]{benoist-quint}) also proved that the following map is an embedding:
\begin{align*}
\cal{F} &\longrightarrow \prod_{\alpha\in \Pi} \mathbb{P}(V^\alpha) \\
\xi=k \eta_0 & \longmapsto (x^\alpha(\xi))_{\alpha\in \Pi}:= (\rho_\alpha(k) x_+^\alpha)_{\alpha \in \Pi}.
\end{align*}
We thus define the following distance on $\cal{F}$ for all $\xi,\eta \in \cal{F}$
\begin{equation}\label{defin-dist}
\dd (\xi,\eta):= \sup_{\alpha \in \Pi} \dd (x^\alpha(\xi),x^\alpha(\eta)).
\end{equation}
For all $\xi \in \calF$ and $\varepsilon\in (0,1)$, we denote the balls for this distance by
\begin{equation}\label{defin-ballFd}
    B(\xi,\varepsilon):= \lbrace \eta\in \calF \; \vert \; \dd(\xi,\eta)<\varepsilon \rbrace.
\end{equation}
\begin{fait}
The distance $\dd$ is equivalent to the Riemanian distance on $\calF$ induced from the embedding on the product space $\Pi_{\alpha\in\Pi}\P(V^\alpha)$. 
\end{fait}
\begin{proof}
Recall that two distances $d,d'$ on a space $X$ are equivalent if $d \asymp d'$ i.e. there exists $C>1$ such that for all $x,y$ in $X$, we have
\[ \frac{1}{C}d'(x,y)\leq d(x,y)\leq Cd'(x,y). \]

On the projective space $\P(V^\alpha)$, for each point $\R v$, its tangent space is given by $v^\perp$, the hyperplane orthogonal to $v$ with respect to the Euclidean norm and we obtain a Riemannian metric by restricting the Euclidean norm to $v^\perp$. 
Denote by $d_\alpha$ the induced Riemannian distance on $\P(V^\alpha)$. 
The distance $d_\alpha$ between two lines is given by their angle in $[0,\pi/2]$. 
Since the distance $\dd$ between two lines defined in \eqref{defin-dist-proj} is the sine of the angle given by $d_\alpha$, we deduce that $\dd \asymp d_{\alpha}$.

Let us now construct a Riemanian distance $d_P$ on the product space $\Pi_{\alpha\in\Pi}\P(V^\alpha)$ using the Riemanian metric of the product space. 
Recall that on any product space $(X\times Y,g)$ where $(X,g_1)$ and $(Y,g_2)$ are endowed with Riemanian metrics $g_1$ and $g_2$, the product Riemanian metric is given for all $(x,y ; v,w) \in T_{(x,y)} X \times Y$
where $(x,v)\in T_x X$ and $(y,w) \in T_y Y$, by
\[g(x,y;v,w)=g_1(x,v)+g_2(y,w). \]
The Riemanian distance $d$ associated to this product Riemannian metric $g$ satisfies
\[ \max\{d_1,d_2\}\leq d \leq d_1+d_2. \]
Since for every $\alpha \in \Pi$, the distances $d_\alpha$ and $\dd$ are equivalent, we deduce that the Riemanian product distance $d_P$ is equivalent to the maximal metric i.e. $d_P \asymp \dd:=\sup_{\alpha \in \Pi} \dd $.
Using Tits' embedding of $\calF$ in to the product space $\Pi_{\alpha\in\Pi}\P(V^\alpha)$, we deduce that the induced metric is non-degenerate on $\calF$. 
Hence, the Riemannian distance on $\calF$ induced by $d_P$ is equivalent to the maximal distance $\dd$. 
\end{proof}
\noindent Similarly, noting that $(\zeta_0)_o^\perp = \eta_0$, we set
\begin{equation}\label{defin-delta}
\delta (\xi,\eta):= \inf_{\alpha \in \Pi} \delta (x^\alpha(\xi),x^\alpha(\eta_o^\perp)) = \inf_{\alpha \in \Pi} \dd(x^\alpha(\xi), x^\alpha(\eta_o^\perp)^\perp ).
\end{equation}
For all $\xi \in \mathcal{F}$ and $\varepsilon\in (0,1)$, we denote the balls for $\delta$ by
\begin{equation}\label{defin-ball-Fdel}
    \cal{V}_\varepsilon (\xi) := \lbrace \eta \in \calF \; \vert \; \delta(\eta,\xi) <\varepsilon \rbrace.
\end{equation}
Using the above notations given for the balls in $\calF$ for $\delta$ and $\dd$ and their $K$-invariance, we upgrade the dynamical Lemma \ref{lem_rhoa-action} to elements in $G$ whose Cartan projection is far from the walls of the Weyl chambers.
\begin{lem}\label{lem-action-aplus}
For all $g \in G$, choose $k,l\in K$ by Cartan decomposition such that $g=k \exp(\underline{a}(g)) l^{-1}$.
Let $\varepsilon>0$ and assume that $\dd ( \underline{a}(g) , \partial \frak{a}^+) \gg -2 \log \varepsilon$,
then $g \cal{V}_\varepsilon(l \zeta_0 )^\complement  \subset B( k \eta_0,\varepsilon)$. 
\end{lem}
\begin{proof}
Note that $\alpha (v) \asymp \dd(v, \ker \alpha)$ for all $v \in \frak{a}^+$.
Hence by taking the infimum over $\alpha \in \Pi$, then using that $\inf_{\alpha \in \Pi} \dd (v, \ker \alpha) = \dd (v, \cup_{\alpha \in \Pi} \ker \alpha )$ and finally, because $\frak a^+$ is a salient cone, $\partial \frak{a}^+ = \frak{a}^+ \cap \big( \cup_{\alpha \in \Pi} \ker \alpha \big)$, we deduce that for all $v \in \frak{a}^+$,
$$ \dd(v, \partial \frak{a}^+) \asymp \inf_{\alpha \in \Pi} \alpha (v).$$
Now using the underlying constant in $\asymp$, we may assume that, $\inf_{\alpha \in \Pi} \alpha(\underline{a}(g)) \geq - 2 \log \varepsilon$.
Apply the dynamical Lemma \ref{lem_rhoa-action} simultaneously for all $\alpha \in \Pi$, using Remark \ref{rem-zeta-eta} that $(\zeta_0)_o^\perp= \eta_0$, we deduce that $e^{\underline{a}(g)} \cal{V}_\varepsilon (\zeta_0)^\complement \subset B(\eta_0,\varepsilon)$.
Finally, we deduce the Lemma by invariance of left $K$-action on both $\dd$ and $\delta$.
\end{proof}

\paragraph{Action of $G$ on $\cal{F}$}
We want to understand how the left action of $G$ on $\calF$ distorts the balls for $\delta$ and $\dd$.
Let $C_\frak{a}>1$ be a positive constant such that for all $v \in \frak{a}$,
 $$\frac{1}{\sqrt{C_{\frak{a}} }} \Vert v \Vert \leq  \sup_{\alpha \in \Pi}  \vert \chi^\alpha ( v ) \vert \leq  \sqrt{C_\frak{a}} \Vert v \Vert  .$$
This constant gives the comparison of the sup-norm induced by the dual basis $(\chi^\alpha)_{\alpha \in \Pi}$ with the Euclidean norm $\Vert \; \Vert$ on $\mathfrak{a}$.

\begin{lem}\label{lem-actiong}
There exist $C_0, C_1>0$ such that for all $g$ in $G$ and $\xi,\eta$ in $\calF$, we have the following inequalities:
\begin{itemize}
    \item[(i)] $\dd (g \xi, g \eta) \leq C_1 e^{ C_0 \dd_X(o,go)   } \dd(\xi,\eta),$
    \item[(ii)] $\delta (g \xi, g \eta) \leq C_1 e^{ C_0 \dd_X(o,go)   } \delta (\xi,\eta),$
    \item[(iii)] $ \Vert \sigma (g, \xi) - \sigma(g,\eta) \Vert  \leq C_1 e^{ C_0 \dd_X(o,go)   } \dd(\xi,\eta),$
    \item[(iv)] $ \Vert \sigma(g,\xi) \Vert \leq C_\frak{a} \dd_X(o,go) .$
\end{itemize}
Furthermore, for every $x,y \in X$ and $\xi \in \calF$, (iv) is the same as
\begin{itemize}
        \item[(iv')]$ \Vert \beta_{\xi}(x,y) \Vert \leq C_{\frak{a}} \dd_X(x,y) .$
\end{itemize}
\end{lem}
\noindent In particular, for all $x\in X$ we set $C_x:= C_1 e^{C_0 \dd_X(o,x)}$. 
Then for all $h_x \in G$ such that $h_xo=x$ and all $\xi \in \calF$ and every $r \in (0, C_x^{-1}) $, the inequalities given by (i) and (ii) imply 
\begin{itemize}
    \item[(i')]  $B(h_x \xi, C_x^{-1} r) \subset h_x B(\xi,r) \subset B(h_x \xi, C_x r)$,
    \item[(ii')] $ \cal{V}_{C_x^{-1}r}(h_x \xi) \subset h_x \cal{V}_r(\xi) \subset  \cal{V}_{C_x r}(h_x\xi)$.
\end{itemize}

\begin{proof}
For each $V^\alpha$, by (13.1) in \cite{benoist-quint}, we have
\[\dd(x^\alpha(g\xi),x^\alpha(g\eta))\leq \|\rho_\alpha(g)\|^2\|\rho_\alpha(g^{-1})\|^2 \dd (x^\alpha(\xi),x^\alpha(\eta)). \]
By \eqref{defin-dist} and $\|\rho_\alpha(g)\|=\|\rho_\alpha \exp(\underline{a}(g)) \|=\exp(\chi^\alpha(\underline{a}(g)))$, we obtain the first inequality for $C_0=4 C_{\frak{a}}$.

For (ii), we first prove that $(x^\alpha((g\eta)_o^\perp))^\perp=\rho_\alpha(g)x^\alpha(\eta_o^\perp)^\perp$. There exist $k_1,k\in K$ such that $\eta=k_1\eta_0$ and $gk_1=kan\in KAN$. Then due to $k$ preserving $o$ and the Euclidean metric on $V^\alpha$, we obtain
\[(x^\alpha((g\eta)_o^\perp))^\perp=(x^\alpha((k\eta_0)_o^\perp))^\perp=\rho_\alpha(k)(x^\alpha((\eta_0)_o^\perp))^\perp. \]
Due to $AN$ preserving $(x^\alpha((\eta_0)_o^\perp))^\perp=(x^\alpha(\zeta_0))^\perp$, we deduce that $\rho_\alpha(k)(x^\alpha((\eta_0)_o^\perp))^\perp=\rho_\alpha(gk_1)(x^\alpha((\eta_0)_o^\perp))^\perp$. Therefore, we obtain $(x^\alpha((g\eta)_o^\perp))^\perp=\rho_\alpha(g)(x^\alpha(\eta_o^\perp))^\perp$.
Then for all $\xi,\eta \in \calF,$
\begin{align*}
    \delta(x^\alpha(g\xi),x^\alpha((g\eta)^\perp_o)) &=\dd(x^\alpha(g\xi),x^\alpha((g\eta)_o^\perp)^\perp)=\dd(\rho_\alpha(g)x^\alpha(\xi)^\perp,\rho_\alpha(g)x^\alpha(\eta_o^\perp)^\perp) \\
    &\leq \Vert \rho_\alpha(g) \Vert^2 \|\rho_\alpha(g)^{-1} \|^2  \;  \dd(x^\alpha(\xi),x^\alpha(\eta_o^\perp)^\perp) .
\end{align*}
Therefore, since $\Vert \rho_\alpha(g) \Vert \|\rho_\alpha(g)^{-1} \| \leq \exp ( 2 \sup ( 
 \chi^\alpha ( \underline{a}(g)), \chi^\alpha( \iota \underline{a}(g)) ) )$ and $C_0=4 C_{\frak{a}}$, we deduce that
\[ \delta(g\xi,g\eta)=\inf_{\alpha\in\Pi}\delta(x^\alpha(g\xi),x^\alpha((g\eta)_o^\perp))\leq C_1e^{C_0\|\underline{a}(g)\|} \delta(\xi,\eta). \]

(iii) is given in \cite[Lemma 13.1]{benoist-quint}.

(iv), see \cite[Lemma 3.12]{dg20} for a similar statement, and it is also a direct consequence of \cite[Lemma 6.33 (ii), Corollary 8.20]{benoist-quint}.

Finally (iv') is a consequence of the formulas $\beta_\xi(x,y)=\sigma( h_x^{-1} h_y, h_y^{-1} \xi)$ and $\dd_X(x,y)=\Vert \underline{a}(h_x^{-1} h_y) \Vert$ independently of the choice of $h_x, h_y \in G$ such that $h_xo=x$ and $h_yo=y$.
\end{proof}

\subsection{Disintegration of the Haar measure}
Patterson--Sullivan measures were generalized to the higher rank setting in \cite{albuquerque99}, \cite{quintMesuresPattersonSullivanRang2002}. 
We follow Thirion's \cite[Chapter 9 §9.e]{thirion} construction of higher rank Patterson--Sullivan measures on the space of Weyl chambers for $\mathrm{SL}(d,\mathbb{R})$, which also works in our more general setting. 

We start by his so-called Patterson densities.
For $x\in X$, let $K_x$ be the stabilizer group of $x$ in $G$. Let $\mu_x$ be the unique $K_x$ invariant probability measure on the Furstenberg boundary $\calF$. Then we have for $g\in G$ and $x\in X$
\begin{equation}\label{equ.gmux}
    g_*\mu_x=\mu_{gx}, 
\end{equation}
where $g_*\mu_x$ is the pushforward of $\mu_x$ under the $g$ action. This relation holds because the stabilizer of $g_*\mu_x$ is given by $gK_xg^{-1}=K_{gx}$.
Let $\rho=\frac{1}{2}\sum_{\alpha\in\Sigma}\alpha$ be the half of the sum of positive roots with multiplicities. By \cite[Lemma 6.3]{quintMesuresPattersonSullivanRang2002} or \cite[I 5.1]{helgasonGroupsGeometricAnalysis2000}, we have for $g$ in $G$
\begin{equation}\label{equ.gmu0}
    \frac{ \dd g_*\mu_o}{\dd\mu_o}(\xi)=e^{-\rho \sigma(g^{-1},\xi)},
\end{equation}
which is a $G$ quasi-invariant measure. Then we will introduce the Gromov product to obtain a $G$-invariant measure on $\calF^{(2)}$.
\begin{defin}\label{defin_gromov}
For a pair $(\xi,\eta)\in\calF^{(2)}$, we associate it with the unique element in the Lie algebra $\frak a$ such that for all weights $\chi^\alpha$
\[\chi^\alpha(\xi|\eta)_o:= -\log \delta (x^\alpha(\xi), x^\alpha(\eta_o^\perp)) =  -\log\frac{|\varphi(v)|}{\|\varphi\|\|v\|}, \]
where $v\in V^\alpha-\{0\}$ is a representative of $x^\alpha(\xi)$ and $\varphi$ is a non zero linear form such that $\ker \varphi = x^{\alpha}(\eta_o^\perp)^\perp$.
\end{defin}
Since the $\delta$ function \eqref{defin-delta-dist} takes value in $(0,1]$, then $\chi^\alpha (\xi | \eta)_o \in [0, + \infty)$.
This definition already appears in \cite[Section 8.10]{bochiAnosovRepresentationsDominated2019}, \cite[Section 4]{sambarinoOrbitalCountingProblem2015} for semisimple Lie groups and \cite{thirion} for $\slr$. 
Our definition of $\delta$ seems different from the one in \cite{bochiAnosovRepresentationsDominated2019}, \cite{sambarinoOrbitalCountingProblem2015}. By using the correspondence between linear forms and hyperplanes for Euclidean spaces, we can verify that they are the same.
An important property is that \cite[Lemma 4.12]{sambarinoOrbitalCountingProblem2015}: for all $g\in G$ and $(\xi,\eta)\in\calF^{(2)}$, we have 
\begin{equation}\label{equ.gromov}
    (g\xi|g\eta)_o-(\xi|\eta)_o=\iota\sigma(g,\xi)+\sigma(g,\eta),
\end{equation}
where $\iota$ is the inverse involution on $\frak a$.
We also define the Gromov product at other points $x$ in $X$ by $G$-invariance, by setting
\[(\xi|\eta)_x=(h_x^{-1}\xi|h_x^{-1}\eta)_o, \]
where $h_x$ is some element such that $h_xo=x$.
Since by \eqref{equ.gromov}, the Gromov product at $o$ is left $K$-invariant, this definition is independent of the choice of $h_x$. 
For all $x\in X$ and $(\xi,\eta)\in\calF^{(2)}$, we define the $(0,1]$-valued function
\[f_x(\xi,\eta)=\exp(-\rho(\xi|\eta)_x). \]
We define measures $\nu_x$ on $\calF^{(2)}$ by
\begin{equation}\label{equ.nu}
    \dd\nu_x(\xi,\eta)=\frac{\dd\mu_x(\xi)\dd\mu_x(\eta)}{f_x(\xi,\eta)}.
\end{equation}

\begin{prop}\label{prop-PS-BMS}
For all $x\in X$, the measure $\nu_x$ is $G$-invariant and equal to $\nu_o$. We denote it by $\nu$.
\end{prop}
\noindent In the hyperbolic case, the measures $\mu_x$ are called Patterson-Sullivan and $\nu \otimes Leb_{\mathbb{R}}$ is the Bowen-Margulis-Sullivan measure.
In the $\slr$ case, Thirion \cite{thirion} gave a construction of this measure and proved those properties. 
We include a proof for completeness.
\begin{proof}
By \eqref{equ.gromov}, for all $x\in X$, all $(\xi,\eta)\in \calF^{(2)}$ and every $h_x \in G$ such that $h_x o =x$
\[f_x(\xi,\eta)=f_o(h_x^{-1}\xi,h_x^{-1}\eta)=f_o(\xi,\eta)\exp(-\rho(\iota\sigma(h_x^{-1},\xi)+\sigma(h_x^{-1},\eta)) \]
On the other hand,
\[\frac{\dd\mu_x}{\dd\mu_o}(\xi)=\frac{\dd(h_x)_*\mu_o}{\dd\mu_o}(\xi)=e^{-\rho\sigma(h_x^{-1},\xi)}. \]
We obtain the same formula for $\eta$.
Combing the above two equations and using that $\rho\iota\sigma(h_x^{-1},\xi)=\rho\sigma(h_x^{-1},\xi)$, we obtain that 
\[\nu_x=\nu_o. \]
By definition of the Gromov product, we have for all $g \in G$
\[f_x(g\xi,g\eta)=f_{g^{-1}x}(\xi,\eta). \]
By equation \eqref{equ.gmux} and using that $\nu_{g^{-1}x}=\nu_x$,
\begin{align*}
    \dd\nu_x(g\xi,g\eta)=\frac{\dd\mu_x(g\xi)\dd\mu_x(g\eta)}{f_x(g\xi,g\eta)}=\frac{\dd\mu_{g^{-1}x}(\xi)\dd\mu_{g^{-1}x}(\eta)}{f_{g^{-1}x}(\xi,\eta)}=\dd\nu_{g^{-1}x}(\xi,\eta)=\dd\nu_x(\xi,\eta).
\end{align*}
Hence $\nu_x$ is $G$-invariant.
\end{proof}
Bochi, Potrie and Sambarino proved that the Gromov product $(\xi|\eta)_o$ in norm is almost the same as the distance between $o$ and the maximal flat $(\xi\eta)_X \subset X$.
\begin{lem}\label{lem-xietao}
\cite[Proposition 8.12]{bochiAnosovRepresentationsDominated2019}
There exist $C_3>1,C'>0$ such that for any $(\xi,\eta)\in \calF^{(2)}$, we have
\begin{equation*}
    \frac{1}{C_3}\|(\xi|\eta)_o\|\leq \dd_X(o,(\xi\eta)_X)\leq C_3\|(\xi|\eta)_o\|+C'.
\end{equation*}
\end{lem}
\noindent By $G$-invariance, we deduce that for every $x \in X$ and $(\xi,\eta) \in \calF^{(2)}$
$$ \frac{1}{C_3} \Vert (\xi | \eta)_x \Vert \leq \dd_X(x, (\xi \eta)_X) \leq 
C_3 \Vert (\xi | \eta)_x \Vert + C' .$$
With this $G$-invariant measure $\nu$ on $\calF^{(2)}$, now we can disintegrate the Haar measure on $G/M$ along Hopf coordinates.

\begin{prop}\label{prop-disintegration}
The product measure $\nu\otimes Leb$ on $\calF^{(2)}\times\frak a$ is a disintegration in Hopf coordinates of a Haar measure on $G/M$.
\end{prop}

\begin{proof}
The product measure $\nu\otimes Leb$ is $G$-invariant by Proposition \ref{prop-PS-BMS} and the Hopf coordinates. So it is a Haar measure on $G/M$.
\end{proof}

\subsection{Cartan regular isometries}
Recall that by Cartan decomposition, for every element $g \in G$ there exist $k,l\in K$ and a unique element $\underline{a}(g) \in \frak{a}^+$ such that $g=k \exp (\underline{a}(g)) l^{-1}$. 
Note that $k$ and $l$ are defined up to right multiplication by elements in $Z_K( \exp(\underline{a}(g)))$. 



\begin{defin}\label{defin-cartan-regular}
For all $x \in X$, we denote by $\underline{a}_x: G \rightarrow \mathfrak{a}^+$
the map that assigns to every $g \in G$ the $\frak{a}^+$-distance between $x$ and $gx$, i.e. $\underline{a}_x(g):= d_{\underline{a}}(x,gx).$ We say that $g$ is \emph{$x$-cartan regular} if $\underline{a}_x(g) \in \frak{a}^{++}$.

Let $g$ be an $x$-cartan regular element, consider $h,h'\in G$ such that $ho=h'o=x$ with $he^{\underline{a}_x(g)}o=gx$ and $h'e^{ \underline{a}_x(g^{-1})}o=g^{-1}x$.
We set $g_x^+:=h \eta_0$ and $g_x^-:=h'\eta_0$. In particular, when $x=o$, we can take $h=k$ and $h'=lk_\iota$.
\end{defin}
Note that every $g\in G$ we have $\underline{a}_x(g)= \underline{a}( h_x^{-1} g h_x)$, independently of the choice of $h_x \in G$ such that $h_xo=x$.
Furthermore, provided that $g$ is $x$-cartan regular, 
\begin{equation}\label{equ.gammax}
    g_x^{\pm}=h_x(h_x^{-1} g h_x)_o^{\pm}.
\end{equation}
Remark that $(x,g_x^+) \in X \times \cal{F}$ (resp. $(x,g_x^-)$) is the unique geometric Weyl chamber based on $x$ containing $gx$ (resp. $g^{-1}x$).
In the $\mathrm{PSL}(2,\mathbb{R})$ case, an element $g$ is $x$-cartan regular when $gx\neq x$, then $g_x^+ \in \partial \mathbb{H}^2$ (resp. $g_x^-$) is the asymptotic endpoint of the half geodesic based on $x$ going through $gx$ (resp. $g^{-1}x$).

\begin{lem}\label{lem-cartan-points}
For all $g\in G$, every $x,y\in X$, the following bound holds:
$$ \Vert \underline{a}_x(g)-\underline{a}_y(g) \Vert \leq 2 \mathrm{d}_X(x,y) .$$ 
\end{lem}
\begin{proof}
By \cite[Lemma 2.3]{kasselCorank08} in our choice of notations: for all $h,h'\in G$ we have the following inequalities,
$ \Vert \underline{a}( h h') - \underline{a}(h) \Vert \leq  \Vert \underline{a}(h') \Vert $
and $\Vert \underline{a}( h' h) - \underline{a}(h) \Vert \leq  \Vert \underline{a}(h') \Vert $.
Let $x,y\in X$ and choose $h_x,h_y \in G$ such that $h_x o =x$ and $h_y o =y$.
We compare the Cartan projection of $h= h_y^{-1} g h_y$ to the Cartan projection of its conjugate by $h'=h_y^{-1} h_x$, using that $\Vert \underline{a}(h') \Vert =\Vert \underline{a}(h'^{-1}) \Vert $ we get
$$ \Vert \underline{a}_x(g) - \underline{a}_y(g)  \Vert \leq 2 \Vert \underline{a}(h_y^{-1} h_x) \Vert .$$
Since $\Vert \underline{a}(h_y^{-1} h_x) \Vert= \dd_X(x,y)$, we deduce the Lemma.
\end{proof}

\subsection{Volume growths and decay}
We introduce here some subsets on $G$. They will be used to obtain the main term and the exponentially decaying error term in our main Theorems \ref{corol-counting}, \ref{thm-introequid}.

For $t>1$, let
$$D_t:= K \exp(B_{\frak{a} }(0,t))K,$$
and its subset of Cartan-regular elements
$$D_t^{reg}:= K \exp(B_{\frak{a} }(0,t)\cap \frak a^{++} )K.$$
For $0<s<t$, let 
$$D_t^{s} := \lbrace g \in D_t \; \vert \; \underline{a}(g) \in \overline{B_{\frak{a}}( \partial \frak{a}^+ ,s )  } \rbrace $$
be the set of elements in $D_t$ whose Cartan projection have distance at most $s$ to the boundary of the Weyl chamber.

For all $x\in X$, we define similar sets
$$D_t(x):= h_x D_t h_x^{-1},$$
$$D_t^{reg}(x):= h_x D_t^{reg} h_x^{-1},$$
$$D_t^{s}(x):= h_x D_t^{s} h_x^{-1}. $$
These sets are independent of the choice of $h_x$.

For a subset $S$ of $G$, its volume is defined as its Haar measure $m_G(S)$. 
Recall volume estimates from \cite{knieper_asymptotic_1997}, \cite[Thm 5.8]{helgasonGroupsGeometricAnalysis2000}, \cite[Thm 6.1]{gorodnikIntegralPointsSymmetric2009}. There exist $C_0>0$ and $\delta_0>0$ such that as $t\rightarrow \infty$, we have
\begin{equation}\label{volume_D_t}
\vol(D_t) \sim C_0 t^{\frac{\dim A -1}{2}}e^{\delta_0 t},
\end{equation}
where $\delta_0:=2\max_{Y\in B_{\frak a}(0,1)}\rho(Y)$ and $\rho$ is equal to the half of the sum of positive roots with multiplicities.

\begin{lem}[Prop. 7.1 \cite{gorodnikErgodicTheoryLattice2010}]\label{lem-vollip}
The function $t\mapsto \log \vol(D_t)$ is uniformly locally Lipschitz for $t>1$. 
\end{lem}
This means that there exists $C>0$ such that for all $0<\epsilon<1$, we have
\[\vol(D_{t+\epsilon})\leq e^{C\epsilon}\vol(D_t). \]

\begin{lem}\label{volume_reste}
There exists $\epsilon_G>0$ such that for every $0<\epsilon<\epsilon_G$, there exists $\kappa(\epsilon)>0$ such that for $t>1$ 
\begin{equation}
\frac{ \vol(D_t^{\epsilon t})}{\vol (D_t) } = O(\vol(D_t)^{-\kappa(\epsilon)} ).
\end{equation}
\end{lem}
\begin{proof}
The proof is similar to Lemma 9.2 and 9.4 in \cite{gorodnikDistributionLatticeOrbits2007}. Let $\frak a^+(s,t)=\{v\in\frak a^+ \cap B_{\frak a}(0,t),\,d(v,\partial \frak a^+)\leq s \}$. Then by Harish-Chandra's formula (see \cite[Chapter I Theorem 5.8]{helgasonGroupsGeometricAnalysis2000}), we have
\[\vol(D_t^s)=\int _{\frak a^+(s,t)}\xi(v)\dd v, \]
where 
$$\xi(v)=\prod_{\alpha\in\Sigma^+}\sinh(\alpha(v))^{m_\alpha}\ll e^{2\rho(v)},$$
and $2 \rho = \sum_{\alpha \in \Sigma^+} m_\alpha  \alpha$ where $m_\alpha=\dim \frak g_\alpha$.
By Lemma 9.2 in \cite{gorodnikDistributionLatticeOrbits2007}, if $\epsilon$ smaller than some constant $\epsilon_G$, then by the strict convexity of $\frak a_1^+$, there exists $\kappa'(\epsilon)>0$ such that
\[\max_{v\in \frak a^+(\epsilon,1)}2\rho(v)\leq \delta_0-\kappa'(\epsilon). \]
So by Harish-Chandra's formula, we have $\vol(D_t^{\epsilon t})\ll t^{\dim A}e^{(\delta_0-\kappa'(\epsilon))t}$. Due to the asymptotic of $\vol(D_t)$ \eqref{volume_D_t}, the proof is complete.
\end{proof}

\begin{lem}\label{count-reste}
Let $\Gamma$ be a lattice in $G$, then
for all $t>1$,
\begin{equation*}
    \frac{|\Gamma\cap D_t^{\epsilon t}|}{\vol(D_t)}=O(\vol(D_t)^{-\kappa(\epsilon)}).
\end{equation*}
\end{lem}
\begin{proof}
Let $\epsilon'>0$ be a small constant such that the ball contered at $e$ with radius $\epsilon'$ satisfies $B(e,\epsilon')^2\cap\Gamma=\{e\}$. Then we have
\begin{equation*}
    |\Gamma\cap D_t^{\epsilon t}|\leq \frac{\vol( B(e,\epsilon')D_t^{\epsilon t})}{\vol(B(e,\epsilon'))}.
\end{equation*}
By \cite[Lemma 2.3]{kasselCorank08}, we have for $h'\in B(e,\epsilon')$ and $h\in D_t^{\epsilon t}$,
\[ \Vert \underline{a}( h' h) - \underline{a}(h) \Vert \leq  \Vert \underline{a}(h') \Vert\leq \ell\epsilon', \]
for some $\ell>0$. Therefore the product set
\[ B(e,\epsilon')D_t^{\epsilon t}\subset D_{t+\ell\epsilon'}^{\epsilon t+\ell\epsilon'}.  \]
Hence we have
\[|\Gamma\cap D_t^{\epsilon t}|\leq  \frac{\vol(D_{t+\ell\epsilon'}^{\epsilon t+\ell\epsilon'})}{\vol(B(e,\epsilon'))}, \]
which is $O(\vol(D_t)^{1-\kappa(\epsilon)})$ due to Lemma \ref{lem-vollip} and \eqref{volume_reste}.
\end{proof}

As a corollary, we have
\begin{lem}\label{lem-dtx}
For $0<\epsilon<\epsilon_G/2$, $t>1$ and $x\in X$ with $d_X(o,x)<\min\{\frac{\epsilon}{2(1-2\epsilon)},\frac{\kappa(2\epsilon)}{4(1-\kappa(2\epsilon))} \} t$, we have
\begin{equation*}
    \frac{|\Gamma\cap D_t^{\epsilon t}(x)|}{\vol(D_t)}=O(\vol(D_t)^{-\kappa(2\epsilon)/2}).
\end{equation*}
\end{lem}

\begin{proof}
By Lemma \ref{lem-cartan-points}, we have
\[\| a_o(\gamma)-a_x(\gamma )\|\leq 2 d_X(x,o). \]
Therefore by Lemma \ref{count-reste} with $2\epsilon$ we obtain
\begin{align*}
|\Gamma\cap D_t^{\epsilon t}(x)|&\leq |\Gamma\cap D_{t+2 d_X(x,o)}^{\epsilon t+2 d_X(x,o)}|\ll \vol(D_{t+2 d_X(x,o)})^{1-\kappa(2\epsilon)},
\end{align*}
where we use the hypothesis that $\epsilon t+2 d_X(x,o)\leq 2\epsilon( t+2 d_X(x,o))$.

By hypothesis, we have
\[ (1-\kappa(2\epsilon))(t+2d_X(o,x))\leq (1-\kappa(2\epsilon)/2)t. \]
Then by $\vol(D_t)\in [1/C,C]e^{\delta_0 t}t^{\frac{\dim A -1}{2}}$, we have
\begin{align*}
\vol(D_{t+2 d_X(x,o)})^{1-\kappa(2\epsilon)}=O(\vol(D_t)^{1-\kappa(2\epsilon)/2}).
\end{align*}
The proof is complete.
\end{proof}

\subsection{Angular distribution of Lattice points}

\begin{theorem}\label{theo-gorodnik-nevo}
\hypG
Let $\Gamma <G$ be an irreducible lattice. There exist $\kappa>0$ and $C_4>0$. 
Let $x \in X$ and $(B_t(x))_{t >0}$ be $D_t(x)^{reg}$.
Then for all Lipschitz test functions $\psi \in Lip(\cal{F}\times\calF)$, there exists $E(t,\psi, x)=O(Lip(\psi) C_x \vol(D_t)^{-\kappa})$ when $t > C_4 d_X(o,x)$ such that

$$ \frac{1}{\vol(B_t)} \sum_{\gamma \in B_t(x)\cap \Gamma} \psi(\gamma_x^+,\gamma_x^-) = \frac{1}{\vol(\Gamma \backslash G)}   \int_{ \cal{F}\times\calF  } \psi \dd \mu_x\otimes\mu_x + E(t,\psi, x) .$$
\end{theorem}

This is due to Gorodnik-Nevo in \cite{gorodnikCountingLatticePoints2012}. We include the proof of this version for Lipschitz functions in the appendix.
As a corollary, combined with Lemma \ref{lem-dtx} we have

\begin{lem}\label{lem-dtxt}
There exists $C_5>0$ and $C>0$ such that if $t> C_5d_X(o,x)$, then 
\[ \frac{|\Gamma\cap D_t(x)|}{\vol(D_t)}\leq C. \]
\end{lem}
\begin{proof}
Due to the definition $C_x=C_1e^{C_0d_X(o,x)}$, we know that if $t\gg d_X(o,x)$, then by taking $\psi=1$ Theorem \ref{theo-gorodnik-nevo} implies that
\[ |\Gamma\cap D_t^{reg}(x)|\ll \vol(D_t). \]
For the part $|\Gamma\cap(D_t(x)-D_t^{reg}(x))|$, if $t\gg d_X(o,x)$, then we can use Lemma \ref{lem-dtx} to bound it. Combing these two parts, we obtain the lemma.
\end{proof}

\section{A configuration for being loxodromic}

Recall Definition \ref{def-Jordan-proj} that the elements in $G$ of Jordan projection in $\frak{a}^{++}$ are called loxodromic.
Equivalently, loxodromic elements are conjugated to elements in $A^{++}M$.
Let $g\in G^{lox}$ be a loxodromic element, choose $h_g \in G$ such that $h_g^{-1} gh_g \in \exp(\lambda(g)) M$.
Note that $g h_gM= h_g e^{\lambda(g)}M$.
Denote by $g^+:= h_g \eta_0$ (resp. $g^-:= h_g \zeta_0$) the attracting (resp. repelling) fixed points in $\calF$ for the action of $g$. They are independent of the choice of $h_g$.
Hence for every $Y \in \frak{a}$, in Hopf coordinates 
\begin{equation}\label{eq_gg+}
    g(g^+, g^- , Y)=(g^+,g^-, Y+ \lambda(g)).
\end{equation}

\subsection{Distances on $G/M$ }\label{sec-dist}

Denote by $\dd_1$ the left $G$-invariant and right $K$-invariant Riemannian distance on $G/M$.

\paragraph{Distance for the Hopf coordinates}

For every pair $(\xi^+,\xi^-, v), (\eta^+,\eta^- ,w) \in \calF^{(2)} \times \frak{a}$, we define 
\begin{equation}\label{defin-distHopf}
    \dd_2 \big( (\xi^+,\xi^-,v), (\eta^+,\eta^-,w) \big) := \sup ( \dd(\xi^+,\eta^+), \dd(\xi^-,\eta^-), \Vert v-w \Vert ).
\end{equation}
Due to the Definitions \eqref{defin-dist}, the distance $\dd_2$ is not left $G$-invariant even though it is left $K$-invariant.
Abusing notations, for every $z_1, z_2 \in G$, we also denote by  
$\dd_2( z_1M , z_2 M ) := \dd_2\big( \cal{H}(z_1 M) , \cal{H}(z_2M)  \big).$
For all $(\xi^+,\xi^-,v) \in \calF^{(2)} \times \frak{a}$, all $r\in (0,\frac{1}{2}\delta(\xi^+,\xi^-) )  $, the ball of radius $r$ for $\dd_2$ centered in that element is
$$ B(\xi^+,r) \times B(\xi^-,r) \times B_{\frak{a}}(v,r) .$$

\begin{lem}\label{lem-d2}
For $g\in G$ and $z_1,z_2$ in $G$, we have
\[\dd_2(gz_1M,gz_2M)\leq \sup\big( C_1e^{C_0\|\underline{a}(g)\|},1 \big)\dd_2(z_1M,z_2M). \]
\end{lem}
\begin{proof}
We write down $z_1M$ and $z_2M$ in Hopf coordinates, we denote by $(\xi_i^+,\xi_i^-,v_i):= \cal{H}(z_iM)$ for $i=1,2$.
By \eqref{equ-hxi} and Lemma \ref{lem-actiong} (i),(ii),(iii), we have
\begin{align*}
    \dd_2(gz_1M,gz_2M)&= \dd_2((g\xi_1^+,g\xi_1^-,v_1+\sigma(g,\xi_1^+)),(g\xi_2^+,g\xi_2^-,v_2+\sigma(g,\xi_2^+))\\
    &= \sup\big( \dd (g \xi_1^+, g\xi_2^+ ), \dd(g \xi_1^-,g \xi_2^-) , \Vert v_1 -v_2 + \sigma(g, \xi_1^+) - \sigma(g,\xi_2^+) \Vert   \big)\\
    &\leq \sup \big( C_1e^{C_0\|\underline{a}(g)\|} \dd(\xi_1^+,\xi_2^+) ,  C_1e^{C_0\|\underline{a}(g)\|} \dd(\xi_1^-,\xi_2^-), C_1e^{C_0\|\underline{a}(g)\|} \dd(\xi_1^+,\xi_2^+)+\|v_1-v_2\| \big) \\
     &\leq \sup \big( C_1e^{C_0\|\underline{a}(g)\|},1 \big) \dd_2(z_1M,z_2M).
\end{align*}
The proof is complete.
\end{proof}

\paragraph{Local equivalence}

Denote by $B_1(zM, r)\subset G/M$ the ball of radius $r$ centered on $zM$, for the distance $\dd_1$. 
\begin{lem}\label{lem-dist-local}
There exist a neighbourhood $O$ of $eM$ and $C_2>0$ such that for every $z_1, z_2 \in O$,
\[\frac{1}{C_2} \dd_2(z_1,z_2)   \leq \dd_1(z_1,z_2)\leq C_2 \dd_2(z_1,z_2). \]
\end{lem}

The main idea is to use the fact that two Riemannian metrics on a manifold are locally equivalent. We have already constructed a Riemannian metric $d_\calF$ on $\calF$ and proved that it is equivalent to the supreme distance $d$ defined in \eqref{defin-dist}.

On the product space $\calF\times\calF\times\frak a$, we have the product distance $d_2$ from $\dd$ on $\calF$ and $d_\frak a$ on $\frak a$. We also have the product Riemannian distance from $d_\calF$ on $\calF$ and $d_\frak a$ on $\frak a$, which is denoted by $d_3$. Due to $d_\calF$ and $\dd$ equivalent, $d_2$ and $d_3$ are equivalent. Now we can use a lemma about comparing Riemannian distances. We call two distances $d,d'$ locally equivalent if for any $x\in M$, there exists an open set $V$ containing $x$ such that $d,d'$ restricted to $V$ are equivalent.
\begin{lem}\label{lem-dist1}
Let $d$ and $d'$ be two Riemannian distances on the same open manifold $M$. Then $d$ and $d'$ is locally equivalent.
\end{lem}
The proof is classic and we skip it here.

\begin{proof}[Proof of Lemma \ref{lem-dist-local}]
Applying Lemma \ref{lem-dist1} to $d_1,d_3$, we obtain Lemma \ref{lem-dist-local} by noticing $d_2$ and $d_3$ are equivalent.
\end{proof}

We will upgrade Lemma \ref{lem-dist-local} to a version with base  point $eM$ replaced by any $gM$.
We first obtain an expanding rate estimate of the action of $G$ on $G/M$ with respect to the distance $d_2$.

\begin{defin}
For $x\in X$, let 
\begin{equation}\label{equ-cx}
    C_x=8C_2C_1\exp(C_0d_X(o,x)).
\end{equation}
Fix 
\begin{equation}
    \epsilon_0>0
\end{equation}
such that $O$ contains both balls centered at $eM$ of radius $\epsilon_0$ with respect to $d_1,d_2$ repsectively.
\end{defin}

\begin{lem}\label{lem-dist-compare}
For $x\in X$ and $z_1,z_2\in G/M$ with $x=\pi(z_1)$, if $d_2(z_1,z_2)<\epsilon_0/C_x$ or $d_1(z_1,z_2)<\epsilon_0$, then
\[d_1(z_1,z_2)\leq C_xd_2(z_1,z_2)/4. \]
\end{lem}

\begin{proof}

We take $h_x$ such that $h_x^{-1}z_1=eM$.
Then we have either 
$$d_2(h_x^{-1}z_1,h_x^{-1}z_2)\leq C_xd_2(z_1,z_2)<\epsilon_0,$$
(due to Lemma \ref{lem-d2})
or $d_1(h_x^{-1}z_1,h_x^{-1}z_2)=d_1(z_1,z_2)<\epsilon_0$. Due to the choice of $\epsilon_0$, we can apply Lemma \ref{lem-dist-local} and \ref{lem-d2} to obtain
\[d_1(z_1,z_2)=d_1(h_x^{-1}z_1,h_x^{-1}z_2)\leq C_2d_2(h_x^{-1}z_1,h_x^{-1}z_2)\leq C_xd_2(z_1,z_2)/4. \]
The proof is complete.
\end{proof}

\subsection{Corridors of maximal flats}
Recall from Definition \ref{defin-max-flat}, for every point $y\in X$ and every $\xi \in \mathcal{F}$, we denote by $\xi_y^\perp \in\calF$ the opposite element such that $y \in \big(\xi \xi_y^\perp \big)_X$. 
\begin{defin}\label{defin_corridors}
Let $x \in X$ and $r>0$.
We denote by $\cal{F}^{(2)}(x,r)$ the open \emph{corridor of maximal flats} at distance $r$ of $x$  
\begin{equation}\label{eq_corridor}
\cal{F}^{(2)} (x,r) := \lbrace (\xi,\eta) \in \mathcal{F}^{(2)} \; \vert \; \dd_X (x, (\xi \eta)_X ) <r \rbrace. 
\end{equation}
We denote by $\widetilde{ \cal{F}^{(2)}} (x,r)$ the set of Weyl chambers based in $B_X(x,r)$
\begin{equation}\label{eq_based-wc}
\widetilde{ \cal{F}^{(2)}} (x,r) := \Big\lbrace \big(\xi,\xi_y^\perp , \beta_{\xi} (o,y) \big) \in \mathcal{F}^{(2)}\times \frak{a} \; \Big\vert \; y \in B_X(x,r) \Big\rbrace. 
\end{equation}
\end{defin}

By \eqref{equ_hopfmap}, we obtain

\begin{fait}\label{prop-corridor}
For all $x \in X$ and $r>0$, $\widetilde{ \cal{F}^{(2)}} (x,r)$ is the preimage of $B_X(x,r)$ by the projection $G/M \rightarrow G/K$. 
\end{fait}

\begin{lem}\label{lem-compare}
Let $x \in X$ and $\min\{\frac{\epsilon_0}{2},\frac{\log 2}{C_0} \}>r>0$.
Then for every $\varepsilon \in (0, C_x^{-1}r)$, all $(\xi^+,\xi^-) \in \cal{F}^{(2)}(x,r)$,
$$ B(\xi^+,\varepsilon) \times B(\xi^-,\varepsilon) \subset \cal{F}^{(2)}(x,2r).$$
\end{lem}

\begin{proof}
We can find a point $z\in(\xi^+,\xi^-,\frak a)$ (maximal flat of $\xi^+,\xi^-$) such that $d_X(\pi(z),x)<r$. For $(\xi,\eta)\in B(\xi^+,\epsilon)\times B(\xi^-,\epsilon)$, we can find $z'\in(\xi,\eta,\frak a)$ with the same $\frak a$ coordinate as $z$. Then
\[d_2(z,z')=d(\xi,\xi^+)+d(\eta,\xi^-)<2\epsilon<\epsilon_0/C_x. \]
We can apply Lemma \ref{lem-dist-compare} to $z,z'$ and we obtain
\[d_X(\pi(z),\pi(z'))\leq d_1(z,z')\leq C_{\pi(z)} d_2(z,z')/4. \]
We have
\[C_{\pi(z)}=8C_2C_1e^{C_0d_X(o,\pi(z))}\leq C_xe^{C_0r}\leq 2C_x. \]
Therefore
\[d_X(x,\pi(z'))\leq d_X(x,\pi(z))+d_X(\pi(z),\pi(z'))<r+2C_x(2\epsilon)/4\leq 2r. \]
Hence $(\xi,\eta)\in \calF^{(2)}(x,2r)$.
\end{proof}

%

\begin{lem}\label{lem-cartan-lox}
Let $g\in G$ and $x\in X$.
Assume there is a transverse pair $(\xi^+,\xi^-) \in \cal{F}^{(2)}$ of fixed points for the action of $g$ on $\cal{F}$.
Then there exists $w$ in the Weyl group $\calW$ such that
$$ \Vert w(\lambda(g)) - \underline{a}_x(g) \Vert \leq 2 \dd_X (x, (\xi^+ \xi^-)_X ) .$$
\end{lem}
\begin{proof} 
For every transverse pair $(\xi^+,\xi^-)$, there exists, up to right multiplication by elements of $AM$, an $h\in G$ such that $h(\eta_0,\zeta_0)=(\xi^+,\xi^-)$. 

By assumption, $\xi^+$ and $\xi^-$ are fixed by $g$, i.e. $gh \in hAM$.
By Cartan decomposition, there exists $w\in\calW$ such that for every $p\in h AM o$, we have $\underline{a}_{p}(g)=w(\lambda(g))$.

Since $hAM o=hAo$, which is equal to the flat $(\xi^+\xi^-)_X$. 
It then follows from Lemma \ref{lem-cartan-points} that
for every $p \in (\xi^+ \xi^-)_X$
$$ \Vert w(\lambda(g)) - \underline{a}_x(g) \Vert=\|\underline{a}_p(g)-\underline{a}_x(g) \| \leq 2 \dd_X(x,p). $$
Taking the infimum over the points in the flat $(\xi^+\xi^-)_X$ yields the upper bound.
\end{proof}

\subsection{The configuration}
Recall that for all $x\in X$, we defined the constant
$C_x=8C_2C_1 e^{C_0\dd_X(o,x)}.$
\begin{defin}\label{defin_t0}
Denote by $r_0$ the unique zero in $(0,1)$ of the real valued function $r \mapsto -\log r -\max\{C_3,2\}r$.
For all $\varepsilon>0$ and $x \in X$ we define some function
$$t_0(x,\varepsilon) \gg 2\log C_x-2 \log(\varepsilon),$$
where the constant underlying $\gg$ is the same as in Lemma \ref{lem-action-aplus}.
\end{defin}

\begin{prop}\label{prop-lemroblin}
For all $x\in X$ and $r\in (0,r_0)$ and $\varepsilon\in (0,\min\{C_x^{-1}r,\epsilon_0\} )$, 
 every $\gamma \in G$ satisfying the following conditions is loxodromic.
\begin{itemize}
\item[(i)] $\underline{a}_x(\gamma) \in \frak{a}^{++}$ and $\dd(\underline{a}_x(\gamma), \partial \frak{a}^+ ) \geq  t_0(x,\epsilon)$,
\item[(ii)] $(\gamma_x^+,\gamma_x^-) \in \mathcal{F}^{(2)}$ are transverse and $\dd_X \big(x, (\gamma_x^+ \gamma_x^-)_X  \big) <r$.
\end{itemize}
Furthermore, its attracting and repelling point satisfy $\gamma^\pm \in B(\gamma_x^\pm, \varepsilon)$.
\end{prop}

\begin{proof}
There exist $k_{\gamma_x^+}, l_{\gamma_x^-} \in h_x K$ (as $h$ and $h'k_\iota$ in Definition \ref{defin-cartan-regular}), defined up to right multiplication by elements of $M$ and independent of the choice of representative $h_x \in G$ such that
$\gamma = k_{\gamma_x^+} e^{\underline{a}_x(\gamma)} l_{\gamma_x^-}^{-1} .$
Apply Lemma \ref{lem-action-aplus}, to the element $h_x^{-1}\gamma h_x=h_x^{-1}k_{\gamma_x^+} e^{\underline{a}_x(\gamma)} (h_x^{-1}l_{\gamma_x^-})^{-1} \in K A^{++}K$,
$$ h_x^{-1}\gamma h_x\; \cal{V}_{C_x^{-1}\varepsilon}( h_x^{-1}\gamma_x^-)^\complement \subset B(h_x^{-1}\gamma_x^+,C_x^{-1}\varepsilon).$$
We multiply by $h_x$ on the left 
$ \gamma h_x \cal{V}_{C_x^{-1}\varepsilon}(h_x^{-1}\gamma_x^- )^\complement \subset h_x B(h_x^{-1}\gamma_x^+,C_x^{-1}\varepsilon).$
Using the properties of $C_x>0$ (Lemma \ref{lem-actiong}), we deduce the following inclusions 
\begin{itemize}
\item[•] $h_x B(h_x^{-1}\gamma_x^+,C_x^{-1}\varepsilon) \subset B(\gamma_x^+, \varepsilon)$,
\item[•] $\cal{V}_{\varepsilon}(\gamma_x^-)^\complement \subset h_x \cal{V}_{C_x^{-1}\varepsilon}(h_x^{-1}\gamma_x^-)^\complement$.
\end{itemize}
Hence
$ \gamma \cal{V}_{\varepsilon}(\gamma_x^-)^\complement \subset B(\gamma_x^+,\varepsilon).$
Recall that $\iota$ is the opposition involution and $k_\iota \in N_K(A)$ such that $\iota=-Ad(k_\iota)$, then
$$\gamma^{-1}=l_{\gamma_x^-}k_\iota \; e^{\iota \underline{a}_x(\gamma)}\; (k_{\gamma_x^+}k_\iota)^{-1}.$$
Since $\iota \underline{a}_x(g)$ is at distance at most $t_0$ from $\partial \frak{a}^+$ and $(\gamma^{-1})_x^\pm = \gamma_x^\mp$, we deduce that
$ \gamma^{-1} \cal{V}_\varepsilon(\gamma_x^+)^\complement \subset B(\gamma_x^-,\varepsilon).$

Due to $d_X(o,((h_x^{-1}\gamma_x^+)(h_x^{-1}\gamma_x^-))_X)=d_X(x,(\gamma_x^+\gamma_x^-)_X)<r$, by Lemma \ref{lem-xietao} and Definition \ref{defin_gromov}, we obtain
\[ \delta(h_x^{-1}\gamma_x^+,h_x^{-1}\gamma_x^-)\geq e^{-C_3r}. \]
Then by Lemma \ref{lem-actiong}, we have
\[ \delta(\gamma_x^+,\gamma_x^-)\geq C_x^{-1}\delta(h_x^{-1}\gamma_x^+,h_x^{-1}\gamma_x^-)\geq C_x^{-1}e^{-C_3r}. \]
Due to the choice of $\epsilon,r$, we have $C_x^{-1}e^{-C_3r}>2\epsilon$. Hence we have $B(\gamma_x^\pm, \varepsilon) \subset \cal{V}_\varepsilon(\gamma_x^\mp)^\complement$. Then we deduce that $\gamma$ (resp. $\gamma^{-1}$) has an attracting fixed point $\xi^+\in B(\gamma_x^+,\varepsilon)$ (resp. $\xi^-\in B(\gamma_x^-,\varepsilon)$).

Since $\gamma$ admits a fixed maximal flat $(\xi^+\xi^-)_X$, we apply Lemma \ref{lem-cartan-lox},
$$ \Vert w(\lambda(\gamma))-\underline{a}_x(\gamma) \Vert \leq 2 \dd_X(x, (\xi^+\xi^-)_X ),$$
for some $w$ in the Weyl group.
By hypothesis $\varepsilon <C_x^{-1}r$, Lemma \ref{lem-compare} implies that 
$ B(\gamma_x^+,\varepsilon) \times B(\gamma_x^-,\varepsilon) \subset \cal{F}^{(2)}(x,2r)$.
Hence $w(\lambda(\gamma)) \in B( \underline{a}_x(\gamma), 4r)$.
Using that $r<r_0$ and $\varepsilon < C_x^{-1}r$, we get a lower bound
$t_0(x,\varepsilon) >  - 2 \log r > 4r.$
We deduce that $B(\underline{a}_x(\gamma),4r) \subset \frak{a}^{++}$, therefore $w=id$ and $\gamma$ is loxodromic.

Finally, because the bassin of attraction of $\gamma^+$ (resp. $\gamma^-$) is a dense open set of $\calF$, there are points in $B(\gamma_x^+,\varepsilon)$ (resp. $B(\gamma_x^-,\varepsilon)$) that $\gamma$ (resp. $\gamma^{-1}$) will attract to $\gamma^+$ (resp. $\gamma^-$). 
Since $\calF$ is Hausdorff for $\dd$, we deduce that $\gamma^+=\xi^+$ (resp. $\gamma^-=\xi^-$). 
\end{proof}

\section{Loxodromic elements and periodic tori}

In this part, we give a relation between conjugacy classes of loxodromic elements and periodic tori.
We denote in brackets the $\Gamma$-conjugacy classes of elements in $\Gamma$. Set $\calG_{lox}$ the set of $\Gamma$-conjugacy classes of loxodromic elements and
\begin{equation}\label{defin_Glox_t}
\cal{G}_{lox}(t) := \lbrace [\gamma] \in [\Gamma^{lox}] \; \vert \; \lambda(\gamma) \in B_{\frak{a}} (0,t)  \rbrace.
\end{equation}
For every loxodromic element $\gamma \in \Gamma^{lox}$ denote by  $\calL_\gamma$ the measure of $G/M$ supported on the $A$-orbit of Hopf coordinates $( \gamma^+ , \gamma^-; \frak a)$ such that its disintegration in Hopf coordinates is given by
\begin{equation}\label{defin_mes-tore}
\calL_\gamma := D_{\gamma^+} \otimes D_{\gamma^-} \otimes Leb_{\frak{a}},
\end{equation}
where $D_{\gamma^\pm}$ is the Dirac measure at $\gamma^\pm$. 
Note that the quotient in $\Gamma\backslash G/M$ of the $A$-orbit $( \gamma^+ , \gamma^-; \frak a)$ only depends on $[\gamma]$.
Denote by $F_{[\gamma]}$ the quotient of this $A$-orbit in $\Gamma\backslash G/M$. 
By \eqref{eq_gg+}, that is $\gamma(\gamma^+,\gamma^-,Y)=(\gamma^+,\gamma^-,Y+\lambda(\gamma))$ for every $Y\in \frak{a}$, we obtain $\lambda(\gamma)\in\Lambda(F_{[\gamma]})$. If we take $g_\gamma$ an element such that $(\gamma^+,\gamma^-,0)=g(\eta_0,\zeta_0,0)$, then the formula also implies $g_\gamma^{-1}\gamma g_\gamma\in \exp(\lambda(\gamma))M$. With this $g_\gamma$, the orbit $F$ can be written as $F_{[\gamma]}=\Gamma g_\gamma AM$.

In this subsection, we always suppose that $\Gamma <G$ is a cocompact lattice of $G$. We have a lemma by Selberg

\begin{lem}\label{lem-selberg}
Let $\Gamma$ be a cocompact lattice.
Let $F$ be a right $A$-orbit in $\Gamma\backslash G/M$. If $\Lambda(F)\cap\frak a^{++}\neq\emptyset$, then $F$ is a compact periodic $A$-orbit.
\end{lem}
\begin{proof}
We can write $F=\Gamma g AM$. For $Y\in \Lambda(F)\cap\frak a^{++}$, by $\Gamma gM=\Gamma g\exp(Y)M$, we know there exists a loxodromic element $\gamma\in\Gamma$ such that $\gamma g=g\exp(Y)m_Y$. By Selberg's lemma in \cite{selberg} or \cite{pr}, we know that $\Gamma_\gamma\backslash G_\gamma$ is compact with $G_\gamma$ and $\Gamma_\gamma$ the centralizer of $\gamma$ in $G$ and $\Gamma$, respectively. Since $\gamma$ is loxodromic, so $G_\gamma$ is a conjugation of a maximal torus. Now $gAMg^{-1}$ commutes with $\gamma$, so $G_\gamma=gAMg^{-1}$. Then $\Gamma_\gamma\backslash G_\gamma=(\Gamma\cap G_\gamma)\backslash G_\gamma$ compact implies that $\Gamma gAM=\Gamma G_\gamma g$ is compact in $\Gamma\backslash G$. So $F$ is compact in $\Gamma\backslash G/M$.
\end{proof}


Let $\calG(A):=\{(Y,F)|\ F\in C(A),\ Y\in\Lambda(F)\cap \frak a^{++} \}$.

\begin{lem}\label{lem-gaglox}
Let $\Gamma$ be a cocompact lattice.
 	If the action of $\Gamma$ on $G/M$ is free, then we have well defined maps
 	\[ \Psi: \calG_{lox} \rightarrow \calG(A),\ [\gamma]\mapsto (\lambda(\gamma),F_{[\gamma]})  \]
 	and
 	\[\Phi: \calG(A)\rightarrow \calG_{lox},\ (Y,F)\mapsto [\gamma_Y]. \]
 	We also have $\Psi\circ\Phi=Id_{\calG(A)}$ and $\Phi\circ\Psi=Id_{\calG_{lox}}$.
\end{lem}

\begin{proof}
	For a compact periodic $A$-orbit $F$, we can write it as $\Gamma g AM$ with some $g\in G$. For $Y\in\Lambda(F)\cap\frak a^{++}$, there exists a $\gamma_Y\in\Gamma$ such that $\gamma_Yg =g\exp(Y)m_Y$ for some $m_Y\in M$. This $\gamma_Y$ is unique. Otherwise, we have $\gamma_Y'g=g\exp(Y)m_Y'$ with $m_Y'\neq m_Y$, then $\gamma_Y^{-1}\gamma_Y'=gm_Y^{-1}m_Y'g^{-1}$. This element $\gamma_Y^{-1}\gamma_Y'$ in $\Gamma$ fixes $gM$ in $G/M$ and is not identity, which contradicts that $\Gamma$ acts on $G/M$ freely.
	
	This $g$ is unique up to left multiplication by $\Gamma$ and right multiplication by $AM$. This defines a $\Gamma$-conjuage class $[\gamma_Y]$ in $\calG_{lox}$, characterised by $g^{-1}\gamma_Y g\in \exp(Y)M$. So the map $\Phi$ is well-defined.
	
	For $[\gamma]$ in $\calG_{lox}$, we have already associated it to a unique periodic orbit $F_{[\gamma]}$, that is $F_{[\gamma]}=\Gamma g_\gamma AM$ with $g_\gamma$ such that $g_\gamma^{-1}\gamma g_\gamma\in \exp(\lambda(\gamma))M$. Due to $\lambda(\gamma)\in F_{[\gamma]}$, by Lemma \ref{lem-selberg}, this orbit $F_{[\gamma]}$ is a compact periodic $A$-orbit. 
	
	For $\Psi\circ\Phi$, due to $g^{-1}\gamma_Y g\in \exp(Y)M$, we know that we can take $g_{\gamma_Y}=g$ and then $\Psi\circ\Phi(Y,F)=\Psi([\gamma_Y])=(Y,F)$.
	
	For $\Phi\circ\Psi$, from $g_\gamma^{-1}\gamma g_\gamma\in \exp(\lambda(\gamma))M$, we know that $\Phi\circ\Psi([\gamma])=\Phi(\lambda(\gamma),F_{[\gamma]})=\gamma$.
	\end{proof}

\section{Equidistribution of flats}

For every loxodromic element $\gamma \in \Gamma^{lox}$, denote by $L_{\gamma}$ the quotient measure on $\Gamma \backslash G/M$ of $\calL_\gamma$ (Cf. \eqref{defin_mes-tore}).
Note that $L_{[\gamma]}$ is supported on $F_{[\gamma]}$ and is equal to the measure $L_{F_{[\gamma]}}$ given in the introduction.
It is also given by the following construction: we push on $F_{[\gamma]}$, the restriction of $Leb_{\frak{a}}$ to any fundamental domain in $\frak{a}$ of the periods $\Lambda(F_{[\gamma]})$, by right $A$-action of the exponential of such a fundamental domain, starting from any base point on $F_{[\gamma]}$.
The construction is independent of both the choice of the fundamental domain and the base point on $F_{[\gamma]}$.

By Lemma \ref{lem-gaglox}, there is a bijection between $\calG_{lox}$ and $\calG(A)$. 
By summing over the compact periodic orbits $F\in C(A)$ first, then summing over $Y \in \Lambda(F) \cap B_{\frak a}^{++}(0,t)$, we deduce that
\begin{equation}
 \frac{1}{\vol(D_t)}\sum_{[\gamma] \in \calG_{lox}(t)}L_{[\gamma]}=\frac{1}{\vol(D_t)}\sum_{F\in C(A)}|\Lambda(F)\cap B_{\frak a}^{++}(0,t) |L_F,
 \end{equation}
 the measure on the right hand side is exactly the measure in the Theorem \ref{thm-introequid}. This formula is also a higher rank analogue of the first part of \eqref{equ-prime}.
Set
\[\calM^t_\Gamma:=\frac{\vol(\Gamma\backslash G)}{ \vol( D_t)}  \sum_{[\gamma] \in \calG_{lox}(t)} 
L_{[\gamma]}.\]

Let $m_{G/M}$ be the Haar measure on $G/M$, given by $\nu\otimes Leb_{\frak a}$ from Proposition \ref{prop-disintegration}.
Let $m_{\Gamma \backslash G/M}$ be the quotient measure on  $\Gamma\backslash G/M$. 
The main theorem \ref{thm-introequid} is equivalent to the following one if $\Gamma$ is torsion free or if it acts on $G/M$ freely. 

\begin{theorem}\label{thm_cocompact-revet}
Let $\Gamma<G$ be a cocompact irreducible lattice which acts freely on $G/M$. Then there exists $u>0$ such that for any Lipschitz function $f$ on $\Gamma\backslash G/M$, as $t\rightarrow \infty$
\begin{equation}\label{eq_equ-plats}
\int f\ \dd \calM^t_\Gamma
=\int f\ \dd m_{\Gamma\backslash G/M}+O(e^{-u t}|f|_{Lip}),
\end{equation}
where the Lipschitz norm is with respect to the Riemannian distance $d_1$ on $\Gamma\backslash G/M$.
\end{theorem}

\begin{rem}
The constant $C_G$ equals $\|m_{\Gamma\backslash G/M} \|/\vol(\Gamma\backslash G)$, which comes from the choice of $m_{G/M}=\nu\otimes Leb_{\frak a}$ and only depends on $G$. 

We can separate a Lipschitz function as the sum of its positive part and its negative part. So it is sufficient to prove Theorem \ref{thm_cocompact-revet} for non negative Lipschitz functions.
\end{rem}

We are going to prove Theorem \ref{thm_cocompact-revet} in this section.
Before starting the argument, we fix the parameters which will be used later. 
They come from Proposition \ref{prop-lemroblin}.
Choose $u_1>0$ small than $\min\{\epsilon_G,1\}/10$, where $\epsilon_G$ is the constant from Lemma \ref{volume_reste}. Set
\begin{equation}\label{equ-parameter0}
\varepsilon:= e^{-u_1 t} \text{ and }t_1 := 3 u_1 t. 
\end{equation}
Consider the decay rate function $u \mapsto \kappa(u) >0$ satisfying Lemma \ref{volume_reste} and the decay coefficient $\kappa>0$ given in Theorem \ref{theo-gorodnik-nevo}. Set 
\begin{equation}\label{equ-parameter}
    u_2 := \frac{1}{ 2\dim (G/AM)} 
    \min \{ \delta_0 \kappa(6 u_1), \delta_0 \kappa, u_1 \}\text{ and } r:= e^{-u_2 t}.
\end{equation}


In this part we use $Lip_2$ to denote Lipschitz norm with respect to the product distance $d_2$ on $G/M$ or the product distance on $\calF^{(2)}$, according to which space the function lives on.

We lift everything to $G/M$ and prove a local version on $G/M$ in Section \ref{sec-loccor} and \ref{sec-corweyl}. Then in Section \ref{sec-global}, we use the partition of unity to obtain a global version (Theorem \ref{thm_cocompact-revet}) on $\Gamma\backslash G/M$.


\subsection{Local convergence on corridors}\label{sec-loccor}

Recall the notation $\underline{a}_x(\gamma):= \dd_{\underline{a}}(x,\gamma x)= \underline{a}(h_x^{-1} \gamma h_x )$.
For every $\gamma \in \Gamma$ such that $\underline{a}_x(\gamma)\in \frak{a}^{++}$, 
the geometric Weyl chamber based on $x$ containing $\gamma x$ (resp. $\gamma^{-1} x$)
determines $\gamma_x^+ \in \mathcal{F}$ (resp. $\gamma_x^-$). 

For $x\in X$ and $t>0$, we define the following measures on $\calF\times\calF$:

\begin{equation}\label{eq_nu1}
\nu_{x,1}^t:= \frac{\vol(\Gamma\backslash G)}{\vol(D_t)}\sum_{ \gamma \in \Gamma \cap D_t^{reg}(x)  } D_{\gamma_x^+} \otimes D_{\gamma_x^-},
\end{equation}
%

\begin{equation}\label{eq_nu2}
\nu_{x,2}^t:= \frac{\vol(\Gamma\backslash G)}{\vol(D_t)} \sum_{ \gamma \in \Gamma^{lox} \cap D_t^{reg}(x) } D_{\gamma^+} \otimes D_{\gamma^-}.
\end{equation}

Recall that $(\mu_x)_{x \in X}$ denotes the Patterson-Sullivan density given in Proposition \ref{prop-PS-BMS} and $\nu$ is the associated conformal measure on $\mathcal{F}^{(2)}$. 
Let $Lip^+_c( \mathcal{F}^{(2)}(x,r))$ be the space of positive compactly supported Lipschitz functions on $\calF^{(2)}(x,r)$.

\begin{lem}\label{lem-corridor-1GN}
Let $\Gamma$ be an irreducible lattice in $G$. Fix $x\in X$. Then for every test function $\psi \in Lip^+_c( \mathcal{F}^{(2)}(x,r))$ for every $t>C_4 d_X(o,x)$, there exists a function $E(t,\psi,x)$ such that 
\begin{equation}\label{eq-lem-corridor-1GN}
e^{-C_3r} \int \psi  \dd \nu - E(t,\psi, x) 
\leq \int \psi  \dd \nu_{x,1}^t
\leq \int \psi \dd \nu + E(t,\psi,x)
\end{equation}
where $E(x,\psi,t)=O( C_x Lip(\psi) \vol(D_t)^{-\kappa})$ when $t \rightarrow \infty$.
\end{lem}
\begin{proof}
By Theorem \ref{theo-gorodnik-nevo}, we obtain the main term with the measure $\mu_x\otimes\mu_x$. Since $(\xi,\eta)\in\calF^{(2)}(x,r)$, so by Lemma \ref{lem-xietao}, we obtain 
\[ 1 \leq  f_x(\xi,\eta)^{-1} \leq e^{C_3r}. \]
Using the relation $d\nu(\xi,\eta)=\frac{d\mu_x(\xi)d\mu_x(\eta)}{f_x(\xi,\eta)}$, we deduce that $\int \psi \dd \mu_x \otimes \mu_x  \leq  \int \psi \dd \nu \leq  e^{C_3 r} \int \psi \dd \mu_x \otimes \mu_x.$
Hence the Lemma.
\end{proof}

%
%
%
%
%
\begin{lem}\label{lem-corridor-3RGN}
Let $\Gamma$ be a lattice in $G$.
Fix $x\in X$, for every $t \geq \frac{2 \log C_x}{u_1}$, for every test function $\psi \in Lip^+_c( \mathcal{F}^{(2)}(x,r))$,

$$
\bigg\vert \int \psi \dd \nu_{x,2}^t - \int \psi \dd \nu_{x,1}^t  \bigg\vert \leq 
\varepsilon Lip_2(\psi) \frac{\vert \Gamma \cap D_t(x) \vert \vol(\Gamma \backslash G)}{\vol(D_t)}
+ 3\Vert \psi \Vert_{\infty} \frac{ \vert \Gamma \cap D_t^{t_1}(x) \vert \vol(\Gamma \backslash G) }{\vol(D_t)},$$
where $\varepsilon$ and $t_1$ are given in \eqref{equ-parameter0}. 
\end{lem}

\begin{proof}
We split the difference between $\frac{\vol(D_t)}{\vol(\Gamma\backslash G)} \int \psi \dd \nu_{x,1}^t$ and $\frac{\vol(D_t)}{\vol(\Gamma\backslash G)} \int \psi \dd \nu_{x,2}^t$,
\begin{align*}
    \sum_{\gamma \in \Gamma \cap D_t^{reg}(x)} \psi (\gamma_x^+,\gamma_x^-) - \sum_{\gamma \in \Gamma^{lox} \cap D_t^{reg}(x)} \psi(\gamma^+,\gamma^-) 
    &=\sum_{\gamma \in \Gamma\cap D_t^{reg}(x)} \psi (\gamma_x^+,\gamma_x^-) - \sum_{\gamma \in \Gamma^{lox} \cap D_t^{reg}(x)} \psi(\gamma_x^+,\gamma_x^-)  \\
    & \hspace{2cm} + \sum_{\gamma \in \Gamma^{lox} \cap D_t^{reg}(x)} \psi(\gamma_x^+,\gamma_x^-) - \psi(\gamma^+,\gamma^-). 
\end{align*}

For the first term on the right hand side, note that $\Gamma^{lox} \subset \Gamma$, hence
$$  \sum_{\gamma \in \Gamma\cap D_t^{reg}(x)} \psi (\gamma_x^+,\gamma_x^-) - \sum_{\gamma \in \Gamma^{lox} \cap D_t^{reg}(x)} \psi(\gamma_x^+,\gamma_x^-) = \sum_{\gamma \in  (\Gamma \setminus \Gamma^{lox}) \cap D_t^{reg}(x)}\psi(\gamma_x^+,\gamma_x^-) .$$
Note that $t\geq t_1=3u_1 t>0$ since $u_1 \leq 1/10$, hence we have the following inclusion 
$$D_t^{reg}(x) \subset D_t^{t_1}(x) \sqcup \big( D_t(x) \setminus D_t^{t_1}(x)\big).$$ 
Using that $t \geq \frac{2 \log C_x}{u_1}$, we deduce that $t_1=3 u_1 t \geq t_0:= 2 \log C_x - 2\log \varepsilon = 2 \log C_x + 2 u_1 t$. 
Apply Proposition \ref{prop-lemroblin} to every 
every $\gamma \in D_t(x) \setminus D_t^{t_1}(x)$ such that $(\gamma_x^+,\gamma_x^-)\in\calF^{(2)}(x,r) $.
Any such element is loxodromic i.e. $D_t(x) \setminus D_t^{t_1}(x) \subset G^{lox}$.
Hence $\Gamma \cap \big( D_t(x) \setminus D_t^{t_1}(x) \big) \subset \Gamma^{lox}$ is a set of loxodromic elements. 
So the non-loxodromic must lie in $ \big(\Gamma \setminus \Gamma^{lox}\big) \cap D_t^{reg}(x) \subset D_t^{t_1}(x)$.
We deduce the following upper bound.
\begin{equation}\label{eq-lem-corridor-ineq1}
\bigg\vert \sum_{\gamma \in  (\Gamma \setminus \Gamma^{lox}) \cap D_t^{reg}(x)}\psi(\gamma_x^+,\gamma_x^-) \bigg\vert \leq  \Vert \psi \Vert_\infty \vert \Gamma \cap D_t^{t_1}(x) \vert .
\end{equation}

For the lower term, we split the sum over the partition $\Gamma^{lox}\cap (D_t(x) \setminus D_t^{t_1}(x))$ and $\Gamma^{lox} \cap D_t^{t_1}(x)$.
\begin{align*}
    \sum_{\gamma \in \Gamma^{lox} \cap D_t^{reg}(x)} \psi(\gamma_x^+,\gamma_x^-) -\psi(\gamma^+,\gamma^-)
    &= \underset{\gamma \in D_t(x) \setminus D_t^{t_1}(x)}{\sum_{\gamma \in \Gamma^{lox}}} \psi(\gamma_x^+,\gamma_x^-) -\psi(\gamma^+,\gamma^-) \\
    & \hspace*{1 cm} + \sum_{\gamma \in \Gamma^{lox} \cap D_t^{t_1}(x)\cap D_t^{reg}(x)} \psi(\gamma_x^+,\gamma_x^-) -\psi(\gamma^+,\gamma^-).
\end{align*}
We bound the lower term. 
\begin{equation}\label{eq-lem-corridor-ineq2}
\bigg\vert \sum_{\gamma \in \Gamma^{lox} \cap D_t^{t_1}(x)\cap D_t^{reg}(x)} \psi(\gamma_x^+,\gamma_x^-) -\psi(\gamma^+,\gamma^-) \bigg\vert \leq  2 \Vert \psi \Vert_\infty \vert \Gamma \cap D_t^{t_1}(x) \vert.
\end{equation}
By Proposition \ref{prop-lemroblin}, the elements $\gamma \in \Gamma \cap (D_t(x) \setminus D_t^{t_1}(x))$ with $(\gamma_x^+,\gamma_x^-)\in\calF^{(2)}(x,r) $ are loxodromic and their attractive and repelling points are at distance at most $\varepsilon$ of respectively $\gamma_x^\pm$.
Using that $\psi$ is Lipschitz and supported on $\calF^{(2)}(x,r)$, we bound above the last term.
\begin{equation}\label{eq-lem-corridor-ineq3}
 \bigg\vert \sum_{\gamma \in \Gamma \cap (D_t(x) \setminus D_t^{t_1}(x))} \psi(\gamma_x^+,\gamma_x^-) -\psi(\gamma^+,\gamma^-) \bigg\vert \leq \varepsilon Lip_2(\psi) \; \vert \Gamma \cap D_t(x)\vert .
\end{equation}
Finally, we use the triangular inequality, regroup the terms \eqref{eq-lem-corridor-ineq1}, \eqref{eq-lem-corridor-ineq2} and \eqref{eq-lem-corridor-ineq3}, then multiply everything by $\frac{\vol(\Gamma \backslash G)}{\vol(D_t)}$ to obtain the main upper bound.
\end{proof}

\subsection{From corridors to Weyl chambers}\label{sec-corweyl}

\begin{lem}\label{lem-couloirs-lipschitz}
Let $\widetilde{\psi} \in Lip^+_c( \widetilde{ \cal{F}^{(2)}} (x,r) )$ be a compactly supported nonnegative, Lipschitz function and set 
$$ \psi := \int_{\frak{a}} \widetilde{ \psi} (.,. \; ; \; v)  \dd v.$$
Then $\psi \in Lip^+_c( \cal{F}^{(2)} (x,r))$ and the following norm bounds hold:
\begin{itemize}
\item[(a)] $Lip_2(\psi) \leq 2(2r)^{\dim \frak{a}} Lip_2 (\widetilde{\psi}) $.
\item[(b)] $\Vert \psi \Vert_\infty \leq (2r)^{\dim \frak{a}} \Vert \widetilde{\psi} \Vert_\infty.$
\end{itemize}
\end{lem}

For $x\in X$ and $t>0$, we define the following measure on $\calF^{(2)}\times\frak a$ by

\begin{equation}\label{eq_M2}
\cal{M}_{x,2}^t:= \frac{\vol(\Gamma\backslash G)}{\vol(D_t)}\sum_{ \gamma \in \Gamma^{lox}\cap D_t^{reg}(x) } \calL_\gamma = \nu_{x,2}^t \otimes Leb_{\frak{a}}. 
\end{equation}

\begin{lem}\label{lem_M2-haar}
Let $\Gamma$ be an irreducible lattice in $G$.
Fix $x\in X$, for every $t \geq \max\{\frac{2 \log C_x}{u_1}, C_4d_X(o,x) \}$, for every test function $\widetilde{ \psi }  \in Lip^+_c( \widetilde{\cal{F}}^{(2)}(x,r))$,  

\begin{align*}
\bigg\vert \int \widetilde{\psi} \;  \dd \cal{M}_{x,2}^t -
 \int & \widetilde{\psi} \; \dd m_{G/M} \bigg\vert \leq  \quad  C_3r \int \widetilde{\psi} \dd m_{G/M} \quad + \\
&  (2r)^{\dim \frak{a}} \bigg( E(t,\widetilde{\psi},x) + 2\varepsilon Lip_2(\widetilde{\psi})  \frac{\vert \Gamma \cap D_t(x) \vert \vol(\Gamma \backslash G) }{\vol(D_t)} + 3 \Vert \widetilde{\psi} \Vert_\infty \frac{\vert \Gamma \cap D_t^{t_1}(x) \vert \vol(\Gamma \backslash G)}{\vol(D_t)}  \bigg), 
\end{align*}
where $E(x,\widetilde{ \psi},t)= O(C_x Lip(\widetilde{\psi}) \vol(D_t)^{-\kappa}  )$ as introduced in Lemma~\ref{lem-corridor-1GN} and $\varepsilon$, $t_1$ are given in \eqref{equ-parameter0}.
\end{lem}

\begin{proof}
We set $\psi (\xi^+,\xi^-) := \int_{\frak{a}} \widetilde{\psi}(\xi^+,\xi^- ;  v) \dd v.$
Using Fubini's theorem on the $\frak{a}$ coordinate and Proposition \ref{prop-disintegration} that $m_{G/M}=\nu \otimes Leb_{\frak a}$, we deduce that
\begin{align*}
 \int \widetilde{\psi} \dd \cal{M}_{x,2}^t = \int \psi \dd \nu_{x,2}^t \text{ and }  \int \widetilde{\psi} \dd m_{G/M} = \int \psi \dd \nu.
\end{align*}
We only need to bound $\int \psi \dd \nu_{x,2}^t - \int \psi \dd \nu$.
By definition of these measures,
$$
 \int \psi \dd \nu_{x,2}^t - \int \psi \dd \nu  = \quad \int \psi \dd \nu_{x,1}^t - \int \psi \dd \nu + \int \psi \dd \nu_{x,2}^t -  \int \psi \dd \nu_{x,1}^t .
$$
Using Lemma \ref{lem-corridor-3RGN} on the last term on the right, then Lemma \ref{lem-corridor-1GN}, the convexity inequality $e^{-r}-1 \geq -r$ and nonnegativity of $\psi$ to the other term, we deduce the following bound. 
\begin{align*}
\bigg\vert \int \psi \dd \nu_{x,2}^t - \int \psi \dd \nu \bigg \vert \leq 
C_3 r &\int \psi \dd \nu + E(t, \psi, x) \\
&+ \varepsilon Lip_2(\psi) \frac{\vert \Gamma \cap D_t(x) \vert \vol(\Gamma \backslash G)}{\vol(D_t)}
+ 3\Vert \psi \Vert_{\infty} \frac{ \vert \Gamma \cap D_t^{t_1}(x) \vert \vol(\Gamma \backslash G)}{\vol(D_t)}.
\end{align*} 
By Lemma \ref{lem-couloirs-lipschitz} (a) (b), the Lipschitz constants and norms between $\psi$ and $\widetilde{\psi}$ satisfy 
$ Lip_2 (\psi) \leq 2 (2r)^{\dim \frak{a}} Lip_2(\widetilde{\psi}) $
and $\Vert \psi \Vert_\infty \leq (2r)^{\dim \frak{a}} \Vert \widetilde{\psi} \Vert_\infty$.  
We deduce the domination $E(t,\psi, x)= (2r)^{\dim \frak{a}} O(Lip_2(\widetilde{\psi}) C_x \vol(D_t)^{-\kappa})$ and abusing notation we write 
$$E(t, \psi, x) = (2r)^{\dim \frak{a}} E(t, \widetilde{\psi},x). $$
Replacing the Lipschitz constants and norms in the upper bound by abuse of notation on $E(t,\psi,x)$ and lastly applying Fubini on the first term yields
\begin{align*}
\bigg\vert \int \psi \dd \nu_{x,2}^t - \int & \psi \dd \nu \bigg \vert \quad \leq \quad
C_3 r \int \widetilde{ \psi} \dd m_{G/M} \quad +\\
&(2r)^{\dim \frak{a}} \bigg( E(t,\widetilde{\psi},x) +  2\varepsilon Lip_2(\widetilde{\psi}) \frac{\vert \Gamma \cap D_t(x) \vert \vol(\Gamma \backslash G)}{\vol(D_t)}
+ 3\Vert \widetilde{\psi} \Vert_{\infty} \frac{ \vert \Gamma \cap D_t^{t_1}(x) \vert \vol(\Gamma \backslash G)}{\vol(D_t)} \bigg).
\end{align*} 
\end{proof}

From now on, to the end of this section, we suppose that $\Gamma$ is a \textit{cocompact irreducible lattice} in $G$ which acts freely on $G/M$. The measure in equidistribution is denoted by
\begin{equation}\label{eq_Mes}
\cal{M}^t:= \frac{\vol(\Gamma\backslash G)}{\vol(D_t)} \underset{ \Vert \lambda (\gamma) \Vert \leq t}{ \sum_{ \gamma \in \Gamma^{lox}}}\calL_\gamma. 
\end{equation}

\begin{lem}\label{lem-cocom-weyl}
There exists $C>0$. Fix $x\in X$, for every test function $\widetilde{\psi}\in Lip^+_c( \widetilde{\cal{F}}^{(2)}(x,r))$,
\begin{equation}\label{eq-lem-cocom-weyl}
(1-Cr) \int  \widetilde{\psi} \dd \cal{M}_{x,2}^{t-2r} \leq  \int  \widetilde{\psi} \dd \cal{M}^t   \leq (1+Cr) \int  \widetilde{\psi} \dd \cal{M}_{x,2}^{t+2r}+\|\widetilde{\psi}\|_\infty\frac{|\Gamma\cap D_{t+2r}^{2r}(x)|\vol(\Gamma \backslash G)}{\vol(D_t)}. 
\end{equation}
\end{lem}
\begin{proof}
By Lemma \ref{lem-cartan-lox}, for every loxodromic element $g \in G^{lox}$ such that $(g^+,g^-)\in \cal{F}^{(2)}(x,r)$ then 
$$ \Vert \lambda(g) - \underline{a}_x(g) \Vert \leq 2r .$$
Hence using triangular inequality we deduce the inclusions
\begin{align*}
\Gamma^{lox} \cap D_{t-2r}^{reg}(x)\cap \{\gamma|\ (\gamma^+,\gamma^-) &\in \cal{F}^{(2)}(x,r) \} \subset \\
&\big\lbrace \gamma \in \Gamma^{lox} \;\big\vert \; \Vert \lambda(\gamma) \Vert \leq t \; \text{and}\; (\gamma^+,\gamma^-) \in \cal{F}^{(2)}(x,r) \big\rbrace \\
& \hspace*{4cm} \subset (\Gamma^{lox} \cap D_{t+2r}^{reg}(x))\cup (\Gamma\cap D_{t+2r}^{2r}(x)) ,
\end{align*} 
here the set $\Gamma\cap D_{t+2r}^{2r}(x)$ is used to contain all the $\gamma$ in the middle set with $\underline{a}_x(\gamma)$ singular.
By integrating $\widetilde{\psi}$ over $\calL_\gamma$, summing and using that  $\widetilde{\psi}$ is supported on $\widetilde{\cal{F}}^{(2)}(x,r)$, we deduce
\begin{equation}\label{eq-lem-cocom-weyl-ineq1}
\frac{ \vol(D_{t-2r})}{\vol(\Gamma \backslash G)} \int \widetilde{\psi} \dd \cal{M}_{x,2}^{t-2r}
\leq \frac{ \vol(D_t)}{\vol(\Gamma \backslash G)} \int \widetilde{\psi} \dd \cal{M}^t
\leq \frac{ \vol(D_{t+2r})}{\vol(\Gamma \backslash G)} \int \widetilde{\psi} \dd \cal{M}_{x,2}^{t+2r}
+\|\widetilde{\psi}\|_\infty|\Gamma\cap D_{t+2r}^{2r}(x)|.
\end{equation}
Finally, we multiply by $\frac{\vol(\Gamma \backslash G)}{ \\vol(D_t)}$, apply the local Lipschitz property of $t \mapsto \log (\\vol(D_t))$ (Lemma \ref{lem-vollip}).
\end{proof}

\subsection{Proof of the equidistribution}\label{sec-global}

Fix a nonnegative test function $\widetilde{\psi}_\Gamma \in Lip^+_c(\Gamma \backslash G/M)$.  
We want to prove the following convergence and dominate its rate
$$ \int \widetilde{\psi}_\Gamma \dd \cal{M}_\Gamma^t  \xrightarrow[t \rightarrow + \infty]{} \int \widetilde{\psi}_\Gamma \dd m_{\Gamma\backslash G/M}.$$ 

For balls $B(z,s)$ with $z\in G/M$ and $s>0$, they will be balls with respect to the Riemannian distance $d_1$.

\begin{lem}\label{lem.changemetric}
Recall $\epsilon_0$ from Lemma \ref{lem-dist-compare}. For $0<s<\min\{\epsilon_0,(\log 2)/C_0 \}$ and any $z\in G/M$ and $x=\pi(z)\in X$, we have
\[B(z,s)\subset \widetilde\calF^{(2)}(x,s) \]
and for $\widetilde\varphi$ supported on $B(z,s)$
\[Lip_2\widetilde\varphi\leq C_x Lip \widetilde\varphi .\]
\end{lem}

\begin{proof}
By Lemma \ref{prop-corridor}, we have the first part.

By Lemma \ref{lem-dist-compare}, we have for $z_1,z_2\in B(z,s)$
\[\dd_1(z_1,z_2)\leq C_{\pi(z_1)}\dd_2(z_1,z_2)/4. \]
Now due to the definition of $C_x$, we have $C_{\pi(z_1)}\leq C_{\pi(z)}\exp(C_0d_X(\pi(z),\pi(z_1))\leq 2C_{\pi(z)}$. Therefore
\[\dd_1(z_1,z_2)\leq C_x \dd_2(z_1,z_2). \]
Then use the definition of Lipschitz norm.
\end{proof}

\paragraph{Partition of unity}

By applying Vitali's covering lemma to the collection $\{B(y,r/10)\}_{y\in \Gamma\backslash G/M}$, there exists a finite set $\{y_i\}_{i\in I}$ such that $B(y_i,r/10)$ are pairwisely disjoint and $\cup_{i\in I}B(y_i,r/2)$ is a covering of $\Gamma\backslash G/M$. By disjointness, we know $|I|\ll r^{-dim(G/M)}$. Fix a partition of unity of $\frac{1}{r}$-Lipschitz functions associated to the open cover $\cup_{i\in I}B(y_i,r)$. 
For the function $\widetilde\psi_\Gamma$ on $\Gamma\backslash G/M$, we can write it as $\widetilde\psi_\Gamma=\sum_{i\in I}\widetilde\psi_{\Gamma,i}$ using the partition of unity.
For each $y_i$, we can find a lift $z_i$ in $G/M$ such that $d(o,z_i)$ is less than the diameter of $\Gamma\backslash G/M$. By Lemma \ref{lem.changemetric}, we know that for $x_i=\pi(z_i)\in X$
\[B(z_i,r)\subset \widetilde{\calF}^{(2)}(x_i,r). \]
We can take $t$ large such that $r=e^{-u_2t}$ is smaller then the injectivity radius of $\Gamma\backslash G/M$. Then the two balls $B(z_i,r)$ and $B(y_i,r)$ are homeomorphic.
Let $\widetilde\psi_i$ be the lift of $\widetilde\psi_{\Gamma,i}$ on $B(z_i,r)$.

Furthermore, for every $i\in I$, the function $\widetilde{\psi}_i$ is Lipschitz and satisfies the following norm bounds:
\begin{itemize}
\item[(p1)] $Lip_2(\widetilde\psi_i)\leq C_x Lip\widetilde\psi_i \leq C_x(Lip\widetilde\psi_\Gamma+\frac{1}{r}\|\widetilde\psi_\Gamma\|_\infty)\leq \frac{C_x}{r} |\widetilde\psi_\Gamma |_{Lip},$
\item[(p2)] $\Vert \widetilde{\psi}_i \Vert_\infty  \leq \Vert \widetilde{\psi}_\Gamma \Vert_\infty,$
\item[(p3)] $\sum_{i\in I} \Vert \widetilde{\psi}_i \Vert_1  \leq \Vert \widetilde{\psi}_\Gamma \Vert_1,$
\end{itemize}
where the first inequality is due to Lemma \ref{lem.changemetric}.

\paragraph{Local domination}
For every $i \in I$, due to Lemma \ref{lem_M2-haar} and \ref{lem-cocom-weyl}, we have

\begin{align*}
\pm \bigg(  \int &\widetilde{\psi}_i \dd \cal{M}^t  - \int \widetilde{\psi}_i \dd m_{G/M}  \bigg)  
\leq r (C_3+ C)  \int \widetilde{\psi}_i  \dd m_{G/M} \quad + \\
&  (2r)^{\dim \frak{a}} \bigg( E(t\pm2r,\widetilde{\psi}_i,x_i) +  
 2\varepsilon Lip_2(\widetilde{\psi}_i)  \frac{\vert \Gamma \cap D_{t\pm 2r}(x_i) \vert \vol(\Gamma \backslash G)}{\vol(D_{t\pm 2r})} + 
 4 \Vert \widetilde{\psi}_i \Vert_\infty \frac{\vert \Gamma \cap D_{t \pm 2r}^{t_1}(x_i) \vert \vol(\Gamma \backslash G)}{\vol(D_{t \pm 2r})}  \bigg).
\end{align*}
Let's estimate the error term in the lower part. 
By Lemma \ref{lem-corridor-1GN}, (p1) and Lemma \ref{lem-vollip}, we have
$$E(t\pm2r,\widetilde{\psi}_i,x_i)=O(C_{x_i}Lip_2(\widetilde\psi_i)\vol(D_t)^{-\kappa})=
    O\bigg( \frac{C_{x_i}^2}{r} \vol(D_t)^{-\kappa} | \widetilde{\psi}_{\Gamma} |_{Lip} \bigg).$$
By compactness, the $x_i$ are in a bounded set, therefore the $\lbrace C_{x_i} \rbrace_{i\in I}$ are uniformly bounded.
Hence    
\begin{equation}\label{equ-et}
  E(t\pm2r,\widetilde{\psi}_i,x_i) =
    O\bigg( \frac{\vol(D_t)^{-\kappa}}{r}  | \widetilde{\psi}_{\Gamma} |_{Lip} \bigg). 
\end{equation}
By Lemma \ref{lem-dtxt} and (p1), we have
\begin{equation}\label{equ-gammadt}
2\varepsilon Lip_2(\widetilde{\psi}_i)  \frac{\vert \Gamma \cap D_{t\pm 2r}(x_i) \vert \vol(\Gamma \backslash G)}{\vol(D_{t\pm 2r})} 
= O \bigg(\frac{\varepsilon}{r} | \widetilde{\psi}_{\Gamma} |_{Lip}\bigg).
\end{equation}

Using that $t_1=3u_1t$, we get by applying Lemma \ref{lem-dtx} and (p2), 
\begin{equation}\label{equ-gammadt0}
    3 \Vert \widetilde{\psi}_i \Vert_\infty \frac{\vert \Gamma \cap D_{t \pm 2r}^{t_1}(x_i) \vert \vol(\Gamma \backslash G)}{\vol(D_{t \pm 2r})} 
    = O \big( \|\widetilde\psi_\Gamma\|_\infty  \vol(D_t)^{-\kappa(6u_1)} \big) 
    = O \bigg( \frac{\vol(D_t)^{-\kappa(6 u_1)}}{r} \vert \widetilde{\psi}_\Gamma \vert_{Lip}  \bigg).
\end{equation}

\paragraph{Global domination}
By the partition of unity, we have
\[\int \widetilde\psi_\Gamma \dd \calM_\Gamma^t=\sum_i\int\widetilde\psi_{\Gamma,i} \dd \calM_\Gamma^t=\sum_i\int\widetilde\psi_i \dd \calM^t \]
and
\[\int \widetilde\psi_\Gamma \dd m_{\Gamma\backslash G/M}=\sum_i\int\widetilde\psi_{\Gamma,i} \dd m_{\Gamma\backslash G/M}=\sum_i\int\widetilde\psi_i \dd m_{G/M}.\]
Therefore, by local dominations, $|I|\ll r^{-dim(G/M)}$ and \eqref{equ-et}-\eqref{equ-gammadt0}, we obtain
\begin{align*}
\int \widetilde\psi_\Gamma d\calM_\Gamma^t-\int \widetilde\psi_\Gamma \dd &m_{\Gamma\backslash G/M} = O \Bigg(  r  \sum_{i\in I}  \Vert \widetilde{\psi}_i \Vert_1
\\
&+  r^{-dim(G/AM)} \bigg(\frac{\vol(D_t)^{-\kappa}}{r}  | \widetilde{\psi}_{\Gamma} |_{Lip}
+ \frac{\varepsilon}{r} | \widetilde{\psi}_{\Gamma} |_{Lip}+\frac{\vol(D_t)^{-\kappa(6 u_1)}}{r} \vert \widetilde{\psi}_\Gamma \vert_{Lip} \bigg) \Bigg). 
\end{align*}
Using (p3) and $\Vert \widetilde{\psi}_\Gamma \Vert_1 \leq \|m_{\Gamma\backslash G/M} \| \; \vert \widetilde{\psi}_\Gamma \vert_{Lip}$, we deduce that
$$
\int \widetilde\psi_\Gamma d\calM_\Gamma^t-\int \widetilde\psi_\Gamma \dd m_{\Gamma\backslash G/M} =
O\Bigg( \bigg(  r  
+  \frac{\vol(D_t)^{-\kappa} + \varepsilon +\vol(D_t)^{-\kappa(6 u_1)}
}{r^{ \dim (G/AM)+1 }} 
  \bigg) \vert \widetilde{\psi}_\Gamma \vert_{Lip} \Bigg). 
$$

Recall the choice of parameter in \eqref{equ-parameter} where $\varepsilon=e^{-u_1t}$ and $r=e^{-u_2t}$.
Collecting all the error terms together, we obtain that there exists $u>0$ such that 
\[ \Big|\int \widetilde\psi_\Gamma \dd\calM_\Gamma^t-\int \widetilde\psi_\Gamma \dd m_{\Gamma\backslash G/M} \Big|= O(e^{-u t}|\widetilde\psi_\Gamma|_{Lip}). \]

\section{Finite index subgroups of $\sld$}\label{sec_sld}
\begin{center}
\fbox{
\begin{minipage}{.7 \textwidth}
In this section, $G= \mathrm{SL}(d,\mathbb{R})$ where $d\geq 2$ and $\Gamma_1<\Gamma_0= \mathrm{SL}(d,\mathbb{Z})$ is a finite index subgroup of $\Gamma_0$ which acts freely on $G/M$. We use $\Gamma$ to denote both $\Gamma_1$ and $\Gamma_0$ before section \ref{sec:equicpt}. Starting from Section \ref{sec:equicpt}, we only use $\Gamma$ to denote $\Gamma_1$.
\end{minipage}
}
\end{center}

Let us start with examples of finite index subgroups of $\Gamma_0$ that act freely on $G/M$.
For every prime number $p\geq 3$, we claim that the finite index subgroup $\Gamma_1 := \ker( \Gamma_0 \rightarrow \mathrm{SL}(d, \Z/p\Z ) )$ acts freely on $G/M$.
Indeed, assume $\gamma_1$ fixes an element $G/M$, then $\gamma_1$ is conjugated in $G$ to an element $m$ in the sign group $M$. Its projection to $\mathrm{SL}(d,\Z/p\Z)$ has the same eigenvalues given by the projection of $m$. Since $\gamma_1$ projects to the identity and $p\geq 3$, we deduce that $m$ is trivial.

\paragraph{Torus in linear algebraic groups}
We recall some concepts from linear algebraic groups. 
For more details please see \cite{borel_linear_1991} and \cite{borel_arithmetic_1962}. A subgroup $T$ of $\rm{GL}(d,\C)$ is a \emph{torus}, if $T$ is diagonalizable over $\C$ and isomorphic to $(\C^*)^n$. Let $k$ be a subfield of $\C$. We say that $T$ is a \emph{$k$-torus} if it is defined over $k$ i.e. if $T$ as an algebraic subvariety of $\rm{GL}(d,\C)$ is defined by polynomial equations with coefficients in $k$.
Denote by $T_k$ the $k$-points of $T$. 
A $k$ torus $T$ is \emph{$k$-split} (here we only need $k=\Q$ or $\R$) if $T$ can be diagonalized to $(\C^*)^n$ by a matrix with coefficients in $k$. 
Let $T$ be a $\Q$-torus $T$, then by \cite[Thm 9.4, Lem 8.4]{borel_arithmetic_1962} the following conditions are equivalent:
\begin{itemize}
    \item $T_\Z\backslash T_\R$ is compact.
    \item $T$ is \emph{$\Q$-anisotropic} i.e. all the $\Q$ characters from $T_\Q$ to $\Q^*$ are trivial.
    \item $T$ contains no non-trivial $\Q$-split subtorus.
\end{itemize}

\paragraph{Systole of elements in $\Gamma \backslash G$}
For $g\in G$, let $s(g)$ be the systole of the lattice $\Z^dg$ in $\R^d$ i.e.
\[s(g):=\min_{v\in \Z^dg \setminus \{0\}}\{\|v\| \}. \]
Note that the systole is preserved by right multiplication by $K$ since the norm on $\R^d$ is Euclidean.
Now $\Gamma$ preserves $\Z^d$ and the right action of the sign group $M$ also preserves any lattice $\Z^d g$ for all $g \in G$. 
Hence, this definition extends to $\Gamma g M$ in $\Gamma\backslash G /M$. For $R>0$, let
\[\Omega(R):=\{\Gamma g M\in \Gamma\backslash G /M |\ 1/s(g)\leq R \}. \]
Then the Mahler criteria implies that $\Omega(R)$ is compact. The union of $\Omega(R)$ for $R>0$ is the full space $\Gamma\backslash G/M$ and $\{\Omega(R),\ R>0\}$ is an increasing family of compact sets. 

\paragraph{Siegel domains}
In \cite[Section 4]{borel_arithmetic_1962}, Borel and Harish-Chandra define Siegel domains for the $KAN$ decomposition. We take the inverse of groups in their statement.  

Let $G=NAK$ be the Iwasawa decomposition, where $N$ is the upper triangular maximal unipotent subgroup.
\begin{defin}\label{defin-siegel}\cite{borel_arithmetic_1962}
For all $s>0$ and $u>0$, set $N_s:=\{n\in N|\ \|n\|\leq s\}$ and $A_u:=\{a\in A|\ a_j/a_{j+1}>u \}$. 
A \emph{Siegel domain} is a subset of $G$ of the form $N_s A_u K$, it is a \emph{standard Siegel domains} if $s>1/2$ and $0 < u < \sqrt{3}/2$.
\end{defin}
In \cite[Proposition 4.5]{borel_arithmetic_1962}, they prove that when $ N_s A_u K $ is standard, then 
$$G= \Gamma_0 N_s A_u K ,$$
which in some sense means that a standard Siegel domain is almost a fundamental domain for the left action of $\Gamma_0$ on $G$.
Furthermore, for any standard Siegel domain, the number of elements $ \gamma \in \Gamma_0$ such that $\gamma N_s A_u K \cap N_s A_u K \neq \emptyset$ is finite. 

From now on, we will denote by $n(g), a(g), k(g)$ the $N, A, K$ components of $g$ in the Iwasawa decomposition $NAK$.
Note that $a(g) = \exp ( -\sigma(g^{-1},\eta_0) ) .$

We give a relation between the systole and the Iwasawa cocycle in Siegel domains.
\begin{lem}\label{lem_syssiegel}
For all $0<u\leq 1$ and $g\in N A_u K$, we have
\[ a_d(g)u^{d-1}\leq s(g)\leq a_d(g). \]
\end{lem}
\begin{proof}
Using first the definition of the systole, then that the row vector $e_d$ is right $N$-invariant and finally that the norm on $\R^d$ is $K$-invariant, we deduce the upper bound of the systole
\[ s(g)\leq \|e_dg\|=\|e_da(g)\|=a_d(g). \]

For the lower bound, it suffices to prove that for every $v\in\Z^d \setminus \{0\}$, 
\[  a_d(g) u^{d-1} \leq \Vert vg \Vert. \]
First by $K$-invariance, $ \Vert vg \Vert= \Vert v\; n(g) a(g) \Vert.$
Let us write the coefficients of the row vector $v=(v_1,\cdots, v_d)$.
Assume that the $j$-th coefficient $v_j$ is the first non-zero coordinate, where $1 \leq j \leq d$. 
Then $v\; n(g) a(g)$ is a row vector with all its first $j-1$ coefficients equal to $zero$ and its $j$-th coefficient is 
$ (v \; n(g) a(g))_j = v_j n_{j j}(g) a_j(g).$
Using first that $n_{j j}(g) = 1$ and $\Vert v \; n(g) a(g) \Vert \geq \vert (v\; n( g) a(g))_j \vert$, then that $a( g) \in A_u $ and $\vert v_j \vert$ is a non-zero positive integer, we deduce that
\[\|vg\|\geq |v_j| a_j(g) \geq a_d(g) u^{d-j} . \]
Finally for every $v\in\Z^d \setminus \{0\}$, then $d-j \leq d-1$ and since $u \in (0,1)$, we deduce that $ a_d(g) u^{d-1} \leq  \Vert vg \Vert$, hence the lower bound for the systole.
\end{proof}

\paragraph{Injectivity radius}
We find a lower bound of the injectivity radius in every point of $\Omega(R) \subset \Gamma_1 \backslash G/M$.
For every point $z \in \Gamma_1 \backslash G/M$, denote by $inj(z)$, the injectivity radius with respect to the Riemannian metric $\dd_1$, i.e. the largest radius for which the exponential map at $z$ is a diffeomorphism. 
\begin{lem}\label{lem-injradius}
There exists $C_7>1$ such that for all large enough $R>2$, every $z\in \Omega(R)$,
\[ inj(z)\geq R^{-C_7}. \]
Furthermore, there exists a representative $h\in G$ such that $z=\Gamma_1 hM$ and 
\[\dd_X(o,ho)\leq C_7\log R. \]
\end{lem}
\begin{proof}
We first construct $h$.
Let $z \in \Omega(R)$ and we start with a representative $g \in G$ such that $z=\Gamma_1gM$. 
We choose $h_0 \in N_s A_u K$ a representative in the coset $\Gamma_0 g$, where $N_s A_u K$ is a standard Siegel domain (Definition \ref{defin-siegel}) with $u \in (0,1)$.
Note that $s(h_0)= s(z)>1/R$ by hypothesis, then by the above Lemma \ref{lem_syssiegel}, we deduce that $a_d(h_0)\geq 1/R$.
Since $a(h_0) \in A_u$, then $a_j(h_0) \geq \frac{u^{d-j}}{R}$ for all $1 \leq  j \leq d$.
Hence 
$$ \frac{u^{d-1}}{R} \leq   a_1(h_0)= a_2(h_0)^{-1}...a_d(h_0)^{-1} \leq R^{d-2} u^{-(d-2)(d-1)/2}$$ 
from which we deduce that $ a_j(h_0) \leq R^{d-2} u^{-(j-1)-(d-2)(d-1)/2}$ for all $1 \leq j \leq d$.
Since $N_s A_u K$ is standard, with $u \in (0,1)$, one can write it as some negative power of $R$ and deduce the following upper bound for $a(h_0)$ that there is a positive constant $C>0$ such that 
$$\|a(h_0^{-1})\|, \|a(h_0)\|\leq R^C.$$
Now since $N_s$ is bounded and for the operator norm $\Vert .\Vert$ of the action on row vectors induced by the Euclidean norm on $\R^d$, we deduce that $\|h_0\|=\|n(h_0)a(h_0)k(h_0)\| \ll R^C$, similarily for $\|h_0^{-1}\|$. 
Since $\Gamma_1$ is a finite index subgroup of $\Gamma_0$, there exists a finite set $\{\gamma_j\}_{j\in J}$ such that $\Gamma_0=\cup_{j\in J}\Gamma_1\gamma_j$. 
Therefore there exists $\gamma_j$ such that $\Gamma_1 g=\Gamma_1 \gamma_j h_0$.
We set $h:=\gamma_jh_0$ and deduce that 
\begin{equation}\label{equ-hh-}
 \|h^{-1}\|, \|h\|\ll R^{C}.
\end{equation}

Let us compute the Cartan projection of $h$, using \cite[Lemma 2.3]{kasselCorank08} and the compactness of $N_s$ and finiteness of $\lbrace \gamma_j \rbrace_{j\in J}$, 
\[\dd_X(o,ho) \ll \dd_X(o,h_0o) = \|\underline{a}_o(h_0)\|\ll 
\Vert \underline{a}_o(a(h_0)) \Vert \ll \log R. \]

Denote by $\vert .\vert_1$ the Riemannian metric at $e$ associated with the Riemannian distance$\dd_1$.
We choose $r_0 >0$ such that for all $Y \in \frak g$ of norm smaller than $r_0$, the exponential map is a local diffeomorphism, so that we have
$$ \Vert \exp(Y)-e \Vert \asymp \vert Y \vert_1 \asymp \dd_1(\exp(Y),e).$$
We prove that if the exponential map for the ball of radius $r \in (0,r_0)$ centered at $z=\Gamma_1 h M$ is not injective, then $r \gg R^{-C'}$ for some positive constant $C'$.
Assume there exist $h_1\neq h_2 \in G$ such that $\Gamma_1 h_1M=\Gamma_1 h_2M$ and $h_1M, h_2M \in B(hM,r)$. 
Abusing notations, since $\dd_1$ comes from the left $G$-invariant and right $K$-invariant Riemannian metric on $G$, we can assume that $h_1,h_2 \in B(h,r)$.
Then there exists $(\gamma,m)\in \Gamma_1\times M$, with $\gamma\neq e$ such that $\gamma h_1 = h_2 m$ i.e.
$$\gamma=h_2m h_1^{-1} .$$
Note that because $\Gamma_1$ acts freely on $G/M$, then $\gamma$ cannot be conjugated to an element in the sign group, therefore $\gamma^2 \neq e$.
Since $\gamma^2$ is a matrix with integer coefficient, we deduce on one hand the lower bound
\[ \|h^{-1}\gamma^2 h-e \|= \|h^{-1}(\gamma^2-e)h \| \geq \frac{1}{ \|h\|\|h^{-1}\|}\ . \]
On the other hand, set $g_0:=h^{-1}\gamma h-m$, so that $h^{-1} \gamma h = g_0 + m$ and deduce the upper bound 
\[\|h^{-1}\gamma^2 h-e \|=\|g_0^2+g_0m+mg_0\|\leq \|g_0^2\|+2\|g_0\|. \]
By triangle inequality, $\|g_0\|=\|h^{-1}h_2mh_1^{-1}h-m \|\leq \|h^{-1}h_2-e\|\|h_1^{-1}h\|+\|e-h_1^{-1}h \|$.
Now $h_1,h_2 \in B(h,r)$, therefore $\Vert g_0\Vert \ll r$ and 
\[\frac{1}{\|h\|\|h^{-1}\|} \leq \|g_0\|^2+ 2\|g_0\|\ll r. \]
Finally, by \eqref{equ-hh-}, therefore $r\gg R^{-C'}$ for some constant $C'>0$ and we deduce the lower bound for the injectivity radius at $z$.
\end{proof}

\paragraph{Action of the Weyl group}

\begin{lem}\label{lem_weyl}
There exists $c>0$, such that for any $\eta\in \calF$, there exists $w\in\calW$ such that 
\[\delta(w\eta,\eta_0)>c. \]
\end{lem}
\begin{proof}
In the $\mathrm{SL}(d,\R)$ case, the Furstenberg boundary $\calF$ is the space of complete flats of $\R^d$.
Therefore, there exists a basis $(v^j)_{1 \leq j \leq d}$ of $\R^d$ such that $\eta \in \calF$ is represented by $(\R v^1,\R v^1\wedge v^2,\cdots, \R v^1\wedge\cdots\wedge v^{d-1})$.
The Weyl group in the $\mathrm{SL}(d,\R)$ case is isomorphic to the permutation group $\mathfrak{S}_d$. 
It consists in square matrices $(w_{i j}) \in K$ of coefficients $w_{i j} = \delta_{\tau(i) j}$ where $\tau \in \mathfrak{S}_d$.   
Left multiplication of $(v_i^j)_{1 \leq i,j\leq d}$ by an element of the Weyl group permutes the columns, right multiplication by the transvection matrices in the upper triangular unipotent group $N$ correspond to operations on the lines of $(v_i^j)_{1 \leq i,j\leq d}$.
By Gaussian elimination, one can assume that $(v_i^j)_{1 \leq i,j\leq d}$, representative of $w \eta$ for some $w \in \calW$, is lower triangular and the coefficient in the diagonal is the highest in norm of the whole column i.e.
\begin{equation}\label{equ_vj}
v^j_l=0, \text{ for all }l<j, \text{ and } \ |v^j_j|=\max_{j\leq l\leq d}\{|v^j_l| \} \text{ for all $1 \leq j \leq d$}.
\end{equation}
On one hand, using that $\wedge^j \R^d$ are the Tits representations for $\mathrm{SL}(d,\R)$ and $(\zeta_0)_o^\perp =\eta_0$ in \eqref{defin-delta}, we compute
\begin{align*}
    \delta(w\eta,\zeta_0)=\inf_{1 \leq j \leq d}  \dd\big(\R v^1\wedge\cdots\wedge v^j,(\R e_1\wedge\cdots\wedge e_j)^\perp\big)=\inf_j \frac{ |v^1_1\cdots v^j_j|}{\|v^1\wedge \cdots\wedge v^j \|}.
\end{align*}
On the other hand, 
$ v^1 \wedge ... \wedge v^j = \sum_{1 \leq l_1 <...<l_j \leq d} \sum_{\tau \in \mathfrak{S}_j} sign(\tau) v_{l_1}^{\tau(1)} ... v_{l_j}^{\tau(j)} e_{l_1} \wedge ... \wedge e_{l_j}$ where $sign(\tau) \in \lbrace \pm 1\rbrace$ is the signature of the permutation $\tau$.
Hence for all $1 \leq j \leq d$, by triangle inequality and \eqref{equ_vj}
$$\Vert v^1 \wedge ... \wedge v^j \Vert \leq \binom{d}{j} j! \; \vert v^1_1...v_j^j \vert 
\leq d! \;  \vert v^1_1...v_j^j \vert .$$
We deduce that $\delta(w \eta, \zeta_0) \geq (d!)^{-1}$. The Lemma then follows by left multiplication by $k_\iota$ of $w \eta$ and $\zeta_0$, which by $K$-invariance of $\delta$ does not change the inequality.
\end{proof}

\begin{lem}
For any $g\in G$, there exists $w\in \calW$ such that for any $b=wb'w^{-1}$ with $b'\in \exp(-\frak a^{++})$, we have
\begin{equation}\label{equ_gb}
\log a(gb)=\log b'+\log a(g)+v,
\end{equation}
where $v$ is a vector of bounded length in $\frak{a}$ with the bound only depending on $c$ in Lemma \ref{lem_weyl}.
\end{lem}

\begin{proof} Since $NAK$ is a Iwasawa decomposition, we can compute the $A$ part by the Iwasawa cocycle. We have
\[\log a(g)=-\sigma(g^{-1},\eta_o). \]
Then if we multiple on the right of $g$ by an element $b\in A$, we obtain
\begin{align*}
    \log a(gb)=-\sigma(b^{-1}g^{-1},\eta_o)=-\sigma(b^{-1},g^{-1}\eta_o)+\log a(g).
\end{align*}
Due to $b=wb'w^{-1}$ with $w$ in the Weyl group and $b'$ in the negative Weyl chamber, then
\[\log a(gb)=-\sigma((b')^{-1},w^{-1}g^{-1}\eta_o)+\log a(g). \]
By Lemma \ref{lem_weyl}, there exists $w$ such that $\delta(w^{-1}g^{-1}\eta_o,\zeta_o)>c$.
By Lemma 14.2(i) and Lemma 6.33 in \cite{benoist-quint}, we finish the proof.
\end{proof}

\subsection{Compact periodic diagonal orbits}
The first difference with the cocompact case is that not every loxodromic element gives a periodic $A$-orbit in the quotient $\Gamma\backslash G/M$. So Selberg's lemma is not true. There is a general sufficient condition in \cite{pr}. For $\sld$, we know exactly when it fails. Recall for $\gamma$ loxodromic, we have defined an $A$-orbit $F_{[\gamma]}$ on $\Gamma\backslash G/M$.

\begin{lem}\label{lem_cpt}
Let $\gamma\in \Gamma$ be a loxodromic element. Then for  the following conditions:
\begin{itemize}
    \item[1] The $A$-orbit $F_{[\gamma]}$ is compact periodic;
    \item[2] The characteristic polynomial $p_\gamma(x)=\det(x-\gamma)$ of $\gamma$ is irreducible on $\Q[x]$;
    \item[3] There exists no non-trivial subset $I$ of $\{1,\cdots, d \}$ such that 
    \[\sum_{i\in I}t_i=0, \]
    where $(t_1,\cdots,t_d)$ is the Jordan projection of $\gamma$;
\end{itemize}
we have that $(1),(2)$ are equivalent and $(3)$ implies $(2)$.
\end{lem}

\begin{rem}
Here we give an example when $d=4$ that (1), (2) holds but (3) fails. We can find $\gamma$ in $\rm{SL}_4(\Z)$ by using the companion matrix such that $p_\gamma(x)=(x^2+(5-\sqrt{2})x+1)(x^2+(5+\sqrt{2})x+1)$. This polynomial $p_\gamma(x)$ is irreducible on $\Q[x]$ and has four different real roots. We can number them by their absolute values as $\lambda_1$ to $\lambda_4$. Then its roots satisfy that $\log|\lambda_1|+\log|\lambda_4|=\log|\lambda_2|+\log|\lambda_3|=0$.
\end{rem}

Before proving Lemma \ref{lem_cpt}, we need another lemma. Let $G_\gamma$ be the centralizer of $\gamma$ in $G$.
\begin{lem}\label{lem-rational}
 Let $\gamma$ be an element in $\Gamma$ such that its characteristic polynomial $p_\gamma$ is irreducible.
 If $\beta$ is an element in $G_\gamma$ with all eigenvalues rational, then $\beta$ is identity or minus identity.
\end{lem}

\begin{proof}
The element $\gamma$ is diagonalizable in the splitting field of $p_\gamma$, a Galois extension $K$ of $\Q$. There exists a vector $v_1\in K^d$ such that $\gamma v_1=\lambda_1v_1$ with $\lambda_1\in K$. Since $p_\gamma$ is irreducible, the Galois group $Gal(K/\Q)$ acts transitively on the roots of $p_\gamma$. We can get eigenvalues $\lambda_2,\cdots,\lambda_d$ and eigenvectors $v_2,\cdots, v_d$ as Galois conjugates of $\lambda_1$ and $v_1$ with $\gamma v_j=\lambda_j v_j$. The numbers $\lambda_j$ are distinct, hence $v_1,\cdots, v_d$ form a basis.

Due to $\beta$ commutes with $\gamma$, we have 
\[\beta v_j=\mu_j v_j\]
for some $\mu_j$ rational. Take $\sigma$ in the Galois group $Gal(K/\Q)$, then $\beta\sigma(v_j)=\mu_j\sigma(v_j)$. The Galois group $Gal(K/\Q)$ acts on the set $\{v_1,\cdots ,v_d \}$ transitively ($p_\gamma$ irreducible), which implies that $\mu_j$'s are equal. Since we are in $\slr$, we obtain the lemma.
\end{proof}

\begin{proof}[Proof of Lemma \ref{lem_cpt}]
We first prove $(3)$ implies $(2)$: If $p_\gamma(x)$ is reducible then $p_\gamma(x)=p_1(x)p_2(x)$ with $p_1,p_2$ monic and constant terms of $p_1,p_2$ equaling $\pm 1$. Suppose the absolute values of roots of $p_1$ are $\exp(t_i)$ for $i\in I\subset\{1,\cdots, d\}$. Then we obtain $\sum_{i\in I}t_i=0$ with $I$ non-trivial.

Now we prove that $(1)$ is equivalent to $(1')$, a condition about the centralizer of $\gamma$. Let $T_\gamma$ be the centralizer of $\gamma$ in $\rm{SL}(d,\C)$. The $A$-orbit $F_{[\gamma]}$ can be written as $F_{[\gamma]}=\Gamma gAM$, and $\gamma\in gAMg^{-1}$ due to the definition of $F_{[\gamma]}$. Because $\gamma$ is loxodromic, the centralizer $G_\gamma$ of $\gamma$ in $G$ equals $gAMg^{-1}$, the real points of the maximal $\R$-split $\Q$-torus $T_\gamma$. Now $F_{[\gamma]}$ is compact in $\Gamma\backslash G/M$ is equivalent to $\Gamma\cap G_\gamma\backslash G_\gamma=\Gamma_\gamma\backslash G_\gamma$ compact, where $\Gamma_\gamma$ is the centralizer of $\gamma$ in $\Gamma$. Notice that $\Gamma_\gamma\backslash G_\gamma$ is a finite cover of $(\Gamma_0)_\gamma\backslash G_\gamma= (T_\gamma)_\Z\backslash (T_\gamma)_\R$ for the $\Q$-torus $T_\gamma$.
Then due to \cite[Thm 9.4]{borel_arithmetic_1962}, $(1)$ is equivalent to 
\begin{itemize}
\item [1'] $T_\gamma$ is a $\Q$-anisotropic $\Q$-torus.
\end{itemize}
 
Then we prove $(2)$ implies $(1')$. Take a $\gamma$ satisfying $(2)$. If $T_\gamma$ is not $\Q$-anisotropic, then there exists a $\Q$-split subtorus \cite[Lem 8.4]{borel_arithmetic_1962}. Take $\beta$ in the $\Q$-points of this $\Q$-split torus, then all the eigenvalues of $\beta$ is rational. 
Hence by Lemma \ref{lem-rational}, the element $\beta$ must be $\pm \rm{Id}_d$. There is no nontrivial $\Q$-split subtorus of $T_\gamma$. We obtain a contradiction. So $T_\gamma$ is $\Q$-anisotropic.

Finally, we prove $(1')$ implies $(2)$. If $p_\gamma$ is reducible, suppose $\lambda_1,\cdots, \lambda_\ell$ with $1\leq \ell<d$ is an orbit of the Galois group $Gal(K/\Q)$ on the roots of $p_\gamma$. Here $K$ is the splitting field of $p_\gamma$. Set $v_j\in K^d$ the corresponding eigenvectors of $\lambda_j$, which is also an orbit of the Galois group $Gal(K/\Q)$. For any $\beta$ in the $\Q$ points of $T_\gamma$, since $\lambda_j$'s are different, we have for $1\leq j\leq \ell$
\[\beta v_j=\mu_jv_j. \]
On the symmetric power $Sym^\ell\R^d$, we have
\[\beta v_1\cdots v_\ell=\mu_1\cdots \mu_\ell(v_1\cdots v_\ell). \]
Now the vector $v_1\cdots v_\ell$ is fixed under the Galois group, so it is rational, hence $\mu_1\cdots \mu_\ell$ is also rational. We can define a $\Q$ character by $\chi(\beta)=\mu_1\cdots \mu_\ell$. Due to $1\leq \ell<d$, this $\Q$ character is non-trivial. So $T_\gamma$ is not $\Q$-anisotropic.
\end{proof}

\paragraph{Sparse set of loxodromic elements}
Let $\Gamma^{lox}_c$ be the subset of $\Gamma^{lox}$ whose elements also satisfy the condition (1) or (2) in Lemma \ref{lem_cpt}.
\begin{lem}\label{lem_noncpc}
There exists $1>\kappa_1>0$ such that for $t>1$, 
$$ \vert ( \Gamma^{lox} \setminus \Gamma_c^{lox}) \cap D_t \vert \ll \vol(D_t)^{1-\kappa_1} .$$
\end{lem}
Before proving this lemma, we need a result similar to Theorem 1.8 in \cite{gorodnikLiftingRestrictingSifting2012}. The proof is given in the appendix \ref{sec-subvar}.
\begin{prop}\label{prop_subvar}
Let $h$ be a polynomial on $\slr$ with $\Z$ coefficients and not vanishing identically on $\slr$. Then there exists $\kappa_h>0$ such that for $t>1$
\[ |\{\gamma\in \Gamma\cap D_t,\ h(\gamma)=0 \}|\ll \vol(D_t)^{1-\kappa_h} . \]
\end{prop}

\begin{proof}[Proof of Lemma \ref{lem_noncpc}]
By Lemma \ref{lem_cpt}, the number of elements $\gamma$'s not satisfying condition (1) is less than that not satisfying condition (3). The condition (3) in Lemma \ref{lem_cpt} can be translate to equations: $h_i(\gamma):=\det(1-\wedge^i\gamma)\det(1+\wedge^i\gamma)=0$. Then by power saving of integer points in subvarieties (Proposition \ref{prop_subvar}), we obtain the result.
\end{proof}

As a corollary, we can replace $o$ by another point $x$ in $X$, similar to Lemma \ref{lem-dtx}.
\begin{lem}\label{lem_axo}
 For $x\in X$ with $d_X(x,o)\leq \frac{\kappa_1 t}{4(1-\kappa_1)}$, we have
 \[ \vert ( \Gamma^{lox} \setminus \Gamma_c^{lox}) \cap D_t(x) \vert \ll \vol(D_t)^{1-\kappa_1/2}. \]
\end{lem}
\begin{proof}
By Lemma \ref{lem-cartan-points}, we have
\[\| \underline{a}_o(\gamma)-\underline{a}_x(\gamma )\|\leq 2 d_X(x,o)\leq\frac{\kappa_1 t}{2(1-\kappa_1)}. \]
Therefore by Lemma \ref{lem_noncpc} and $\vol(D_t)\in [1/C,C]e^{\delta_0 t}t^{\frac{\dim A -1}{2}}$, we obtain
\begin{align*}
|\{\gamma\in \Gamma^{lox}-\Gamma^{lox}_c,\ \underline{a}_x(\gamma) \in B_{\frak a}(0,t) \}|&\leq |\{\gamma\in \Gamma^{lox}-\Gamma^{lox}_c,\ \underline{a}_o(\gamma) \in B_{\frak a}(0,t+\frac{\kappa_1 t}{2(1-\kappa_1)}) \}|\\
&\ll \vol(D_{t+\frac{\kappa_1 t}{2(1-\kappa_1)}})^{1-\kappa_1}\ll \vol(D_t)^{1-\kappa_1/2}.
\end{align*}
The proof is complete.
\end{proof}

\subsection{Equidistribution for compactly supported functions}
\label{sec:equicpt}

In order to make $\Gamma\backslash G/M$ a manifold, from now on we only consider $\Gamma=\Gamma_1$.
Similar to the cocompact case, we also need to change the formulation to conjugacy classes of loxodromic elements. Let $\calG_c^{lox}$ be the set of $\Gamma$ conjugacy classes of $\Gamma_c^{lox}$. 
\begin{lem}\label{lem-gloxc}
There is a bijection between $\calG_c^{lox}$ and $\calG(A)$.
\end{lem}
\begin{proof}
The proof is almost the same as the proof of Lemma \ref{lem-gaglox}. We replace the use of Lemma \ref{lem-selberg} by Lemma \ref{lem_cpt}. The only difference is that we obtain $\gamma_Y$ from $(Y,F)$, but we only know that $\gamma_Y$ is in $\Gamma^{lox}$. Due to $F=F_{\gamma_Y}$ compact, from Lemma \ref{lem_cpt}, we know that indeed $\gamma_Y$ is in $\Gamma^{lox}_c$.
\end{proof}

By the previous lemma, we obtain
\[\calM_\Gamma^t:=\frac{\vol(\Gamma\backslash G)}{\vol(D_t)} \sum_{F\in C(A)}|\Lambda(F)\cap B_{\frak a}^{++}(0,t)|L_F=\frac{\vol(\Gamma\backslash G)}{\vol(D_t)} \underset{ \Vert \lambda (\gamma) \Vert \leq t}{ \sum_{ [\gamma] \in \calG^{lox}_c}}L_\gamma.  \]
We consider the lift of the measure $\calM^t_\Gamma$ to $G/M$,
\begin{equation}\label{eq_Mesc}
\cal{M}^t:= \frac{\vol(\Gamma\backslash G)}{\vol(D_t)} \underset{ \Vert \lambda (\gamma) \Vert \leq t}{ \sum_{ \gamma \in \Gamma^{lox}_c}}\calL_\gamma. 
\end{equation}

The main result of this part is the equidistribution on large compact sets. 

\begin{prop}\label{prop-omegat}
There exist $\zeta>0$ and $u>0$ such that for all $t>0$ and all $f \in Lip_c(\Omega(e^{\zeta t}))$, 
\begin{equation}
    \bigg\vert \calM^t_\Gamma(f)-\int f \dd m_{\Gamma\backslash G/M} \bigg\vert \ll    e^{-ut} \vert f\vert_{Lip} .
\end{equation}
\end{prop}

Before starting the argument, we fix the parameters which will be used later. 
Choose $u_1>0$ smaller than $\min\{\epsilon_G,1\}/10$, where $\epsilon_G$ is the constant from Lemma \ref{volume_reste}. 
Set
$$
\varepsilon:= e^{-u_1 t} \text{ and }t_1 := 3 u_1 t. $$
Consider the decay rate function $u \mapsto \kappa(u) >0$ satisfying \eqref{volume_reste}, the decay coefficient $\kappa>0$ given in Theorem \ref{theo-gorodnik-nevo} and $\kappa_1$ given in Lemma \ref{lem_noncpc}.
We set 
\begin{equation}\label{equ-noncocomp}
    u_2 :=  \frac{1}{4  \dim (G/M)}  \min \bigg\lbrace  \delta_0 \kappa(6 u_1), \frac{\delta_0 \kappa}{3}, u_1  ,  \delta_0 \kappa_1 \bigg\rbrace,\ r:= e^{-u_2 t}.
\end{equation}
Consider the constant $C_7$ coming from the injectivity radius Lemma \ref{lem-injradius}, the constant $C_5$ from the counting Lemma \ref{lem-dtxt} and $C_0$ coming from the growth rate of $C_x$ given in \eqref{equ-cx}.
Set the exponential decay rate of the systole
\begin{equation}\label{equ-zeta}
    \zeta:= \frac{1}{C_7} \min \bigg\lbrace u_2,  \frac{\kappa_1}{4(1-\kappa_1)}, \frac{1}{C_5},   \frac{3u_1}{2(1-6u_1)}, \frac{\kappa(6u_1)}{4(1-\kappa(6u_1))},\frac{\kappa \delta_0}{6 C_0},\frac{u_1}{4C_0} \bigg\rbrace.
\end{equation}

\paragraph{Equidistribution of Weyl chambers}
We define 
\begin{equation}\label{eq_M4}
\cal{M}_{x,3}^t:= \frac{\vol(\Gamma\backslash G)}{\vol(D_t)} \sum_{ \gamma \in \Gamma^{lox}_c\cap D^{reg}_t(x) } \calL_\gamma. 
\end{equation}
The following Lemma is a direct consequence of the definition of $\calM_{x,2}^t$ given in \eqref{eq_M2}.
\begin{lem}\label{lem-M3M2}
For all $t>1$, all $x\in X$ and 
for every test function $\widetilde{ \psi }  \in Lip^+_c( \widetilde{\cal{F}}^{(2)}(x,r))$,  

\begin{equation*}
\bigg|\int \widetilde{ \psi }\ d\cal{M}_{x,3}^t-\int \widetilde{ \psi }\ d\calM_{x,2}^t\bigg| \leq
\frac{ \vert (\Gamma^{lox} \setminus \Gamma^{lox}_c) \cap D_t(x) \vert \vol(\Gamma \backslash G)}{\vol(D_t)} 
\|\widetilde{\psi}\|_\infty .
\end{equation*}

\end{lem}
The following statement and its proof is the same as Lemma \ref{lem-cocom-weyl} provided one replaces $\calM_{x,2}^t$ with $\calM_{x,3}^t$.
\begin{lem}\label{lem-sldz-weyl}
There exists $C>0$. Fix $x\in X$, for every test function $\widetilde{\psi}\in Lip^+_c( \widetilde{\cal{F}}^{(2)}(x,r))$, 
\begin{equation}\label{eq-lem-sldz-weyl}
(1-Cr) \int  \widetilde{\psi} \dd \cal{M}_{x,3}^{t-2r} \leq  \int  \widetilde{\psi} \dd \cal{M}^t   \leq (1+Cr) \int  \widetilde{\psi} \dd \cal{M}_{x,3}^{t+2r}+\|\widetilde{\psi}\|_\infty\frac{|\Gamma\cap D_{t+2r}^{2r}(x)| \vol(\Gamma \backslash G)}{\vol(D_t)}. 
\end{equation}
\end{lem}

\paragraph{Partition of unity}
Then we only need to use a partition of unity to obtain the global version as in Section \ref{sec-global}. 
On the compact set $\Omega(e^{\zeta t})$, by Vitali's covering lemma, there exists a finite set $\{y_i\}_{i\in I}$ in $\Omega(e^{\zeta t})$ such that the $B(y_i,r/10)$ are pairwise disjoint and $\cup_{i\in I}B(y_i,r/2)$ covers $\Omega(e^{\zeta t})$. By disjointness, $|I|\ll r^{-dim(G/M)}$. By the injectivity radius Lemma \ref{lem-injradius} and choice of $\zeta$ in \eqref{equ-zeta} such that $C_7\zeta\leq u_2$, the balls $B(y_i,r)$ are diffeomorphic to balls of radius $r$ in $G/M$. We can take a partition of unity of $\frac{1}{r}$-Lipschitz functions associated to the open cover $B(y_i,r)$. For each $y_i$, by Lemma \ref{lem-injradius} we can find a lift $z_i\in G/M$ such that
\begin{equation}\label{equ_oxi}
    d_X(o, x_i)\leq C_7\zeta t.
\end{equation}
where $x_i=\pi_X(z_i)$.
We have the same Lipschitz bounds on $\widetilde\psi_i$. By Lemma \ref{lem.changemetric} and (p1)
\[ \supp\widetilde{\psi_i}\subset B(z_i,r)\subset \widetilde{\calF}^{(2)}(x,r)\text{ and }Lip_2(\widetilde{\psi}_i)\leq C_{x_i} Lip(\widetilde{\psi_i}) \leq \frac{C_{x_i}}{r} \vert \widetilde{\psi}_\Gamma \vert_{Lip} . \]
Hence, by the above equation \eqref{equ_oxi}, then $Lip_2 (\widetilde{\psi}_i) \ll \frac{e^{C_0 C_7 \zeta t}}{r} \vert \widetilde{\psi}_\Gamma \vert_{Lip} .$
\paragraph{Local domination}
By using Lemma \ref{lem_M2-haar}, \ref{lem-M3M2} and \ref{lem-sldz-weyl}, we obtain similar local domination:
\begin{align*}
\pm \bigg(  \int &\widetilde{\psi}_i \dd \cal{M}^t  - \int \widetilde{\psi}_i \dd m_{G/M}  \bigg)  
\leq r (C_3 + C)   \int \widetilde{\psi}_i  \dd m_{G/M} \quad + \quad
\frac{ \vert (\Gamma^{lox} \setminus \Gamma^{lox}_c) \cap D_t(x_i) \vert \vol(\Gamma \backslash G) }{\vol(D_t)} 
\|\widetilde{\psi}_i\|_\infty \\
&  (2r)^{\dim \frak{a}} \bigg( E(t\pm2r,\widetilde{\psi}_i,x_i) +  2\varepsilon Lip_2(\widetilde{\psi}_i)  \frac{\vert \Gamma \cap D_{t\pm 2r}(x_i) \vert \vol(\Gamma \backslash G)}{\vol(D_{t\pm 2r})} + 4 \Vert \widetilde{\psi}_i \Vert_\infty \frac{\vert \Gamma \cap D_{t \pm 2r}^{t_1}(x_i) \vert \vol(\Gamma \backslash G)}{\vol(D_{t \pm 2r})}  \bigg).
\end{align*}
For the right term in the first line, by choice of $\zeta$ in \eqref{equ-zeta} and Lemma \ref{lem_axo}, we deduce
\[ \frac{ \vert (\Gamma^{lox} \setminus \Gamma^{lox}_c) \cap D_t(x_i) \vert \vol(\Gamma \backslash G)}{\vol(D_t)} \ll (\vol(D_t))^{-\kappa_1/2}.
\]

The lower line contains terms similar to \eqref{equ-et} \eqref{equ-gammadt} \eqref{equ-gammadt0} that appear in the cocompact case. 
Now that $x_i$ can be far from $o$, the constant $C_{x_i}$ can be big.
However, by \eqref{equ_oxi}, this distance is bounded above by $C_7\zeta t$, which implies the following upper bound of $C_{x_i}^2$:
\begin{equation}\label{equ-cxi}
C_{x_i}^2= (8 C_2 C_1)^2 e^{2C_0d_X(x_i,o)} \ll e^{2C_0C_7\zeta t }.
\end{equation}
Hence for the left-lower term, we deduce that
$$E(t\pm2r,\widetilde{\psi}_i,x_i)=
    O\bigg( \frac{C_{x_i}^2}{r} \vol(D_t)^{-\kappa} | \widetilde{\psi}_{\Gamma} |_{Lip} \bigg) = 
    O \bigg( \frac{e^{2 C_0 C_7 \zeta t}}{r}  \vol(D_t)^{-\kappa} | \widetilde{\psi}_{\Gamma} |_{Lip}\bigg) .$$
    
For the mid-lower term similar to \eqref{equ-gammadt}, by Lemma \ref{lem-dtxt} and since $C_5 C_7 \zeta <1$ according to the choice of $\zeta$ in \eqref{equ-zeta}, we deduce that 
$ \frac{\vert \Gamma \cap D_{t \pm 2r}(x_i) \vert}{\vol(D_{t\pm 2r})} \leq C. $

For the right-lower term similar to \eqref{equ-gammadt0}, using that $C_7 \zeta \leq \min \lbrace \frac{3 u_1}{2(1-6 u_1)}, \frac{\kappa(6 u_1)}{4(1- \kappa(6 u_1))} \rbrace$ as given in \eqref{equ-zeta}, the hypothesis of Lemma \ref{lem-dtx} are satisfied. Hence 
$\frac{\vert \Gamma \cap D_{t \pm 2r}^{t_1}(x_i) \vert \vol(\Gamma \backslash G)}{\vol(D_{t \pm 2r})}\ll \vol(D_t)^{-\kappa(6u_1)}.$

\paragraph{Global domination}
Finally, by summing over the partition of unity and by $|I|\ll r^{-dim(G/M)}$, then collecting the above estimates together, we deduce that
\begin{align*}
\bigg|\int \widetilde{\psi}_\Gamma \dd \calM_\Gamma^t - \int \widetilde{\psi}_\Gamma \dd m_{\Gamma \backslash G /M}\bigg| \quad  \ll \quad r &\Vert \widetilde{\psi}_\Gamma \Vert_1 \quad + \quad \frac{\vert \widetilde{\psi}_\Gamma \vert_{Lip}}{r^{\dim(G/M)}} \vol(D_t)^{- \frac{\kappa_1}{2}} \\
&+ \frac{\vert \widetilde{\psi}_\Gamma \vert_{Lip}}{r^{\dim(G/AM)+1}} 
\Big(e^{2 C_0 C_7 \zeta t} \vol(D_t)^{-\kappa} + e^{ C_0 C_7 \zeta t} \varepsilon + \vol(D_t)^{-\kappa(6 u_1)} \Big).
\end{align*}
The proof of Proposition \ref{prop-omegat} is complete due to the choices of $\zeta$ and $r=e^{-u_2 t}$ in \eqref{equ-noncocomp} and \eqref{equ-zeta}, where $\varepsilon=e^{-u_1 t}$.

\subsection{Non-escape of mass}\label{sec-nonescape}
In order to prove the equidistribution for all bounded Lipschitz functions,
we only remains to prove that $\calM^t_\Gamma(\Omega(e^{\vartheta t})^c)$ tends to zero as $t$ tends to infinity, where we set
\[ \vartheta=\zeta/2.\] 
\begin{lem}[non-escape of mass]\label{lem_nonescape}
There exists $c_4>0$ such that
\[\calM^t_\Gamma(\Omega(e^{\vartheta t})^c)\ll e^{-c_4t}. \]
\end{lem}
\begin{proof}[Proof the main theorem for $\Gamma<\sld$]
Take a Lipschitz cutoff function $\phi$ supported on $\Omega(e^{\zeta t})$ and equals $1$ on $\Omega(e^{\vartheta t})$. Let $f_1=\phi f$ and $f_2=(1-\phi)f$. Then $f_1$ is supported on $\Omega(e^{\zeta t})$ and $f_2$ is supported on $\Omega(e^{\vartheta t})^c$ and with the same Lipschitz bound as $f$.
By applying Proposition \ref{prop-omegat} to $f_1$ and Lemma \ref{lem_nonescape}, we obtain 
\begin{align*}
    \bigg|\int f\dd\calM_\Gamma^t-\int f\dd m_{\Gamma\backslash G/M}\bigg| &\leq \bigg|\int f_1\dd\calM_\Gamma^t-\int f_1\dd m_{\Gamma\backslash G/M}\bigg|+\bigg|\int f_2\dd\calM_\Gamma^t-\int f_2\dd m_{\Gamma\backslash G/M}\bigg|\\
    &\ll e^{-ut}|f_1|_{Lip}+m_{\Gamma\backslash G/M}((\Omega(e^{\vartheta t})^c)|f_2|_\infty+\calM_\Gamma^t((\Omega(e^{\vartheta t})^c)|f_2|_\infty\ll e^{-u't}|f|_{Lip},
\end{align*}
here we need a volume estimate (see for example Proposition 7.1 in \cite{kleinbock_margulis}), that is
\[m_{\Gamma\backslash G/M}(\Omega(e^{\vartheta t})^c)\ll e^{-ct}. \]
The proof is complete.
\end{proof}

For $0<t_1<t_2$ We define 
\[\Omega(t_1,t_2):=\{\Gamma g\in\Gamma\backslash G/M,\ t_1< 1/s(g)\leq t_2 \}=\Omega(t_2) \setminus \Omega(t_1). \]
Let's state our key observation.

\begin{theorem}
\label{lem_key}
There exists $C>0$. For $t>C$ and $\gamma\in \Gamma^{lox}_c$ with $\|\lambda(\gamma)\|\leq t$, then
\[Leb(F_{[\gamma]}\cap \Omega(e^{\vartheta t})^c)\ll Leb(F_{[\gamma]}\cap\Omega(e^{\vartheta t/8},e^{\vartheta t})). \]
\end{theorem}

\begin{proof}[From Theorem \ref{lem_key} to Lemma \ref{lem_nonescape} ]
Take a Lipschitz function $f$ such that $f$ takes value in $[0,1]$, $f$ equals $1$ on $\Omega(e^{\vartheta t/8},e^{\vartheta t})$, the support of $f$ is contained in the $1$ neighbourhood of $\Omega(e^{\vartheta t/8},e^{\vartheta t})$ and $|Lip(f)|\leq 2$. Then we obtain
\begin{align*}
\calM^t_\Gamma(\Omega(e^{\vartheta t})^c)&\ll \calM^t_\Gamma(\Omega(e^{\vartheta t/8},e^{\vartheta t}))\leq  \calM^t_\Gamma(f).
\end{align*}
For an element in the support of $f$, we can write it as $\Gamma gh$ with $g\in \Omega(e^{\vartheta t/8},e^{\vartheta t})$ and $h\in B(e,1)$. Then due to the region of $h$ we have for $v\in \Z^d$
\[\|vgh\|=\| (vg)h\|\in\|vg\|[1/C',C']  \]
for some $C'>1$. Therefore the 1 neighbourhood of $\Omega(e^{\vartheta t/8},e^{\vartheta t})$ is contained in $\Omega(e^{\zeta t})$ if $t$ large. Applying Proposition \ref{prop-omegat}, we have
\[ \bigg|\calM^t_\Gamma(f)-\int f\ dm_{\Gamma\backslash G/M}\bigg|\ll e^{-ut}|f|_{Lip}.\]
Then by the choice of $f$, we obtain
\begin{align*}
\calM^t_\Gamma(\Omega(e^{\vartheta t})^c)&\ll m_{\Gamma\backslash G/M}(\Omega(e^{\vartheta t/8}/C',C'e^{\vartheta t}))+e^{-ut}\ll e^{-ct},
\end{align*}
here we need a volume estimate $m_{\Gamma\backslash G/M}(\Omega(e^{\vartheta t/8}/C')^c)\ll e^{-ct}$.
\end{proof}

In order to prove Theorem \ref{lem_key},
we start with a lemma between the systole of $\Z^d g$ for $\Gamma g\in F$ and the length of an element in $\Lambda(F)$. Since $F$ itself is a torus, we can also interpret it as the relation between the systole of the torus $F$ with the cusp excursion of $F$ on $\Gamma\backslash G$.

\begin{lem}\label{lem_syslength}
There exists $C_d>0$. For $\gamma\in \Gamma^{lox}_c$ and $F_{[\gamma]}$ a compact periodic $A$-orbit in $\Gamma\backslash G/M$, we have 
\[F_{[\gamma]}\subset \Omega(\exp(C_d\|\lambda(\gamma)\|)). \]
\end{lem}

\begin{rem}
This lemma is inspired by the discriminant of compact $A$-orbit defined in \cite{einsiedler_distribution_2011}. Here we give a direct relation without using the discriminant.
\end{rem}

\begin{proof}
Take a point $\Gamma g M\in F_{[\gamma]}$, then there exists $a\in AM$ such that $a=g^{-1}\gamma g$. The Jordan projection of $a$ is the same as $\gamma$, that is $\lambda(a)=\lambda(\gamma)$.

Take a nonzero vector $x\in \Z^dg$. Then by $\Z^dga=\Z^d\gamma g=\Z^dg$, we obtain
\[xa\in \Z^dg,\cdots, xa^{d-1}\in \Z^dg. \]
Now $x,xa,\cdots, xa^{d-1}$ generates a sublattice in $\Z^dg$. There is no $j$ such that $x_j=0$, otherwise the length of $xb$ for $b\in A$ can be arbitrarily small, which contradicts the fact that $F_{[\gamma]}$ is compact. Hence its covolume satisfies
\[  \vol(\R^d/\left<x,xa,\cdots, xa^{d-1}\right>)=|\prod_{1\leq j\leq d}x_j\det(1,a,\cdots, a^{d-1})|,  \]
where in $\det(1,a,\cdots, a^{d-1})$, the element $a_j$ is seen as a column vector. 
Now different coordinates of $a$ are different, so the determinant of the Vandermonde matrix in the above formula is nonzero. Hence the lattice generated by $x,xa,\cdots, xa^{d-1}$ has rank $d$ and its covolume is greater than 1.
Hence
\[\frac{1}{|\prod_{1\leq j\leq d}x_j |}\leq |\det(1,a,\cdots, a^{d-1})|\leq \exp(C\|\lambda(a)\|)=\exp(C\|\lambda(\gamma)\|). \]
Therefore by the inequality of arithmetic and geometric means
\[\max_{b\in A}\frac{1}{\|xb\|}\ll \max_{b\in A}\frac{1}{|\prod_{1\leq j\leq d}(xb)_j |^{1/d}}= \frac{1}{|\prod_{1\leq j\leq d}x_j |^{1/d}}\leq \exp(C\|\lambda(\gamma)\|/d). \]
Finally, we obtain
\[\min_{b\in A}s(\Z^dgb)\geq \min_{x\in \Z^dg-\{0\},b\in A}\|xb\|\gg \exp(-C\|\lambda(\gamma)\|/d). \]
\end{proof}

This lemma tells us that the compact periodic $A$-orbit appearing in $\calM^t_\Gamma$ is always contained in $\Omega(\exp(C_dt))$. 
In order to prove Theorem \ref{lem_key}, we need another lemma describing the growth of systole under the $A$ action.
\begin{lem}\label{lem_growth}
There exist $C, C_8>0$. Let $y$ be an element in $\Gamma\backslash G/M$ with $1/s(y)\in[e^{\vartheta t},\exp(C_dt)]$ and $t>C$, there exists $b\in A$ such that $1/s(yb)\in (e^{\vartheta t/4},e^{\vartheta t/2})$ and $\|\log b\|\leq C_8 t$.
\end{lem}
\begin{rem}
This lemma is similar to Proposition 4.1 in \cite{tomanov_weiss}, where they also study the growth of systole to prove that there exists a compact set which intersects each orbit of some maximal $\R$-split torus.

The idea of the proof is that in the Siegel domain, the systole and a cocycle is comparable. Since this cocycle is additive, we can estimate its value after $A$ action, which in turn gives the estimate of systole. 
\end{rem}

\begin{proof}
We only give the proof for $\Gamma_0$, since our definition of Siegel domain only works for $\Gamma_0$. If we have an element $y$ in $\Gamma\backslash G/M$, we can project it to $y'$ in $\Gamma_0\backslash G/M$ and apply the lemma there to find a $b$. Then due to the invariance of the systole under covering $s(yb)=s(y'b)$, this $b$ also works for $y$ in the lemma.  

For $y$ in $\Gamma_0\backslash G/M $ with $1/s(y)\in[e^{\vartheta t},\exp(C_d t) ]$, with out loss of generality, we can find a $g$ in a standard Siegle domain $N_{s_0}A_{u_0}K$ such that $y=\Gamma_0 gM$. Then by Lemma \ref{lem_syssiegel}
\begin{equation}\label{equ_sg}
    a_d(g)u_0^{d-1}\leq s(g)\leq a_d(g).
\end{equation}
For $\log a\in\frak a$, we define a character 
$$\chi(\log a)=\log(a_1\cdots a_{d-1})=-\log a_d. $$
By \eqref{equ_sg}, we obtain that $a(g)$ is in
$$\{a\in A|\  \log a_j-\log a_{j+1}\geq \log u_0,\ 1\leq j\leq d-1,\ \chi(\log a)\in [\vartheta t,C_dt-(d-1)\log u_0] \}.$$ 
Using the affine coordinate, this is equivalent to say that $\log a(g)=v^0+Y$, 
with $v^0\in\frak a$, $v^0_j-v^0_{j+1}=\log u_0$, $Y\in\frak a^{++}$ and 
$$\chi(Y)\in[\vartheta t,C_dt-(d-1)\log u_0]-(d\log u_0)/2.$$
Applying \eqref{equ_gb} to this $g$, there exists $w$ in $\calW$ such that for all $b'$ in the negative Weyl chamber, the equation \eqref{equ_gb} holds.
\begin{figure}
    \centering
    \includegraphics[width=5cm]{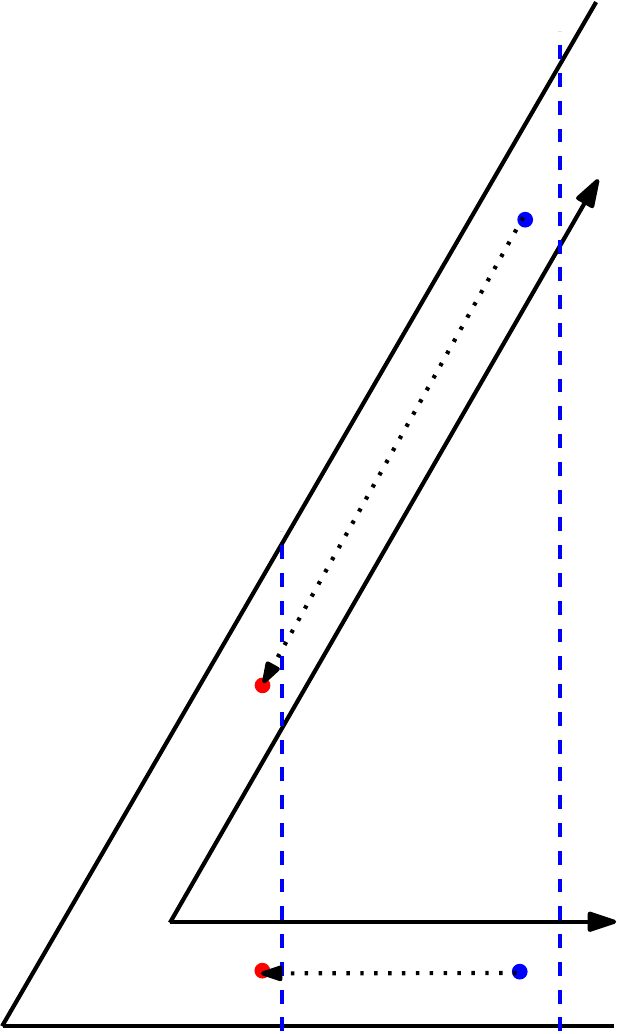}
    \caption{This is a neighbourhood of the positive Weyl chamber. Two blue dashed lines are lines with the fundamental weight $\chi(a)=\log a_1+\log a_2$ equal to $t$ and $\vartheta t$. Two blue points are $\log a(g)$ and two red points are $\log a(g)+\log b'$. The point $\log a(gb)$ has bounded distance to the red point.}
    \label{fig:weyl}
\end{figure}
We can take a $b'=\exp(-sY)$ with $0<s<1$ such that
\[\chi(Y-sY)=\vartheta t/3. \]
Then $\|\log b'\|=\|sY\|\leq C_8t$ if $t$ is large with respect to $\log u_0$. Therefore, We have 
\[\log a(gb)=v^0+(1-s)Y+v=(v^0+v)+(1-s)Y \]
is in 
\[\{a\in A|\  \log a_j-\log a_{j+1}\geq \log u,\ 1\leq j\leq d-1,\ \chi(\log a)\in [\vartheta t/3-C,\vartheta t/3+C] \},\]
with some $1>u>0$ and $C>0$ only depending on $c$ in Lemma \ref{lem_weyl} and $u_0$.
By Lemma \ref{lem_syssiegel}, we obtain that
\[ a_d(gb)u^{d-1}\leq s(gb)\leq a_d(gb).\]
When $t$ is large, we obtain $gb$ with $1/s(gb)\in(e^{\vartheta t/4},e^{\vartheta t/2})$.
\end{proof}

\begin{proof}[Proof of Theorem \ref{lem_key}]
The compact periodic $A$-orbit $F_{[\gamma]}$ is isometric to the flat torus $\frak a/\Lambda(F_{[\gamma]})$. We use the quotient Euclidean distance on $F_{[\gamma]}$ and a ball $B(x,r)$ for $x\in F_{[\gamma]}$ and $r>0$ will be a ball in $F_{[\gamma]}$ with respect to this distance. We are working on the flat torus with a height function given by the inverse of the systole, $1/s(x)$, which tells us how this flat torus is embedded in the large non compact space $\Gamma\backslash G/M$.


\textbf{We can find a maximal family of points $\{x_j\}_{j\in J}\subset F_{[\gamma]}$ in $\Omega(e^{\vartheta t/4},e^{\vartheta t/2})$ such that $d(x_j,x_{j'})\geq C_8t$ for $j\neq j'$ and the union of balls $\cup_{j\in J}B(x_j,2C_8t)$ covers $F_{[\gamma]}\cap \Omega(e^{\vartheta t})^c$.} 
This is always possible, because if the union of balls doesn't cover $F_{[\gamma]}\cap \Omega(e^{\vartheta t})^c$. Then take a point $x$ not covered by the union, so $x$ has distance greater than $2C_8t$ to $\{ x_j\}_{j\in J}$. Using Lemma \ref{lem_syslength} and \ref{lem_growth}, we can find $xb$ with $xb\in \Omega(e^{\vartheta t/4},e^{\vartheta t/2})$ and $d(x,xb)\leq \|\log b\|\leq C_8t$. This point $xb$ has distance greater than $C_8t$ to $\{x_j\}_{j\in J}$. So we can add this point. By this way we can find the desired family of points.

The Lebesgue measure is the quotient measure on the flat torus $F_{[\gamma]}\simeq\frak{ a}/\Lambda(F_{[\gamma]})$. Then due to covering,
\begin{align}\label{equ-fgamma}
    Leb(F_{[\gamma]}\cap\Omega(e^{\vartheta t})^c)\leq\sum_{j\in J}Leb(B(x_j,2C_8t))\ll \sum_{j\in J}Leb(B(x_j,\vartheta t/8)).
\end{align}
The last inequality $Leb(B(x_j,2C_8t))\ll Leb(B(x_j,\vartheta t/8)) $ is due to that we can use a finite number of balls of radius $\vartheta t/8$ to cover a ball of radius $2C_8t$ in the flat torus and the number of balls needed doesn't depend on $t$.
For any $b\in A$ and $v\in\R^d$, by $\|vb\|\in\|v\|\{\min_i |b_i|,\max_i |b_i| \}$, we obtain that for any $x\in \Gamma\backslash G/M$
\[s(xb)/s(x)\in[e^{-\|\log b\|}, e^{\|\log b\|}]. \]
For any $y\in B(x_j,\vartheta t/8)$, we can write it as $y=x_jb$ with $b\in A$ and $\|\log b\|\leq \vartheta t/8$. Therefore we know that $B(x_j,\vartheta t/8)\subset \Omega(e^{\vartheta t/8},e^{\vartheta t})$. The balls $B(x_j,\vartheta t/8)$ are disjoint by hypothesis that $d(x_j,x_{j'})\geq C_8t$. Therefore
\begin{align}\label{equ-fgammaeps}
    \sum_{j\in J}Leb(B(x_j,\vartheta t/8))= Leb(\cup_{j\in J}B(x_j,\vartheta t/8))\leq Leb(F_{[\gamma]}\cap \Omega(e^{\vartheta t/8},e^{\vartheta t})).
\end{align}
The proof is complete by \eqref{equ-fgamma} and \eqref{equ-fgammaeps}.
\end{proof}

\section{Appendix A: Orbifolds and partitions of unity }
Let $\Gamma$ be any finite index subgroup of $\sld$. 
Its left action on $G/M$ is no longer assumed to be free and the space $\Gamma \backslash G/M$ is now an orbifold. 
In Section \ref{sec_sld}, we used in the partition of unity argument that $\Gamma\backslash G/M$ is a manifold. 
Our main purpose is now to find a covering compatible with the orbifold structure (Lemma \ref{lem-cover}). The partition of unity constructed from this covering will allow us to complete the proof in this case. 

We start with the definition of orbifolds, following \cite[Chap 13]{thurston_three-dimensional_1997}. 
An orbifold $O$ consists of the underlying Hausdorff space $X_O$ and an orbifold atlas $\{(U_i,\widetilde U_i,\Gamma_i,\varphi_i)\}_{i\in I}$ such that different atlas' should be compatible and where 
\begin{itemize}
    \item each $U_i$ is an open set of $X_O$ and their union covers $X_O$,
    \item the family of sets $\lbrace U_i \rbrace_{i\in I}$ is closed under finite intersection,
    \item each set $\widetilde U_i$ is an open subset of $\R^n$, 
    \item $\Gamma_i$ is a finite group of linear transformations which fixes $\widetilde U_i$,
    \item the map $\varphi_i$ is a homeomorphism from $U_i$ to the quotient $\widetilde U_i/\Gamma_i$.
\end{itemize}

\begin{rem} \label{rem-deven}
If $d$ is even, let $p$ be the projection from $\slr$ to $\pslr=\slr/\{\pm id \}$. Then due to $-id\in M$, as double cosets, the space $\Gamma\backslash \slr/M$ is isomorphic to $p(\Gamma)\backslash \pslr/p(M)$ and periodic compact $A$-orbits are also isomorphic under the projection $p$. It is sufficient to prove equidistribution on $p(\Gamma)\backslash \pslr/p(M)$. We make the convention that if $d$ is even, then we consider $G=\pslr$.
\end{rem}

By \cite[Prop 13.2.1]{thurston_three-dimensional_1997}, since the finite group $M$ acts properly discontinuous on $\Gamma\backslash G$,  the space $O:=\Gamma\backslash G/M$ is an orbifold. Since the metric $d_1$ on $G$ is left $G$ invariant and right $K$ invariant, we use the quotient metric on $\Gamma\backslash G/M$ denoted by $d$. For this orbifold, its singular locus is defined as
\[\Sigma(O)=\{ x\in O|\ M_x\neq\{id\} \}, \]
where the isotopy group $M_x$ is the stabilizer of the group action $M$ at a lift of $x$. Since $M$ is abelian, the isotopy group $M_x$ is independent of the choice of the lift of $x$. This singular locus is the union of closed subvarieties of $O$, which has zero measure (Newman's theorem, see \cite{dress_newmans_1969} for a proof). If we are outside this singular locus, we are in the normal covering situation.
Similarly, we have another orbifold structure on $\Gamma\backslash G/M$ coming from the action of $\Gamma$ on $G/M$. The singular locus of the $\Gamma$ action is the same as the singular locus of $M$ action due to the structure of the double coset. We use the same notation $\Sigma(O)$.

Since the group $M$ is much simpler, we will first find a good covering for the action of $M$ on $\Gamma\backslash G$, then pass to $\Gamma$ on $G/M$.

\begin{lem}\label{lem-singularlocus}
Let $F\in C(A)$ be a compact periodic $A$-orbit in $O$, then $F\cap\Sigma(O)=\emptyset$.
\end{lem}
\begin{proof}
Let $x=\Gamma gM$ be a point in $F$, choose a lift $\Gamma g$ in $\Gamma\backslash G$. 
By definition, we have
\[M_x:=\{m\in M,\ \Gamma gm=\Gamma g \}. \]
For every $m \in M_x$, there exists $\gamma\in \Gamma$ such that $gm=\gamma g$, i.e. $\gamma=gmg^{-1}$. 
Using that $F\in C(A)$, we choose for any period $Y\in \Lambda(F)\cap\frak a^{++}$ an element $\gamma_Y\in \Gamma^{lox}_c$ such that $\gamma_Y=g\exp(Y)m_Yg^{-1}$. 
Now since $M$ is abelian, $\gamma \in G_{\gamma_Y}$ the centralizer of $\gamma_Y$ in $G$ and all its eigenvalues are rational.
By Lemma \ref{lem-rational}, we deduce that $m=\pm e$. 
If $d$ is even, then $-e \in M$ and we use the convention in Remark \ref{rem-deven} to deduce that $M_x$ is trivial for every $x \in F$.
\end{proof}

To construct a covering of balls of radius $\asymp r$, which is compatible with the orbifold structure, we study the right action of the discrete abelian group $M$ on the manifold $\Gamma \backslash G$. 
The argument used to prove Lemma \ref{lem-injradius} gives a lower bound for the injectivity radius of $\Omega(R)$, seen as a compact subset of $\Gamma \backslash G$.
\begin{lem}\label{lem-injradius1}
There exists $C_9>1$ such that for all $R>2$ and every $z\in \Omega(R)\subset \Gamma\backslash G$,
\[ inj(z)\geq R^{-C_9}. \]
Furthermore, there exists a representative $h\in G$ such that $z=\Gamma h$ and
\[\dd_X(o,ho)\leq C_9\log R,\text{ and }  \sup( \|h\|,\|h^{-1}\|) \leq R^{C_9}. \]
\end{lem}

We start with a general lemma about the action of $M\simeq (\Z/2\Z)^{d-1}$ on any manifold.
For any point $x$ in the manifold, we denote by $M_x:= \mathrm{Stab}_M(x)$.
\begin{lem}\label{lem-groupM}
Let $N$ be a complete Riemannian manifold such that $M$ act on $N$ isometrically.
For any point $y \in N$, if there exists $m \in M \setminus M_y$ and $s \in (0, inj(y)/5 )$
such that $B(y,s)m\cap B(y,s)\neq \emptyset$, then there exists $z\in B(y,s)$ such that 
$M_z \supset \langle M_y, m \rangle $.
\end{lem}

\begin{proof}
Consider $y \in N$, an element $m \in M \setminus M_y$ and $s>0$ as in the statement. 
Since $B(y,s)m\cap B(y,s)\neq \emptyset$, we deduce that $d(y,ym)<2s$. 
By choice of $s$, the exponential at $T_{y}N$ identifies $B(0,5s)$ with $B(y,5s)$, hence in this local chart, $B(y,2s) \subset B(y,5s)$ identifies with some ball of radius $2s$ in $B(0,5s)$. 
Now $m$ is an isometry of order $2$ of $N$, hence $(ym, y)m =(y,ym)$.
The action of $m$ preserves geodesics on $N$, hence in the local chart induced by the exponential at $T_y N$, the element $m$ reads as an isometry that preserves the geodesic segment $[y,ym]$ and flips the endpoints. 
Therefore, the middle point $z$ of this geodesic segment is fixed by $m$ i.e. $m\in M_z$.
Furthermore, $z \in B(y,s)$ since $d(y,z)= \frac{1}{2} d(y,ym) < s$.

Similarly, $M_y$ acts isometrically in particular on the geodesic segment $[y,ym]$. 
Since $M$ is abelian, the action of $M_y$ fixes the endpoints. Due to the hypothesis of injectivity radius, it fixes the whole geodesic segment.
In particular $M_z \supset M_y$.
\end{proof}
Since the group structure is simple, we use the action of $M$ on $\Gamma\backslash G$ to obtain a covering that is compatible with this action.
Denote by $C_{10}:=11^{d-1}$, which only depends on $d$.
\begin{lem}\label{lem-orbifoldM}
For all $R>2$ and $s \in (0,  R^{-C_9}/C_{10} )$, there exists a covering of $\Omega(R)\subset O$ consisting of balls $\lbrace B(x_i,\frac{s_i}{10}) \rbrace_{i\in I}$ such that $s_i\in[s,C_{10}s]$ and each larger ball $B_i:= B(x_i,s_i)$ is compatible with the orbifold structure. 
Meaning that for each ball $B_i$, there exists $(\widetilde B_i, M_i)\subset \Gamma\backslash G\times M$ where $M_i$ is a subgroup of $M$, such that $\widetilde B_i$ is $M_i$ invariant and $B_i$ is homeomorphic to $\widetilde B_i/M_i$.
\end{lem}
\begin{proof}
In $\Gamma\backslash G$, the injectivity radius can also be defined as: for $y=\Gamma g\in \Gamma\backslash G$, 
\[ inj(y)=\sup_{r_0>r>0}\{ \text{for any }\gamma\neq e,\  \gamma B(g,r)\cap B(g,r)=\emptyset \},  \]
where $r_0$ is the injectivity radius of the group $G$.
For all $x\in\Omega(R)$, by the lower bound for the injectivity radius at $x$ given in Lemma \ref{lem-injradius} and the above characterisation of injectivity radius, 
\begin{equation}\label{equ-injconvex}
 \text{for every }y \in B (x, C_{10}s/2 ),\ inj(y)\geq C_{10}s/2.
\end{equation}

Fix a point $x\in \Omega(R) \subset \Gamma\backslash G/M$ and $s\in (0, R^{-C_9}/C_{10})$. We define a family of radii $r_j=11^{j-1}\times 10s\in [s,C_{10}s]$ for $1\leq j\leq d-1$ and $r_0=s$.
We construct, by induction, a compatible ball containing $x$.
Denote by $y_0 \in \Gamma \backslash G$ a lift of $x=y_0M$. 
If $B(x,s)M\cap \Sigma(O)=\emptyset$, i.e. $B(y_0,r_0)m \cap B(y_0,r_0) = \emptyset$ for all $m \in M$, we add the ball $B(x,r_0/10)$ to the covering of $O$, with $B(y_0,r_0) \in \Gamma \backslash G$ and $M_0 := M_{y_0}$ is the trivial group of $M$.

Assume we have constructed for some $0 \leq k \leq d-1$ a family of points $y_0,..., y_k \in \Gamma \backslash G$, a strictly increasing family of sign subgroups $M_0 := M_{y_0} \subset ... \subset M_k:= M_{y_k}$, such that
\begin{equation}
     d(y_j,y_0)\leq  r_j/10 \text{, for every }j=0\cdots, k-1.
\end{equation}
Assume $B(y_k, r_k )m \cap B(y_k,r_k/10) = \emptyset$ for all $m \in M \setminus M_k$ and we add the ball $B(y_kM, r_k)$ to the covering of $O$, with $B(y_k, r_k) \in \Gamma \backslash G$ and isotopy group $M_k$.
Otherwise, due to \eqref{equ-injconvex}, we can apply the previous Lemma \ref{lem-groupM} to $y_k$ and $r_k>0$. We find $y_{k+1}\in B(y_k,r_k)$ with $M_{k+1}=M_{y_{k+1}}$ strictly containing $M_k$.
By hypothesis, 
$$d(y_{k+1},y_0)\leq  d(y_{k+1},y_k)+d(y_k,y_0)\leq r_k+r_k/10 = r_{k+1}/10. $$
By this way, we construct $y_{k+1}$ and $M_{k+1}$ satisfying the hypothesis of induction. 

Since $M$ is finite, we must stop at some $k$, which means that  $B(y_k, r_k )m \cap B(y_k,r_k) = \emptyset$ for all $m \in M \setminus M_k$, and we can add this ball to the covering.






We can do this for any $x\in \Omega(R)$, and the proof is complete.
\end{proof}

Now we use the orbifold structure from the action of $\Gamma$ on $G/M$. We use the double coset relation to do this step.
\begin{lem}\label{lem-cover}
For all $R>2$ and $s \in (0,  R^{-C_9}/2C_{10} )$, there exists a covering of $\Omega(R)\subset O$ consisting of balls $\lbrace B(x_i,\frac{s_i}{10}) \rbrace_{i\in I}$ such that $s_i\in[s,C_{10}s]$ and each larger ball $B_i:= B(x_i,s_i)$ is compatible with the orbifold structure. 
That is for any $B=B_i$, there exists $(\widetilde B, \Gamma_B)\subset G/M\times \Gamma$, where $\Gamma_B$ is a finite subgroup of $\Gamma$, such that $\widetilde B$ is $\Gamma_B$ invariant and $B$ homeomorphic to $\Gamma_B\backslash \widetilde B$.
\end{lem}

\begin{proof}
Once we have a ball $B(w,r)=B(\Gamma g,r)$ and $M_w$ by Lemma \ref{lem-orbifoldM}. Since this ball intersects $\Omega(R)$ and $s$ small, we have $\Gamma g\in \Omega(2R)$. By Lemma \ref{lem-injradius1}, we can take this $g$ such that $\|g\|,\|g^{-1}\|\leq (2R)^{C_9}$.
For each $m\in M_w$, there exists $\gamma_m$ such that $\gamma_m g=gm$. Let $\Gamma_g$ be the group generated by $\gamma_m$. Through the conjugate action $g$, the group $\Gamma_g$ is isomorphic to $M_w$. Consider the pair $B(gM,r)$ and $\Gamma_g$. 

If there exists $\gamma\notin \Gamma_g$ such that $\gamma B(gM,r)\cap B(gM,r)\neq\emptyset$, then there exist $g_1,g_2\in B(g,r)$ such that $\gamma g_1M=g_2M\in B(gM,r)$. We must have $\gamma g_1=g_2m$. But this means $B(\Gamma g,r)\cap B(\Gamma g,r)m\neq\emptyset$. So $m$ is in $M_w$ and $B(\Gamma g,r) m=B(\Gamma g,r)$. But now, by similar computation as in Lemma \ref{lem-injradius} 
\[\|\gamma_m-\gamma\|=\|gmg^{-1}-g_2mg_1^{-1}\|=\| (g-g_2)mg^{-1}+g_2m(g^{-1}-g_1^{-1})\|\leq 2r\|g\|\|g^{-1}\|\leq 2C_{10}sR^{2C_9} < 1 \]
is small. By discreteness, we have $\gamma=\gamma_m$, contradicts to $\gamma\notin \Gamma_m$.

Therefore for any $\gamma\notin \Gamma_g$, we have $\gamma B(gM,r)\cap B(gM,r)=\emptyset $ and $\gamma\in \Gamma_g$ preserves $B(gM,r)$, due to $\gamma$ preserving the metric.
\end{proof}

Once we have a family of balls $B(x_i,s_i/10)$ as in Lemma \ref{lem-cover} which covers $\Omega(R)$, by Vitali's covering theorem for metric spaces, we can find a subcollection $B(y_j,s_j/10)$ which are disjoint and the union of larger balls $B(y_j,s_j/2)$ covers the union of $B(x_i,s_i/10)$. In particular, the union of $B(y_j,s_j/2)$ covers $\Omega(R)$ and $B(y_j,s_j)$ is compatible with orbifold structure. 

\begin{lem}\label{lem-unity}
There exists a constant $C>0$. With the same assumption of $s,R$ as in Lemma \ref{lem-cover}
There exists a partition of unity $\{\rho_j\}$ subordinated to the open cover $\{B(y_j,s_j)\}$ of $\Omega(R)$ with $s_j\in [s,C_{10}s]$ such that for any $x,y\in B(y_j,s_j)\cap\Omega(R)$, we have
\[|\rho_j(x)-\rho_j(y)|/d(x,y)\leq\frac{C}{s}.   \]
The number of balls $B(y_j,s_j)$ is less than $Cs^{-\dim G}$.
\end{lem}

The construction of a partition of unity subordinated to a covering is classic. We add the proof for completeness.

\begin{proof}

On each ball $B(y_j,s_j)$ we take the function
\[\widetilde\rho_j(x)=\max\{0,1-3d(x,B(y_j,s_j/2))/s_j \}, \]
which takes value 1 on the ball $B(y_j,s_j/2)$ and vanish outside of $B(y_j,s_j)$ with Lipschitz norm bounded by $3/s$.
Let $\rho_j=\widetilde \rho_j/\sum_j \widetilde\rho_j$. This is a partition of unity with respect the covering $B(y_j,s_j)$. Let us compute their Lipschitz norms, for $x,y$ in $B(y_j,s_j)\cap \Omega(R)$
\begin{align*}
    |\rho_j(x)-\rho_j(y)|/d(x,y)&=|\frac{\widetilde\rho_j(x)-\widetilde\rho_j(y)}{\sum\widetilde\rho_l(x)}-\widetilde\rho_j(y)\frac{\sum_l\widetilde\rho_l(x)-\widetilde\rho_l(y)}{(\sum \widetilde\rho_l(x))(\sum \widetilde\rho_l(y))}|/d(x,y)\\
    &\leq \frac{3}{s}\frac{1}{\sum \widetilde\rho_l(x)}+\frac{3\#\{y_l,\ x\text{ or }y\in B(y_l,s_l) \}}{s}\frac{1}{(\sum \widetilde\rho_l(x))(\sum \widetilde\rho_l(y))}.
\end{align*}
Due to $x,y\in \Omega(R)\subset \cup_l B(y_l,s_j/2)$, we obtain that $\sum_l\widetilde\rho_l(x), \sum_l\widetilde\rho_l(y)\geq 1$. Since different $y_l$'s have distance at least $s/10$ to each other, by homegenity of the space, the number of $y_l$ such $x\in B(y_l,s_l)$ is uniformly bounded. Therefore we obtain the lemma.
\end{proof}

For any $\psi_\Gamma$ Lipschtiz function supported on $\Omega(R)$, let $\psi_j=\psi_\Gamma\rho_j$. Then $\psi_\Gamma=\sum \psi_j$. For each $\psi_j$, by Lemma \ref{lem-unity}, we obtain $|\psi_j|_{Lip}\ll 1/s|\psi_\Gamma|_{Lip}$.
Take its lift $\widetilde\psi_j$ on $\widetilde B_j$. Then since singular locus has zero measure, and outside the singular locus $\Sigma(O)$ it is a regular covering and $B_j\simeq \widetilde B_j/\Gamma_{B_j}$, we obtain
\begin{equation}\label{equ-gm}
    \int \psi_j \dd m_{\Gamma\backslash G/M}=\frac{1}{|\Gamma_B|}\int \widetilde\psi_j\dd m_{G/M}.
\end{equation}
Due to $F\in C(A)$ not intersecting singular locus (Lemma \ref{lem-singularlocus}), $\calM^t$ being $\Gamma$ invariant and $B_j\simeq \widetilde B_j/\Gamma_{B_j}$, we obtain
\begin{equation}\label{equ-mt}
    \int \psi_j \dd \calM_\Gamma^t=\frac{1}{|\Gamma_B|}\int \widetilde\psi_j\dd \calM^t.
\end{equation}
\begin{proof}[Proof of Theorem \ref{thm-introequid} for general $\Gamma<\sld$]
With \eqref{equ-gm}, \eqref{equ-mt} and Lemma \ref{lem-injradius1}, we can redo the argument as in Section \ref{sec:equicpt} to obtain the equidistribution for compactly supported functions on the orbifold $\Gamma\backslash G/M$. Then we use this equidistribution result to do the same argument as in Section \ref{sec-nonescape} to finish the proof.
\end{proof}

We also need a version of Lemma \ref{lem-gloxc} in this case
\begin{lem}\label{lem-gloxcorbifold}
There is a bijection between $\calG_c^{lox}$ and $\calG(A)$. 
\end{lem}
\begin{proof}
The proof is almost the same as the proof of Lemma \ref{lem-gaglox} and \ref{lem-gloxc}. The  main difference is that we may have two different $\gamma_Y,\gamma_Y'$ satisfying
\[\gamma_Yg=ge^Ym_Y,\ \gamma_Y'g=ge^Ym_Y'. \]
Take $\gamma=\gamma_Y^{-1}\gamma_Y'$, then $\gamma$ commutes with $\gamma_Y$ and all its eigenvalues are $\pm 1$. By Lemma \ref{lem-rational}, we know that $\gamma=\pm id$. By convention in Remark \ref{rem-deven}, we obtain that $\gamma$ is trivial.
\end{proof}

\begin{rem}
The same construction of the covering and the partition of unity for orbifolds $\Gamma\backslash G/M$ also works if $G$ is $\R$-split, in which case the group $M$ is isomorphic to $\{\pm 1\}^r$ for some $r\in \N$. But we need also prove Lemma \ref{lem-singularlocus} and Lemma \ref{lem-gloxcorbifold}, which relies on Lemma \ref{lem-rational}. If we have s similar version of Lemma \ref{lem-rational} for a cocompact irreducible lattice $\Gamma$ in a $\R$-split group $G$, then Theorem \ref{thm-introequid} also works without the torsion-free assumption.
\end{rem}

\section{Appendix B}

\subsection{Proof of Theorem \ref{theo-gorodnik-nevo}}\label{appendix}
We give a proof of Theorem \ref{theo-gorodnik-nevo} by redoing the proof of Theorem 7.1 in \cite{gorodnikCountingLatticePoints2012} for Lipschitz functions. Here we have one notation issue, the quotient $\Gamma$ is on the right $G/\Gamma$ to be consistent with \cite{gorodnikCountingLatticePoints2012}.  Fix notation $m_G$ and $dm_{G/\Gamma}=dm_G/V(\Gamma)$, which is a probablity measure.

Recall the quantitative mean ergodic theorem on $L^2(G/\Gamma)$, which is the main engine to obtain equidistribution. For an absolutely continuous probability measure $\beta$ on $G$, let $\pi(\beta)f=\int \pi(g)f d\beta(g)$. By Theorem 4.5 in \cite{gorodnikCountingLatticePoints2012}, we have
\begin{equation}\label{equ.meanergodic}
\bigg\|\pi(\beta)f-\int f\bigg\|_{2}\leq C_q\|\beta\|_{q}^{1/n(G,\Gamma)} \|f\|_{2}, 
\end{equation}
where $n(G,\Gamma)$ is an integer depending on $G$, $\Gamma$ and $q$ is any constant in $[1,2)$ such that $\|\beta\|_q<\infty$.

Let $$\epsilon_\inj>0$$ be a constant such that if $\epsilon<\epsilon_\inj$, then the map $\calO_\epsilon$ to $\calO_\epsilon\Gamma$ is injective from $G$ to $G/\Gamma$.

We will prove this version
\begin{theorem}\label{theo-gorodnik-nevo1}
\hypG
Let $\Gamma <G$ be an irreducible lattice. There exists $C_6>0$ only depending on $n(G,\Gamma)$ and $G$. 
Let $x \in X$ and $(B_t)_{t >0}$ be $D_t^{++}$.
Then for all Lipschitz test functions $\psi \in Lip(\cal{F}^{(2)})$, there exists $E(t,\psi)=O(Lip(\psi) \vol(D_t)^{-\kappa})$ when $t > C_6 |\log\epsilon_\inj|$ such that

$$ \frac{1}{\vol(B_t)} \sum_{\gamma \in B_t\cap \Gamma} \psi(\gamma_o^+,\gamma_o^-) = \frac{1}{\vol(\Gamma \backslash G)}   \int_{ \calF\times\calF  } \psi \dd \mu_o\otimes\mu_o + E(t,\psi) ,$$
where all the implied constants $\ll$ only depending on $G$ and $n(G,\Gamma)$.
\end{theorem}

\begin{proof}[Proof of Theorem \ref{theo-gorodnik-nevo}]
Due to \eqref{equ.gammax} $\gamma_x^+=h_x(h_x^{-1}\gamma h_x)^+_o$, we apply Theorem \ref{theo-gorodnik-nevo1} to the lattice $h_x^{-1}\Gamma h_x$ and the Lipschitz function $\psi'(\cdot,\cdot):=\psi(h_x\cdot,h_x\cdot)$. This is the reason that we need a uniformed version for lattices $h_x^{-1}\Gamma h_x$ and we made dependence of constants in Theorem \ref{theo-gorodnik-nevo1} more transparent. The constant $n(G,h_x^{-1}\Gamma h_x)$ is the same as $n(G,\Gamma)$ due to invariance of the Haar measure.
For $\epsilon_\inj$ of $h_x^{-1}\Gamma h_x$, we have
\[\inf_{\gamma\in\Gamma-\{e\}}d_G(o,h_x^{-1}\gamma h_x)\geq e^{-Cd_X(o,x)}\inf_{\gamma\in\Gamma-\{e\}}d_G(o,\gamma). \]
By Lemma \ref{lem-actiong}, the action of $h_x$ on $\calF$ is $C_x$ Lipschitz. From these, we obtain Theorem \ref{theo-gorodnik-nevo}.
\end{proof}

\textbf{Step 1}: 
The first step is to transfer the counting problem to integrals, which can be treated by the mean ergodic theorem.

Let $\calO_\epsilon$ be a neighborhood of identitity in $G$ with radius $\epsilon$. Let $\widetilde{A}^\delta=\{\exp(a),\ a\in\frak a^{++},\ d(a,\partial \frak a^{++})\geq\delta \}$.
\begin{lem}[Effective Cartan decomposition, Proposition 7.3 in \cite{gorodnikCountingLatticePoints2012}, first appeared in \cite{gorodnikStrongWavefrontLemma2010}]\label{lem_effective_cartan}
There exist $\delta>0$ and $l_0,\epsilon_1>0$. If $\epsilon<\epsilon_1$, then for $g=k_1 ak_2\in K \widetilde{A}^\delta K$, we have
\[\calO_\epsilon g\calO_\epsilon\subset (\calO_{l_0\epsilon}\cap K)k_1M(\calO_{l_0\epsilon}\cap A)ak_2(\calO_{l_0\epsilon}K). \]
\end{lem}
For ease of notation, when there is no confusion, we will use $k_1,a,k_2$ to denote elements come from the Cartan decomposition  $g=k_1ak_2$. 
Let $ \widetilde{D}^\delta_t=\{g\in G, a\in \widetilde{A}^\delta,d(gK,K)\leq t\}$. Notice that by identifying $\calF$ with $K/M$, we have $k_1M=\gamma_o^+$ and $k_2^{-1}M=\gamma_o^-$.
Let $$\rho_t(g)=\mathbbm{1}_{\widetilde{D}^\delta_t}(a)\phi(k_1,k_2),$$
where $\phi(k_1,k_2)=\psi(k_1M,k_2^{-1}M)=\psi(g_o^+,g_o^-)$.

We introduce two auxiliary functions, which is the replacement of Lipschitz well-roundness of sets in \cite{gorodnikCountingLatticePoints2012}. Recall 
\[Lip \;\phi=\max\bigg\{|\phi|_\infty, \sup_{x\neq y}\frac{|\phi(x)-\phi(y)|}{d(x,y)} \bigg\}.\] 
Let
\begin{align*}
    \rho_{t,\epsilon}^+(g)&=\mathbbm{1}_{\widetilde{D}^{\delta-l_0\epsilon}_{t+2\epsilon}}(g)(\phi(k_1,k_2)+(Lip\phi)l_0\epsilon)\\
    \rho_{t,\epsilon}^-(g)&=\mathbbm{1}_{\widetilde{D}^{\delta+l_0\epsilon}_{t-2\epsilon}}(g)\max\{\phi(k_1,k_2)-(Lip\phi)l_0\epsilon,0\}.
\end{align*}
From the definition, we know $\rho^-_{t,\epsilon}\leq \rho_t\leq \rho^+_{t,\epsilon}$. 
\begin{lem}
For $g\in\calO_\epsilon\gamma\calO_\epsilon$ with $\epsilon\leq \epsilon_1$ we obtain
\begin{equation}\label{equ_rho-+}
\rho^-_{t,\epsilon}(g)\leq \rho_t(\gamma)\leq \rho^+_{t,\epsilon}(g).
\end{equation}
\end{lem}
\begin{proof}
If $\rho^-_{t,\epsilon}(g)\neq 0$, then $g\in \widetilde D^{\delta+l_0\epsilon}_{t-2\epsilon}$. By $\gamma\in \calO_\epsilon g\calO_\epsilon$ and Lemma \ref{lem_effective_cartan}, we obtain 
\[a(\gamma)\in (\calO_{l_0\epsilon}\cap A)a(g)\cap D_t\subset \widetilde D^\delta_t. \]
So $\mathbbm{1}_{\widetilde D^\delta_t}(a(\gamma))=1$. By Lemma \ref{lem_effective_cartan} and Lipschitz property of $\phi$, we obtain  
\[\phi(k_1(\gamma),k_2(\gamma))\geq \phi(k_1(g),k_2(g))-(Lip\phi)l_0\epsilon. \]
This proves the left hand side. For the other side, the proof is similar.
\end{proof}
Take $\mathbbm{1}_\epsilon=\frac{1}{m_G(\calO_\epsilon)}\mathbbm{1}_{\calO_\epsilon}$ be the normalized characteristic function of $\calO_\epsilon$.
Let $\varphi_\epsilon(g\Gamma)=\sum_{\gamma\in\Gamma}\mathbbm{1}_\epsilon(g\gamma)$. The counting is connected to integral by the following.
\begin{lem}\label{lem_rhot}
For $h$ in $\calO_\epsilon$ with $\epsilon\leq \epsilon_1$, we have
\begin{equation}
    \int \varphi_\epsilon(g^{-1}h\Gamma)\rho^-_{t,\epsilon}(g)dm_G(g)\leq \sum_{\gamma\in \Gamma}\rho_t(\gamma)\leq \int \varphi_\epsilon(g^{-1}h\Gamma)\rho^+_{t,\epsilon}(g)dm_G(g).
\end{equation}
\end{lem}
\begin{proof}
By using \eqref{equ_rho-+}, the proof is almost the same as Lemma 2.1 in \cite{gorodnikCountingLatticePoints2012}.
\end{proof}
\textbf{Step 2:}This step will estimate the error terms in the mean ergodic theorem.

We want to apply the mean ergodic theorem to probability measures $\frac{\rho_{t,\epsilon}^{\pm}}{\int\rho_{t,\epsilon}^{\pm}}$. Before doing so, we need to compute some integrals. The computation is a bit tedious. \textbf{This step is to verify similar stable mean ergodic theorems, the main consequence is \eqref{equ_rho-} and \eqref{equ_rho+}}.

Let's first compute the difference.

\begin{lem}We have
\begin{equation}\label{equ_diff}
\int \rho^+_{t,\epsilon} dm_G-\int \rho^-_{t,\epsilon} dm_G\ll \bigg(\epsilon \int\phi+l_0\epsilon(Lip\phi) \bigg)m_G(D_t).
\end{equation}
\end{lem}

\begin{proof}
\begin{equation*}
\begin{split}
    &\int \rho^+_{t,\epsilon} dm_G-\int \rho^-_{t,\epsilon} dm_G\\
    &\leq m_G\big(\widetilde{D}^{\delta-l_0\epsilon}_{t+2\epsilon}\big)\bigg(\int\phi+l_0\epsilon(Lip\phi)\bigg)-m_G\big(\widetilde{D}^{\delta+l_0\epsilon}_{t-2\epsilon}\big)\bigg(\int\phi-l_0\epsilon(Lip\phi)\bigg)\\
    &=\bigg( m_G\big(\widetilde{D}^{\delta-l_0\epsilon}_{t+2\epsilon}\big)-m_G\big(\widetilde{D}^{\delta+l_0\epsilon}_{t-2\epsilon}\big)\bigg) \int\phi+l_0\epsilon(Lip\phi)\Big(m_G\big(\widetilde{D}^{\delta-l_0\epsilon}_{t+2\epsilon}\big)+m_G\big(\widetilde{D}^{\delta+l_0\epsilon}_{t-2\epsilon}\big)\Big)\\
    &\ll \bigg(\epsilon \int\phi+l_0\epsilon(Lip\phi)\bigg)m_G(D_t),
\end{split}
\end{equation*}
where the last inequality is from the proof of Proposition 7.4 in \cite{gorodnikCountingLatticePoints2012} about volume estimate and Lemma \ref{lem-vollip}. 
\end{proof}

Let $\widetilde\rho_{t,\epsilon}^{\pm}=\rho_{t,\epsilon}^{\pm}/\int \rho_{t,\epsilon}^{\pm}$.

\begin{lem} There exists $t_1>0$ which only depends on $G,\delta,\epsilon_1$ such that the following holds.
For $\epsilon<\min\{\int\phi/2l_0Lip\phi, \epsilon_1\}$, $t>t_1$ and $f\in L^2(G/\Gamma)$
\begin{equation}\label{equ_rho-}
    \|\pi(\widetilde\rho^{-}_{t,\epsilon})f-\int f\|_{2}\leq E(t) \|f\|_{2},
\end{equation}
with 
\begin{equation}\label{equ_et}
E(t)=(\frac{C}{m_G(D_t)^{q-1}}\frac{(Lip \phi)^q}{(\int\phi)^q})^{\kappa},
\end{equation}
 $\kappa=1/qn(G,\Gamma)$ and $C>0$ only depending on $G$.

For $\epsilon\leq \epsilon_1$, $t>t_1$ and $f\in L^2(G/\Gamma)$
\begin{equation}\label{equ_rho+}
    \|\pi(\widetilde\rho^{+}_{t,\epsilon})f-\int f\|_{2}\leq  E(t) \|f\|_{2}.
\end{equation}
\end{lem}

The main difference is that for $\rho^+_{t,\epsilon}$, we don't need an extra condition of $\epsilon$ depending on $\phi$.

\begin{proof}
We compute the integral of $\rho_{t,\epsilon}^-$. We have
\begin{align*}
    \int \rho_{t,\epsilon}^-dm_G\geq m_G(\widetilde{D}^{\delta+l_0\epsilon}_{t-2\epsilon})(\int\phi-(Lip\phi)l_0\epsilon).
\end{align*}
Due to Lemma \ref{volume_reste}, we have $m_G(D^{\delta+l_0\epsilon}_{t-2\epsilon})=O(m_G(D_{t-2\epsilon})^{1-\zeta_1})$ for some $\zeta_1>0$, and Lemma \ref{lem-vollip}, we obtain
\[  m_G(\widetilde{D}^{\delta+l_0\epsilon}_{t-2\epsilon})\geq e^{-C\epsilon}(1-Cm_G(D_t)^{-\zeta_1})m_G(D_t).\]
Hence if $t>t_1$ depending on $C,\delta,\epsilon_1$, then
\[ \int \rho_{t,\epsilon}^-dm_G\gg m_G(D_t)(\int\phi-(Lip\phi)l_0\epsilon). \]
Therefore if $\epsilon\leq \int\phi/2l_0 (Lip\phi),\epsilon_1$, we obtain
\begin{equation}\label{equ_rhot-}
    \int \rho_{t,\epsilon}^-dm_G\gg m_G(D_t)\int\phi.
\end{equation}

Similarly for the integral of $\rho_{t,\epsilon}^+$, we obtain if $t>t_1$ and $\epsilon\leq \epsilon_1$, then by Lemma \ref{volume_reste}
\begin{equation}\label{equ_dt+}
m_G(\widetilde D^{\delta-l_0\epsilon}_{t+2\epsilon})\geq m_G(\widetilde D^\delta_t)\geq m_G(D_t)-C m_G(D_t)^{1-\zeta_1}.
\end{equation}
Therefore
\begin{equation}\label{equ_rhot+}
    \int \rho_{t,\epsilon}^+dm_G\geq \int \rho_tdm_G\geq (m_G(D_t)-C m_G(D_t)^{1-\zeta_1})\int\phi\gg m_G(D_t)\int\phi.
\end{equation}

After these preparation, we can start to compute the integral appears in error term of mean ergodic theorem. By \eqref{equ_rhot-}, we obtain when $t>t_1$ and $\epsilon\leq\int\phi/2l_0(Lip\phi)$
\begin{align*}
    \|\rho^{-}_{t,\epsilon}\|_q^q/(\int\rho_{t,\epsilon}^-)^q \ll \int |\rho_{t}|^q/(m_G(D_t)\int\phi)^q\leq  \frac{1}{m_G(D_t)^{q-1}}(Lip\phi)^q/(\int\phi)^q.
\end{align*}
For $\rho_{t,\epsilon}^+$, by \eqref{equ_rhot+} and \eqref{equ_dt+} we have
\begin{align*}
    \|\rho^{+}_{t,\epsilon}\|_q^q/(\int\rho_{t,\epsilon}^+)^q \ll \frac{1}{m_G(D_t)^{q-1}} \int(\phi+(Lip\phi)l_0\epsilon)^q/(\int\phi)^q,
\end{align*}
We obtain if $t>t_1$,
\begin{equation}
     \|\rho^{+}_{t,\epsilon}\|_q^q/(\int\rho_{t,\epsilon}^+)^q\ll \frac{1}{m_G(D_t)^{q-1}}(Lip\phi)^q/(\int\phi)^q.
\end{equation}

Applying the above formulas for $\widetilde\rho^{\pm}_{t,\epsilon}$, combined with mean ergodic estimate \eqref{equ.meanergodic}, we obtain the lemma.
\end{proof}

\textbf{Step 3:} 
The mean ergodic theorem only gives an estimate of $L^2$ norm, but what we need is an estimate at some points. So we need to use the Chebyshev inequality. The remaining work is to collect the error terms. This part is similar to the proof of Theorem 1.9 in \cite{gorodnikCountingLatticePoints2012}.
\begin{proof}[Proof of Theorem \ref{theo-gorodnik-nevo}]
Applying \eqref{equ_rho+} to $f=\varphi_\epsilon$, by Chebyshev's inequality, we obtain for any $\eta>0$
\begin{equation}
    m_{G/\Gamma}\{h\,| \pi(\widetilde\rho_{t,\epsilon}^{+})(\varphi_\epsilon)(h\Gamma)-\int\varphi_\epsilon|>\eta \}\leq (\frac{E(t)\|\varphi_\epsilon\|_{L^2}}{\eta})^2.
\end{equation}
If $(E(t)\|\varphi_\epsilon\|_{L^2}/\eta)^2<m_{G/\Gamma}(\calO_\epsilon)/2=m_G(\calO_\epsilon)/2V(\Gamma)$,
(here we need $\|\varphi_\epsilon\|^2_{L^2(G/\Gamma)}=m_G(\calO_\epsilon)/V(\Gamma)$.)
we will take $\eta=\frac{2E(t)}{m_G(\calO_\epsilon)}$, there exists $h\in\calO_\epsilon$, such that \[\pi(\widetilde\rho_{t,\epsilon}^{+})(\varphi_\epsilon)(h\Gamma)<\eta+\int\varphi_\epsilon. \]
Then by Lemma \ref{lem_rhot},
\begin{align*}
    \sum_{\gamma\cap \widetilde D_t^\delta}\phi(k_1(\gamma),k_2(\gamma))&\leq\pi(\widetilde\rho_{t,\epsilon}^+)(\varphi_\epsilon)(h\Gamma)\int\rho_{t,\epsilon}^+\leq(\eta+\frac{1}{V(\Gamma)})\int\rho_{t,\epsilon}^+\\
    &=\frac{\int\rho_t}{V(\Gamma)}(1+\eta V(\Gamma))+O(\epsilon(Lip\phi)m_G(D_t)),
\end{align*}
where the last inequality is due to \eqref{equ_diff}.
Therefore
\[ \frac{\sum_{\gamma\cap \widetilde D^\delta_t}\phi(k_1(\gamma),k_2(\gamma))}{\int\rho_t}-\frac{1}{V(\Gamma)} \leq\frac{E(t)}{2m_G(\calO_\epsilon)}+\epsilon\frac{Lip\phi}{\int\phi}\frac{1}{V(\Gamma)}\ll \frac{E(t)}{\epsilon^\rho}+\epsilon\frac{Lip\phi}{\int\phi},\]
where $\rho$ is the dimension of group $G$. We also have
\[|\int \rho_t-m_G(D_t)\int\phi|=\int\phi|m_G( \widetilde{D}_t^\delta)-m_G(D_t)|\leq \delta m_G(D_t)^{1-\zeta_1}\int\phi,\] 
and the trivial bound
\[ |\Gamma\cap D_t^\delta|\leq \frac{m_G(\calO_{\epsilon_\inj}D_t^\delta)}{m_G(\calO_{\epsilon_\inj})}\ll m_G(D_{t+\epsilon_\inj}^{\delta+l_0\epsilon_\inj})\epsilon_\inj^{-dim G}\ll m_G(D_t)^{1-\zeta_1}\epsilon_\inj^{-dim G}, \]
therefore
\begin{equation}\label{equ_phileq}
\begin{split}
    &\sum_{\gamma\cap D_t}\phi(k_1(\gamma),k_2(\gamma))-\frac{m_G(D_t)}{V(\Gamma)}\int\phi\\
    &\ll m_G(D_t)(Lip\phi m_G(D_t)^{-\zeta_1}\epsilon_\inj^{-dim G}
    +\int\phi\left(m_G(D_t)^{-\zeta_1}+ \frac{E(t)}{\epsilon^\rho}+\epsilon\frac{Lip\phi}{\int\phi}\right)).
\end{split}
\end{equation}
We can take $C_6$ large enough such that the term $m_G(D_t)^{-\zeta_1}\epsilon_\inj^{-dim G}$ is exponentially small on $t$.

In order to optimize the error term, we take $$\epsilon=(E(t)\int\phi/Lip\phi)^{1/(1+\rho)},$$
then the error term in the above formula is
\[E(t)^{1/(1+\rho)}(\frac{Lip\phi}{\int\phi})^{\rho/(1+\rho)}\ll m_G(D_t)^{-\zeta}(\frac{Lip\phi}{\int\phi})^{(\rho+q\kappa)/(1+\rho)}\leq m_G(D_t)^{-\zeta}(\frac{Lip\phi}{\int\phi}),\]
where the last equality is due to \eqref{equ_et} and $q\kappa=1/n(G,\Gamma)\leq 1$, and where $\zeta=(q-1)\kappa/(1+\rho)$. Here $\epsilon$ should be less than $\epsilon_1,\epsilon_\inj$, but 
\begin{equation}\label{equ_eps}
    \epsilon\leq \left(\frac{C}{m_G(D_t)^{(q-1)\kappa}}\frac{\int\phi}{Lip\phi} \right)^{1/(1+\rho)}\leq \left(\frac{C}{m_G(D_t)^{(q-1)\kappa}} \right)^{1/(1+\rho)}.
\end{equation}
The condition on $\epsilon$ is satisfied if $t$ is greater than some $t_2>0$ and $C_6|\log\epsilon_\inj|$.
Therefore by \eqref{equ_phileq}, we obtain one part of Theorem \ref{theo-gorodnik-nevo} for $t>t_0=\max\{t_1,t_2\}$, with $t_0$ not depending on $\phi$.

For $\rho_{t,\epsilon}^-$, we can obtain the same bound with extra condition that $\epsilon<\int\phi/2l_0Lip\phi$, that is if $t$ is large. Otherwise, we have $\epsilon\geq \int\phi/2l_0Lip\phi$, by \eqref{equ_eps}, which implies
\begin{equation}
Lip\phi\gg m_G(D_t)^{\zeta_2}\int\phi,    
\end{equation}
with $\zeta_2=(q-1)\kappa/\rho$. Therefore by non-negativeness of $\phi$
\[\frac{m_G(D_t)}{V(\Gamma)}\left(\int\phi-C m_G(D_t)^{-\zeta_2}Lip\phi\right)\leq 0\leq \sum_{\gamma\cap D_t}\phi(k_1(\gamma),k_2(\gamma)). \]
By taking $\min\{\zeta,\zeta_2 \}$, the proof is complete.
\end{proof}

\begin{rem}
We need to check the condition in Theorem 4.5 in \cite{gorodnikCountingLatticePoints2012}. For real linear algebraic semisimple Lie groups, we don't need that the group is simply connected. This condition is only needed for $p$-adic case if we look into the proof of Theorem 4.5. Then the crucial condition is that the representation of $G$ on $L^2_0(G/\Gamma)$ is $L^{p+}$. In \cite{ohUniformPointwiseBounds2002}, an explicit estimate on $p$ is given for some cases. If $G$ is a connected real linear algebraic semisimple Lie group and $\Gamma$ is an irreducible lattice, then the condition should be true.

In Kelmer-Sarnak \cite{kelmerStrongSpectralGaps2009}, they explained this for $G=G_1\times\cdots G_r$ with each $G_j$ simple Lie groups and centre free. There are two steps for proving this. If $r=1$ and the real rank of $G$ is 1, then the spectral gap is true. Otherwise, due to Margulis superrigidity theorem, $\Gamma$ is commensurate to a congruence lattice, for which we strong spectral gap (deep result from number theory), that is, each simple factor $G_j$ has a spectral gap on $L^2_0(G/\Gamma)$. Once we have the spectral gap, we can compute the matrix coefficients for each irreducible subrepresentation, which will be in $L^p$ for some bounded $p$, due to the works of many people. From congruence lattice to commersurable lattice, need the Lemma 3.1 of Kleinbock-Margulis \cite{kleinbock_margulis}. We still need to prove $L^2_0(G/\Gamma)$ is $L^{p+}$ from the property of its subrepresentations. 

In \cite{kleinbock_margulis}, Theorem 3.4, they use a strong spectral gap to obtain an estimate of matrix coefficients for smooth vectors, which implies that $L^2_0(G/\Gamma)$ is $L^p$ for some $p$. The idea is that if we know an irreducible representation is $L^p$, then Howe's work tells us the matrix coefficients decay exponentially, and the constants only depend on the group and $p$ for $K$-finite vectors. Then since we know a strong spectral gap, we can obtain each irreducible subrepresentation of $L^2_0(G/\Gamma)$ is strong $L^p$ for some finite $p$, the Howe's work gives a uniform estimate of matrix coefficients of all the irreducible representations. We can obtain a matrix coefficients estimate of $L^2_0(G/\Gamma)$ from its irreducible subrepresentations.

For specialists, they know well. But for us, the step from subrepresentations in $L^p$ to $L^2_0(G/\Gamma)$ in $L^p$ is highly nontrivial.

Now the condition in Corollary 3.5 in \cite{kleinbock_margulis} is that $G$ is a connected algebraic semisimple Lie group centre free without compact factor and $\Gamma$ is an irreducible lattice. If we can remove centre-free, then we are satisfied.

If we have another group $G_1$ with non trivial center. Then we consider $G:=G_1/Z_1=\pi(G_1)$, which is centre-free, where $Z_1$ the center of $G_1$. Let $\Gamma_1$ be an irreducible lattice of $G_1$. Let $\Gamma_2=\Gamma_1 Z_1$ and $\Gamma=\pi(\Gamma_2)$. Then $G_1/\Gamma_2\simeq G/\Gamma$ and $L^2_0(G_1/\Gamma_2)\simeq L^2_0(G/\Gamma)$. Now for each simple factor of $G$, by Theorem 1.12 in \cite{kleinbock_margulis}, it has no almost invariant vector on $L^2_0(G/\Gamma)$. So for simple factors of $G_1$, it will also have no almost invariant vector on $L^2_0(G_1/\Gamma_2)$. Since $\Gamma_1$ is a finite index subgroup of $\Gamma_2$, by Lemma 3.1 in \cite{kleinbock_margulis}, for each simple factor of $G_1$, it has also no almost invariant vector in $L^2_0(G_1/\Gamma_1)$. Therefore, we can use Theorem 3.4 in \cite{kleinbock_margulis} to deduce the desired version. 


\end{rem}

\begin{rem}
Theorem \ref{theo-gorodnik-nevo} is exactly Theorem 7.2 in \cite{gorodnikCountingLatticePoints2012} with an explicit error term, where no proof of Theorem 7.2 is given. But we cannot obtain this Theorem directly from Theorem 7.1 for Lipschitz well-rounded sets in \cite{gorodnikCountingLatticePoints2012} by approximating Lipschitz functions by level sets because the level sets of a Lipschitz function may not be uniformly Lipschitz well rounded. For one-dimensional cases, (i.e. $SL_2(R)$, Lipschitz function on $SO(2)$), we can take a Lipschitz function $\phi$ as the distance to a Cantor set. Then the level sets $\{\phi<1/n\}$ approximate the Cantor set. Each set is Lipschitz well-rounded, but the constant in Lipschitz well-rounded blow up as $n$ tends to infinity because the number of intervals in $\{\phi<1/n \}$ goes to infinite.
\end{rem}

\subsection{Integer points on subvarieties}\label{sec-subvar}

For the proof of Proposition \ref{prop_subvar}, we need Corollary 1.11 from \cite{gorodnikCountingLatticePoints2012}.

\begin{lem}\label{lem:congruence}
Let $\Gamma(p)=\{\gamma\in\Gamma,\ \gamma\equiv Id\ (\rm{mod} p)  \}$ for prime $p$. There exists $\epsilon>0$ such that for all primes $p$, $\gamma\in\Gamma$ and $t>1$, we have
\[|\{\gamma\Gamma(p)\cap D_t \}|=\frac{\vol(D_t)}{[\Gamma:\Gamma(p)]\vol(\Gamma\backslash G)}+O(\vol(D_t)e^{-\epsilon t}).  \]
\end{lem}

\begin{proof}
Recall Lipschitz well-roundness in \cite[Definiton 1.1]{gorodnikCountingLatticePoints2012}: there exist $C>0,\epsilon_1>0$ and $t_1>0$ such that for $0<\epsilon<\epsilon_1$ and $t>t_1$, we have
\[\vol(\calO_\epsilon D_t\calO_\epsilon)\leq (1+C\epsilon)\vol(\cap_{u,v\in\calO_\epsilon}uD_tv), \]
where $\calO_\epsilon$ is the ball $B(e,\epsilon)$ in $G$. This is true for $D_t$. Because of Lemma \ref{lem-vollip} we have
\begin{align*}
    \vol(\calO_\epsilon D_t\calO_\epsilon)\leq \vol(D_{t+\ell\epsilon})\leq (1+C\epsilon)\vol(D_{t-\ell\epsilon})\leq (1+C\epsilon)\vol(\cap_{u,v\in\calO_\epsilon}uD_tv).
\end{align*}

We can use Theorem 4.5 \cite{gorodnikCountingLatticePoints2012} and Lipschitz well-roundness of $D_t$ to verify conditions in Corollary 1.11 \cite{gorodnikCountingLatticePoints2012}. Then Corollary 1.11 \cite{gorodnikCountingLatticePoints2012} implies the result.
\end{proof}

\begin{proof}[Proof of Proposition \ref{prop_subvar}]
If we replace $D_t$ by the ball $B_t=\{\gamma\in\sld,\ |tr(\gamma^t\gamma)|<e^t \}$, this proposition is Theorem 1.8 in \cite{gorodnikLiftingRestrictingSifting2012}. Since there is no detailed proof of Theorem 1.8 in \cite{gorodnikLiftingRestrictingSifting2012}, we give a proof for $D_t$ for completeness. The idea of proof is similar, the main difference is in the estimate of the number of fibres.

Let $p$ be a prime number to be chosen later depending on $t$. Let $\pi_p$ be the map from $\sld$ to $\rm{SL}_d(\Z/p\Z)$. Then $\{ \gamma\in\sld\cap D_t,\ h(\gamma)=0\}$ is a subset in the preimage $\pi_p^{-1}\{\gamma\in \rm{SL}_d(\Z/p\Z),\ h(\gamma)=0 \}$. This can be seen as a fibre space, each fibre is given by $\gamma \Gamma(p)$ for some $\gamma\in\Gamma$ with $h(\pi_p(\gamma))=0$. We only need to estimate the number of fibres and the size of each fibre.

For the size of each fibre, Lemma \ref{lem:congruence} gives us an asymptotic.

For the number of fibre, we have
\begin{equation}\label{equ-sldp}
|\{\gamma\in \rm{SL}_d(\Z/p\Z),\ h(\gamma)=0 \}|\ll p^{dim-1}, \end{equation}
if $h$ does not vanish on $\rm{SL}_d(\Z/p\Z)$, which is true if $p$ is greater than the coefficients of $h$. Here $dim$ is the dimension of $\rm{SL}_d$.
This bound can be obtained from \cite[Lemma 1]{lang_number_1954} or Lemma 1 \cite{tao}. The constant in the upper bound only depends on the degree of $h$ and $d$.

Therefore, by Lemma \ref{lem:congruence} and \eqref{equ-sldp}, we obtain

\begin{align*}
|\{ \gamma\in\sld\cap D_t,\ h(\gamma)=0\}|&\ll p^{dim-1}\left(\frac{\vol(D_t)}{[\Gamma:\Gamma(p)]\vol(\Gamma\backslash G)}+O(\vol(D_t)e^{-\epsilon t})\right)\\
&\leq \vol(D_t)\left( \frac{1}{p\ \vol(\Gamma\backslash G)}+O(p^{dim-1}e^{-\epsilon t})\right).
\end{align*}
 By the Bertrand–Chebyshev Theorem, there is always a prime $p$ in the interval $(n,2n)$ with $n>1$. So we can take a prime $p$ of size $e^{\epsilon t/dim}$. The proof is complete.
\end{proof}

\bibliographystyle{alpha}
\bibliography{references}

\bigskip
 \noindent 
	\it{Institut f\"ur Mathematik, Universit\"at Z\"urich, 8057 Z\"urich}  \\
	email: {\tt jialun.li@math.uzh.ch} 
		
		\bigskip   
		
\noindent 
		\it{Fakult\"at f\"ur Mathematik und Informatik
Universit\"at Heidelberg,
69120 Heidelberg}  \\
			email: {\tt ndang@mathi.uni-heidelberg.de}

\end{document}